\documentclass{amsart}
\usepackage{amssymb}
\usepackage[mathcal]{euscript}
\usepackage{relsize}

\usepackage{bbold} 

\let\oldtocsection=\tocsection 
\let\oldtocsubsection=\tocsubsection 
\renewcommand{\tocsection}[2]{\hspace{0mm}\oldtocsection{#1}{#2}}
\renewcommand{\tocsubsection}[2]{\hspace{2em}\oldtocsubsection{#1}{#2}}

\usepackage{hyperref}
\hypersetup{
    colorlinks,
    citecolor=black,
    filecolor=black,
    linkcolor=black,
    urlcolor=black
}

\usepackage{tikz-cd}
\tikzcdset{arrow style=math font}
\usetikzlibrary{external}

\numberwithin{equation}{section}
\newtheorem{theorem}{Theorem}[section]
\newtheorem*{theorem*}{Theorem}
\newtheorem{corollary}[theorem]{Corollary}
\newtheorem*{corollary*}{Corollary}
\newtheorem{lemma}[theorem]{Lemma}
\newtheorem{proposition}[theorem]{Proposition}
\theoremstyle{definition}

\newtheorem{definition}[theorem]{Definition}

\newtheorem{example}[theorem]{Example}

\def\CC{\mathcal{C}}
\def\DD{\mathcal{D}}
\def\FF{\mathcal{F}}
\def\GG{\mathcal{G}}
\def\EE{\mathcal{E}}

\def\II{\mathcal{I}}
\def\SS{\mathcal{S}}
\def\RR{\mathcal{R}}
\def\KK{\mathcal{K}}
\def\QQ{\mathbb{Q}}
\def\ZZ{\mathbb{Z}}
\def\Fun{\mathsf{Fun}}
\def\Id{\mathsf{Id}}
\def\id{\mathsf{id}}
\def\Coaug{\mathsf{Coaug}}

\def\Hom{\mathsf{Hom}}
\def\sSet{\mathsf{sSet}}
\def\EE{\mathcal{E}}
\def\NN{\mathbb{N}}

\def\Map{\mathsf{Map}}
\def\Mon{\mathsf{Mon}}
\def\mon{\mathsf{mon}}
\def\1{\mathbb{1}}
\def\Set{\mathsf{Set}}
\def\sMon{\mathsf{sMon}}

\def\NN{\mathbb{N}}

\def\map{\mathsf{map}}
\def\holim{\mathsf{holim}}
\def\Tot{\mathsf{Tot}}
\def\AA{\mathcal{A}}
\def\End{\mathsf{End}}
\def\Coaug{\mathsf{Coaug}}
\def\sm{\wedge}
\def\term{\mathsf{term}}
\def\MM{\mathcal{M}}

\def\max{\mathrm{max}}
\def\op{\mathrm{op}}
\def\Tw{\mathsf{Tw}}
\def\Cat{\mathsf{Cat}}
\def\diag{\mathsf{diag}} 
\def\Kan{\mathsf{Kan}}
\def\Spc{\mathsf{Spc}}

\def\AAA{\mathbf{A}}
\def\CCC{\mathbf{C}}
\def\EEE{\mathbf{E}}

\def\Ner{\mathrm{N}}
\def\Nat{\mathsf{Nat}}
\def\hTot{\mathsf{hTot}}

\begin{document}

\title{Bousfield-Kan completion  as a codensity  $\mathlarger{\mathlarger{\mathlarger{\infty}}}$-monad}
\author{Emmanuel Dror Farjoun}
\email{edfarjoun@gmail.com}
\address{
Max Planck Institute for Mathematics, Bonn, Germany.}

\author{Sergei O. Ivanov} 
\email{ivanov.s.o.1986@gmail.com, ivanov.s.o.1986@bimsa.cn}
\address{
Beijing Key Laboratory of Topological Statistics and Applications for Complex Systems, Beijing Institute of Mathematical Sciences and Applications (BIMSA), Beijing, China.}

\begin{abstract}
Working in the setting of $\infty$-categories, we develop a general theory of the codensity monad  $T_\mathcal{D}$ associated with a full subcategory $\mathcal{D}\subseteq \mathcal{C}$.  
We show that $T_\mathcal{D}$  has a canonical monad structure (unique up to a contractible space of choices),  and characterize it as a terminal monad preserving all objects of $\mathcal{D}$. For a monad $\mathcal{M}$ on an $\infty$-category $\CC$, we consider the $\mathcal{M}$-completion functor defined as the totalization of the cosimplicial resolution associated with $\mathcal{M}$. We show that the $\mathcal{M}$-completion functor is the codensity monad associated with the full subcategory of $\mathcal{C}$ spanned by objects that admit a structure of $\mathcal{M}$-algebra.  
In particular, the $\mathcal{M}$-completion functor is the terminal monad preserving all objects that admit a structure of an $\mathcal{M}$-algebra.  
This gives a full $\infty$-categorical characterization of the classical Bousfield-Kan $R$-completion functor as the terminal monad on the category of spaces preserving the empty space and all products of Eilenberg-MacLane spaces $K(A,n)$, where $A$ is an $R$-module.
\end{abstract}

\maketitle

\section{\bf Introduction}

\subsection{Main results}
 
A codensity $\infty$-monad $T_\DD \colon \CC \to \CC$ associated with a full $\infty$-subcategory $\DD$ of an $\infty$-category $\CC$ is defined as the right Kan extension of the inclusion $\DD \hookrightarrow \CC$ along itself. Initially, $T_\DD$ is just a functor, but we prove that it admits a canonical structure of an $\infty$-monad, unique up to a contractible space of choices.
If the inclusion functor $\DD \hookrightarrow \CC$ is a right adjoint,  $\DD$ is called reflective, and the composition $L_\DD \colon \CC \to \DD \hookrightarrow  \CC$ of the left adjoint with the inclusion is called the localization functor associated with $\DD$. In this case, the codensity $\infty$-monad coincides with the localization functor: $T_\DD \simeq L_\DD$. Therefore, the notion of a codensity $\infty$-monad associated with a full $\infty$-subcategory is a natural generalization of the notion of a localization functor. 

An important source of codensity $\infty$-monads is arbitrary $\infty$-monads. Let  $\CC$ be an $\infty$-category  that admits  totalizations (i.e. limits of cosimplicial objects). An $\infty$-monad $\MM$ on $\CC$ defines a canonical cosimplicial object $\MM^\bullet$ in the $\infty$-category of coaugmented endofunctors $\Fun(\CC,\CC)_{\Id/}$. Its totalization is an endofunctor $\widehat{\MM}=\Tot\: \MM^\bullet$ that we call the $\MM$-completion. For example, if $\CC$ is an ordinary category and $M=(M,\mu,\eta)$ is an ordinary monad, then the $M$-completion is the equalizer of two natural transformations $M\eta, \eta M : M \to M^2$ studied in  \cite{fakir1970monade}. 

\vbox{\begin{theorem*}[Cor. \ref{cor:monad_completion}]
Let $\CC$ be an $\infty$-category that admits totalizations and $\MM$ be an $\infty$-monad on $\CC$. Then the $\MM$-completion $\widehat{\MM}=\Tot\ \MM^\bullet$ is equivalent to the codensity $\infty$-monad $T_{\AA(\MM)}$ associated with the full $\infty$-subcategory $\AA(\MM)\subseteq \CC$ spanned by objects that admit a structure of $\MM$-algebra.  
\end{theorem*}}

The main technical ingredient of the proof of this theorem is the fact that, for an object $c$ of $\CC$, the canonical cosimplicial object in the slice category 
\begin{equation}
\Ner(\Delta)\to \AA(\MM)_{c/}
\end{equation}
induced by $\MM^\bullet$ 
is a left cofinal functor
(Section \ref{section:a_lemma_of_BK}). This statement allows one to replace a limit over the  relatively complicated slice category $\AA(\MM)_{c/}$ by the totalization of a cosimplicial object.

The codensity $\infty$-monad $T_\DD$ associated with a full $\infty$-subcategory $\DD\subseteq \CC$ is sometimes referred to as the terminal $\infty$-monad, for the following reason.  We say that a coaugmented endofunctor $F:\CC\to \CC$ (an endofunctor equipped with a natural transformation $\Id_\CC \to F$) is $\DD$-preserving if the canonical map $d \to F(d)$ is an equivalence for all objects $d$ of  $\DD$. Then $T_\DD$ satisfies the following two universal properties (for the case of ordinary categories see  \cite[Prop. 5.4]{leinster2013codensity}).

\begin{theorem*}[Th. \ref{th:D-preserving-coaugmented}]
Let $\DD$ be a full $\infty$-subcategory of an $\infty$-category $\CC$ and $T_\DD$ be the associated codensity $\infty$-monad. Then  
\begin{itemize}
    \item[(a)] $T_\DD$ is the terminal $\DD$-preserving coaugmented functor; 
    \item[(b)] $T_\DD$ is the terminal $\DD$-preserving $\infty$-monad.
\end{itemize}
\end{theorem*}

\noindent In particular, for an $\infty$-monad $\MM$, the $\MM$-completion is the terminal $\AA(\MM)$-preserving coaugmented functor, and the terminal $\AA(\MM)$-preserving $\infty$-monad. 

Many completion-like endofunctors can be presented as codensity monads of full subcategories in ordinary categories. For example, the profinite (resp. pronilpotent) completion functor on the category of groups is the codensity monad of the subcategory of finite (resp. nilpotent) groups. The double dual functor on the category of vector spaces is the codensity monad of the subcategory of finite-dimensional vector spaces. Let us now consider another example of a completion-like endofunctor on the category of spaces.

On the $\infty$-category of spaces, there is an $\infty$-monad $\MM^R$, called the reduced $R$-homology $\infty$-monad. It sends a space $X$ to a space, whose homotopy groups  are the reduced $R$-homology groups of $X$. The Bousfield–Kan $R$-completion functor 
\begin{equation}
R_\infty : \Spc \longrightarrow \Spc
\end{equation}
is defined as the $\MM^R$-completion. As an application of our results, we obtain the following description of this functor. We denote by 
\begin{equation}
\KK(R)\subseteq \Spc
\end{equation}
the full $\infty$-subcategory of the $\infty$-category of spaces spanned by the empty space and by spaces homotopy equivalent to products of Eilenberg–MacLane spaces of the form $K(A,n)$, where $A$ is an $R$-module and $n \geq 0$.

\begin{theorem*}[Th. \ref{th:BK}]
The  Bousfield-Kan $R$-completion 
is equivalent to the codensity $\infty$-monad $T_{\KK(R)}$ as a coaugmented functor.
\end{theorem*}

\begin{corollary*}[Cor. \ref{cor:BK-universal}] 
The Bousfield-Kan $R$-completion  satisfies the following two universal properties.  
\begin{itemize}
\item[(a)] $R_\infty$ is the terminal object in the $\infty$-category of $\KK(R)$-preserving coaugmented functors. 
\item[(b)] The structure of a coaugmented functor on $R_\infty$ extends to an $\infty$-monad structure uniquely up to a contractible space of choices. This $\infty$-monad is the terminal object in the $\infty$-category of $\KK(R)$-preserving $\infty$-monads.
\end{itemize}
\end{corollary*}

Note that the latter structure of $\infty$-monad on $R_\infty$ is complicated to define in the classical terms. In the original book of Bousfield-Kan there is a mistake in the definition and Bousfield writes an erratum in \cite[\S 8.9]{bousfield2003cosimplicial} about the mistake (see also \cite{libman2003homotopy}). Later Bauer and Libman proved that there is a structure of an  $A_\infty$-monad on $R_\infty$ \cite{bauer2009simplicial}, \cite{bauer2010infinity}. In our theory, the $\infty$-monad structure arises as the canonical $\infty$-monad structure on a codensity $\infty$-monad. In order to define the canonical  $\infty$-monad structure on $T_\DD$, we prove the following general result.

\medskip 

\noindent {\bf Theorem} {(Th. \ref{th:terminal:coaug}).} {\it\ Let $\EE$ be a strict monoidal $\infty$-category and $\EE^{\1/}$ be the $\infty$-category of coaugmented objects. Then a terminal object of $\EE^{\1/}$ has a structure of $\infty$-monoid which is unique up to a contractible space of choices. Moreover, this $\infty$-monoid structure defines a terminal object in the $\infty$-category of $\infty$-monoids in $\EE$.}

\subsection{$\infty$-monads} 

Here, we discuss our approach to the notion of a monad on an $\infty$-category.
 There are several ways to define them. Lurie, in \cite{Lur07}, first defines a monoidal $\infty$-category as a cocartesian fibration over $\Ner(\Delta)^{\op}$ satisfying certain conditions, and defines an algebra object as a section of this fibration satisfying certain properties. An $\infty$-monad is then defined as an algebra object in the monoidal $\infty$-category of endofunctors of an $\infty$-category. Another approach is given by Riehl and Verity in \cite{riehl2016homotopy}. They consider a quasi-categorically enriched category $\mathcal{K}$ and define an $\infty$-monad as a quasi-categorically enriched functor from a certain one-object quasi-categorically enriched category $\underline{\sf Mnd}\to \mathcal{K}$. The equivalence between these two approaches was established by Haugseng \cite[\S 8]{haugseng2020lax}. Since our main examples of monoidal $\infty$-categories are subcategories of the $\infty$-category of endofunctors $\End(\CC)$, which carries a natural associative composition operation, we adopt the more elementary approach of Riehl and Verity, reformulated in a form more suitable for our purposes.

By an $\infty$-category, we always mean a quasi-category. A strict monoidal $\infty$-category $\mathcal{E}$ is defined as a simplicial monoid, whose underlying simplicial set is an $\infty$-category. For example, the nerve $\Ner(E)$ of an ordinary strict monoidal category $E$ is a strict monoidal $\infty$-category. One can also view a strict monoidal $\infty$-category as a quasi-categorically enriched category with a single object. An $\infty$-monoid in a strict monoidal $\infty$-category $\mathcal{E}$ is a morphism of simplicial monoids
\begin{equation}
\mathcal{M} \colon \Ner(\Delta_+) \longrightarrow \mathcal{E},
\end{equation}
where $\Delta_+$ is the augmented simplex category with its standard strict monoidal structure. The join-slice adjunction defines a cosimplicial object in the slice category of (the underlying $\infty$-category of) $\EE$ 
\begin{equation}
\MM^\bullet : \Ner(\Delta) \longrightarrow \EE_{\1/}
\end{equation}
that we call the associated cosimplicial coaugmented object. The $\infty$-category of $\infty$-monoids is defined as the mapping space in the simplicial category of simplicial monoids
\begin{equation}
\mon_\infty(\EE) = \Map_{\sMon}(N(\Delta_+),\EE).
\end{equation}
Given an $\infty$-category $\mathcal{C}$, the $\infty$-category of endofunctors $\End(\mathcal{C})=\Fun(\CC,\CC)$ carries a natural strict monoidal structure defined by the composition. An $\infty$-monad on $\mathcal{C}$ is defined as an $\infty$-monoid in $\End(\mathcal{C}).$

\subsection{Bousfield-Kan completion} Let us briefly recall the basic information about the Bousfield-Kan $R$-completion.
For a commutative ring $R$, Bousfield and Kan, in \cite{bousfield1972homotopy}, introduced the endofunctor of $R$-completion on the category of  spaces 
$R_\infty : \Spc \to \Spc.$
If $X$ is a nilpotent space, then $\QQ_\infty X \simeq X_\QQ$ is the classical rationalization of $X$, and $\ZZ_\infty X \simeq X$. If $X$ has $\ZZ/p$-homology of finite type, then $(\ZZ/p)_\infty X \simeq \widehat{X}_p$ is Sullivan’s $p$-adic completion of $X$. If $X$ is a connected space with perfect fundamental group, then $\ZZ_\infty X \simeq X^+$ is Quillen’s plus construction \cite[\S VII.3.2]{bousfield1972homotopy}, \cite{KT76}. Thus, the Bousfield–Kan completion is a generalization of the rationalization and $p$-adic completion functors to arbitrary spaces $X$ and commutative rings $R$, which  also provides a meaningful construction for $R = \ZZ$. On the other hand, for any $X$, there exists a natural Adams-like spectral sequence, whose first page depends naturally only on $R$-homology of $X$,  and, under some assumptions on $X$, converges to the homotopy groups of $R_\infty X$ (for a general approach to the Bousfield-Kan spectral sequence, see \cite{bendersky2000bousfield}). So, $R_\infty X$ can be treated as a natural target for the unstable Adams spectral sequence. Note that the Bousfield-Kan $R$-completion is not an idempotent functor,   $R_\infty(R_\infty X)$ is not homotopy equivalent to $R_\infty X$ even if  $X$ is the wedge of two circles and $R\subseteq \QQ$ \cite{ivanov2019finite} or $R=\ZZ/n$ \cite{bousfield1992p, ivanov2018discrete}. 

\subsection{Further implications and developments}

An important motivation and direction of further developments is the consideration of generalized homology. The monad associated with a ring spectrum $U,$ is written as:
$$
X\to U(X):= \Omega^\infty (X\otimes U).
$$
 It is often considered in works on generalized homology, chromatic homotopy, and general stable homotopy theory.
On homotopy groups, the coaugmentation  gives the associated Hurewicz map for $U$-homology. The question arises naturally regarding characterization and the basic properties of the associated completion monad  $U_\infty=\Tot \: U^\bullet $. By the general theorems above, this completion has an essentially unique $\infty$-monad structure, and is determined by the expected universal property of being the terminal monad preserving $U$-algebras. Moreover, it is essential to consider its behavior with respect to finite products and a principal fibration sequence.  Further, we
apply the above  appraoch to show that a  modification of
  $X\to  \Tot \,U^\bullet X,$ preserves, in a pro-sense, the
$U$-homology. A similar $U$-homology completion can be associated with the {\em Bredon homology} in the category  of $G$-spaces. A candidate monad  is  given by the  Elmendorf-like co-end:

$$X\to \int_{G/H} G/H\otimes U(X^H), $$
 yielding a $G$-space, which is a $U$-algebra  whose $H$-fixed points are  all the corresponding  $U$-algebras.

In the direction of Goodwillie calculus, one can consider the subcategory spanned by the $\infty$-category of polynomial functors in the functor category $\Spc \to \Spc,$ or $\Spc \to {\sf Sp}.$ By the above theorems, the codensity monad $T_\DD,$ of this sub-category is the terminal monad preserving all polynomial monads.
Its properties as a terminal  monad, e.g with respect to products of functors, are a natural part of the further development and employment of the above concepts.

\subsection{Conventions}  Throughout the paper by an $\infty$-category we mean a quasi-category. The category of simplicial monoids $\sMon$ is  important for us, and we reserve the notation $\Map(\cdot ,\cdot)$ with the capital letter ${\sf M}$ for the mapping space in the category $\sMon$ that we describe in Subsection \ref{subsec:simplicial_model_structure}.   The mapping space between two objects in an $\infty$-category $\CC$ is denoted by $\map_\CC(\cdot ,\cdot ),$ and the mapping space in the category of simplicial sets (resp. pointed simplicial sets) is denoted by $\map(X,Y)$ or $Y^X$ (resp. $\map_*(X,Y)$). For an object $c$ of an $\infty$-category $\CC$, we use three different models of over- and under-categories: ordinary model  $\CC_{c/}$, $\CC_{/c}$ \cite[\S 1.4]{land2021introduction}, fat model $\CC^{c/}$, $\CC^{/c}$ \cite[\S 2.5]{land2021introduction} and ``twisted arrow'' model  $\CC_{c|}$, $\CC_{|c}$ that we discuss in Section \ref{section:functorial-over-cat}. 
\medskip

\subsection{Acknowledgements} 
We are thankful for the contribution of Guy Kapon and Shaul Ragimov
to this project via helpful discussion and a written draft containing relevant results.

\tableofcontents

\section{\bf Simplicial monoids}\label{section:simplicial_monoids}

In this paper, we adopt an approach to the definition of monads in the setting of $\infty$-categories based on the language of simplicial monoids. To support this, we first develop some technical background on simplicial monoids. In this section, we begin by recalling the standard simplicial model category structure on the category of simplicial monoids, with particular emphasis on cofibrations, which are retracts of the so-called \emph{free maps}. We then study certain distinguished classes of free maps, which we refer to as (inner, left, right) anodyne free maps. These are analogous to the (inner, left, right) anodyne maps of simplicial sets. We show that a left anodyne free map of simplicial monoids has a left lifting property with respect to a left fibration of simplicial monoids (Proposition \ref{prop:lifting:anodyne}), which plays an important role in the next section.  We conclude by introducing the notion of a strict monoidal $\infty$-category and proving some properties of terminal objects in them (Lemma \ref{lemma:contractible_fiber}). 

\subsection{Simplicial model structure}
\label{subsec:simplicial_model_structure} In this subsection  we recall the definition of the standard simplicial model category structure on the category of simplicial monoids \cite[\S II.4]{quillen2006homotopical} (see also  \cite{dwyer1980simplicial} for the case $O=*$). We also recall the definition of the tensored and cotensored enrichment of the category of simplicial monoids over pointed simplicial sets. 

We consider the category of simplicial monoids $\sMon$ and the forgetful functor to the category of simplicial sets 
\begin{equation}
U: \sMon \longrightarrow \sSet.
\end{equation}
Then the category $\sMon$ is a  model category with the following definition of weak equivalences, fibrations, and cofibrations. 
\begin{itemize}
    \item[(W)]  A map of simplicial monoids $f$ is a weak equivalence if  $Uf$ is a weak equivalence of simplicial sets.
    \item[(F)] A map of simplicial monoids $f$ is a fibration if the map $Uf$ is a fibration of simplicial sets.
    \item[(C)] A map is a cofibration if it has a left lifting property with respect to those fibrations which are week equivalences. 
\end{itemize}

On the next subsection, we give a more explicit description of the class of cofibrations in terms of \emph{free maps}.

Let us define tensored and cotensored simplicial enrichment of $\sMon$, making it a simplicial model category. If $S$ is a set and $M$ is a monoid, we set 
$S \otimes M = \coprod_{s\in S} M.$ 
If $S$ is a simplicial set and $M$ is a simplicial monoid, the product $S\otimes M$ is defined component-wise $(S\otimes M)(J) =S(J)\otimes M(J)$. Then the enrichment is defined by 
\begin{equation}\label{eq:Map}
  \Map(M,N)(J) = \Hom_{\sMon}(\Delta^J\otimes M,N).  
\end{equation}
If $S$ is a simplicial set and $M$ is a simplicial monoid, the simplicial monoid $M^S$ is defined as the ordinary function complex of simplicial sets $(UM)^S$ equipped with the structure of a simplicial monoid defined point-wise (see \cite[II.1.11]{quillen2006homotopical}). In particular we have the identity 
\begin{equation}
U(M^S) = (UM)^S
\end{equation}
and the adjunction isomorphisms
\begin{equation}\label{eq:isom:map-map}
\Map(S\otimes M,N) \cong \Map(M,N)^S \cong \Map(M, N^S)
\end{equation}
for any simplicial monoids $M,N$ and a simplicial set $S.$
In particular, we have the isomorphism
\begin{equation}
\Map(S\otimes \NN,M) \cong (UM)^S,
\end{equation}
where $\NN$ is the monoid of non-negative integers treated as the constant simplicial monoid.

The category $\sMon$ can also be treated as a category enriched over pointed simplicial sets, because $\Map(M,N)$ has a natural base point. This enrichment is also tensored and cotensored. If $S$ is a pointed set with a base point $s_0$ and $M$ is a monoid, we consider the quotient monoid 
\begin{equation}
S \sm M = {\sf Coker}(\{s_0\}\otimes M \to  S\otimes M).
\end{equation}
Note that there is an isomorphism
$S \sm M \cong  (S\setminus \{s_0\})\otimes M.$ If $S$ is a simplicial set and $M$ is a simplicial monoid then $S\sm M$ is defined component-wise. The pointed cotensoring is  defined as the kernel 
\begin{equation}
\map_*(S,M) = {\sf Ker}(M^S \to M^{\{s_0\}} ).
\end{equation}
It is easy to see that 
\begin{equation}
U(\map_*(S,M)) = \map_*(S,UM)
\end{equation}
and that there are adjunction isomorphisms 
\begin{equation}\label{eq:adjunction_pointed}
\Map(S\sm M,N) \cong \map_*(S,\Map(M,N)) \cong \Map(M, \map_*(S,N)).
\end{equation}
In particular, we have 
\begin{equation}
\Map(S \sm \NN,M) \cong \map_*(S,UM).
\end{equation}

For any simplicial monoids $M$ and $N$, there is an isomorphism 
\begin{equation}\label{eq:Map-power}
\Map(M,N)(J) \cong \Hom_{\sMon}(M,N^{\Delta^J}).
\end{equation}
For any two pointed simplicial sets $X$ and $Y$, there is an isomorphism 
\begin{equation}\label{eq:map*-power}
\map_*(X,Y)(J)\cong \Hom_{\sSet_*}(X,Y^{\Delta^J}).
\end{equation}
Therefore, there is a monomorphism 
\begin{equation}
\Map(M,N)\hookrightarrow  \map_*(UM,UN). 
\end{equation}

\begin{lemma}\label{lemma:Map_eq} 
For any simplicial monoids $M,N$, the simplicial set $\Map(M,N)$ is isomorphic is the equaliser of two maps
\begin{equation}
\Map(M,N) \to \map_*(UM,UN) \rightrightarrows \map_*(UM^2,UN),
\end{equation}
where the first map is induced by the multiplication map $UM^2\to UM,$ and the second map is the composition of the square map $\map_*(UM,UN)\to \map_*(UM^2,UN^2)$, and the map induced by the multiplication $UN^2 \to UN$.
\end{lemma}
\begin{proof} For any two ordinary monoids $M,N$ the hom-set $\Hom_{\Mon}(M,N)$ is the equaliser of the two maps 
\begin{equation}
\Hom_{\Mon}(M,N) \to \Hom_{\Set_*}(UM,UN)\rightrightarrows \Hom_{\Set_*}(UM^2,UN).
\end{equation} 
Therefore, the similar equaliser diagram exists for the hom-set between simplicial monoids $M$ and $N$. Combining this with \eqref{eq:Map-power} and \eqref{eq:map*-power} we obtain  the equalizer diagram 
\begin{equation}
\Map(M,N)_n \to \map_*(UM,UN)_n \rightrightarrows \map_*(UM^2,UN)_n
\end{equation}
for any $n\geq 0.$
A direct computation shows that these two maps are the required maps. 
\end{proof}

\subsection{Free maps}

\begin{definition} \label{def:free_simp_monoid}
We say that a simplicial monoid $A$ is \emph{free} if there is a sequence of subsets $(X_n\subseteq A_n)_{n\geq 0}$ closed under degeneracy maps $s_i(X_n)\subseteq X_{n+1}$ such that $A_n$ is the free monoid generated by $X_n$. The sequence of subsets $(X_n)_{n\geq 0}$ is called a basis of $A$.
\end{definition}
\begin{definition} \label{def:free_map} 
An injective map of simplicial monoids $\alpha : B \to A$ is \emph{free} if is there is a sequence of subsets $(X_n\subseteq A_n)_{n\geq 0}$ closed under degeneracy maps $s_i(X_n)\subseteq X_{n+1}$ such that  $A_n = \alpha(B_n) \sqcup \langle X_n \rangle$, where $\langle X_n \rangle$ denotes the free monoid generated by $X_n$ (see \cite[§7.4]{dwyer1980simplicial}). The sequence of subsets $(X_n)_{n\geq 0}$ is called a \emph{basis} of the free map $\alpha$.
\end{definition}

For a basis $(X_n)_n$ of a free map $\alpha: B \to A$, we define its degenerate part as 
\begin{equation}
X_n^{\sf dgn} =  \bigcup_i s_i(X_{n-1}), \hspace{1cm} n\geq 1, 
\end{equation}
and $X_0^{\sf dgn} = \emptyset$,
and the non-degenerate part as  
\begin{equation}
X_n^\circ = X_n\setminus X_n^{\sf dgn}. 
\end{equation}

The class of cofibrations in $\sMon$ can be described as follows (see \cite[\S 7.6]{dwyer1980simplicial}). 
\begin{itemize}
    \item[(C')] A map of simplicial monoids is a cofibration if and only if it is a retract of a free map. 
\end{itemize}
In particular, free maps are cofibrations, and free simplicial monoids are cofibrant. 

Next, we prove a lemma on a filtration of a free map, which helps to prove statements about free maps by induction. This lemma can be regarded as a refinement of the fact that the class of cofibrations is the weakly saturated class generated by the maps $\partial \Delta^n \otimes \NN \hookrightarrow \Delta^n\otimes \NN$.

\begin{lemma}\label{lemma:free:pushout} 
Let $\alpha : B\to A$ be a free map with a basis $(X_n)_{n\in \mathbb{N}}$. Then there is an exhausting filtration of $A$ by  simplicial submonoids 
\begin{equation}
A^{(-1)} \subseteq A^{(0)}\subseteq A^{(1)} \subseteq \dots \subseteq A 
\end{equation}
such that $\alpha(B) = A^{(-1)},$ and, for any $n\geq 0$, we have $X_n^\circ \subseteq A_n^{(n)},$ the inclusion $A^{(n-1)}\hookrightarrow A^{(n)}$ is free and there is a pushout of the form 
\begin{equation}
\label{eq:pushout:free}
\begin{tikzcd}
\coprod\limits_{X_n^\circ} \partial \Delta^n \otimes \NN  
\ar[r]
\ar[d,hookrightarrow] 
& 
A^{(n-1)} 
\ar[d,hookrightarrow]
\\
\coprod\limits_{X_n^\circ} \Delta^n \otimes \NN  
\ar[r]
&
A^{(n)}
\end{tikzcd}
\end{equation}
where the maps $\Delta^n \otimes \NN \to A^{(n)}$ indexed by $X_n^\circ$ correspond to the simplices of  $X_n^\circ.$
\end{lemma}
\begin{proof}
For simplicity, assume that $B\subseteq A$ and $\alpha$ is the canonical embedding.  Then 
\begin{equation}
A_n = B_n \sqcup \langle X^{\sf dgn}_n \rangle \sqcup \langle X_n^\circ \rangle.   
\end{equation}
Since $s_i(A_{n-1}) \subseteq B_n \sqcup \langle X_n^{\sf dgn} \rangle$, the elements of $X_n^\circ$ are non-degenerate in $A$. We claim that $X_n$ can be presented as the disjoint union  
\begin{equation}
X_n = \coprod_{\substack{0\leq k\leq n\\ \sigma:[n]\twoheadrightarrow [k]} } \sigma^*(X_k^\circ),
\end{equation}
where $\sigma$ runs over all epimorphisms in $\Delta$ from $[n]$ to $[k]$ for all $0\leq k\leq n$. The fact that $X_n$ is the union of $\sigma^*(X_k^\circ)$ is obvious. The fact that it is disjoint follows from the fact that any simplex of any simplicial set can be uniquely presented as $\sigma^*(y),$ where $\sigma$ is an epimorphism in  $\Delta$ and $y$ is a non-degenerate simplex \cite[\S II.3.1]{gabriel2012calculus}.

For any $n\geq 0$, we define a filtration by subsets 
\begin{equation}
X_n^{(0)} \subseteq X_n^{(1)} \subseteq \dots \subseteq X_n^{(n)} = X_n^{(n+1)} = \dots = X_n 
\end{equation}
by $X_n^{(k)}=X_n$ for $n\leq k,$ and 
\begin{equation}
X_n^{(k)} \  = \coprod_{\sigma: [n] \twoheadrightarrow [k]} \sigma^*(X_k) 
\end{equation}
for $k<n$.  It can also be represented as 
\begin{equation}\label{eq:X_n^k:2}
X_n^{(k)} \  = X_n^{(k-1)} \sqcup  \coprod_{\substack{ \sigma: [n] \twoheadrightarrow [k]}} \sigma^*(X^\circ_k). 
\end{equation}
Using the simplicial identities and the inclusion $d_i(X_k) \subseteq B_{k-1} \sqcup \langle X_{k-1} \rangle$, it is easy to check that, for any $n, k\geq 0$, we have  
\begin{equation}\label{eq:d_ix^k}
d_i(X_{n+1}^{(k)}) \subseteq B_{n} \sqcup \langle X_{n}^{(k)} \rangle, \hspace{1cm} s_i(X_n^{(k)}) \subseteq X_{n+1}^{(k)}. 
\end{equation}

Consider the monoids 
$A_n^{(k)} = B_{n} \sqcup \langle X_{n}^{(k)} \rangle.$ The inclusions  \eqref{eq:d_ix^k} implies that the collection $(A_n^{(k)})$ forms a simplicial submonoid $A^{(k)}$. 
Then the formula \eqref{eq:X_n^k:2} implies that
\begin{equation}
A_n^{(k)} = A_n^{(k-1)} \sqcup \left(\coprod_{\sigma:[n]\twoheadrightarrow [k]} \langle \sigma^*(X_n^\circ)  \rangle \right).
\end{equation} 
Note that 
\begin{equation}
(\Delta^k)_n\setminus(\partial \Delta^k)_n = \{ \sigma:[n]\twoheadrightarrow [k]\}   
\end{equation}
and the inclusions $d_i(X_k)\subseteq  A^{(k-1)}_{k-1}$ for any $i$, imply the inclusions  $f^*(X_k^\circ) \subseteq A^{(k-1)}_n$ for any $f\in (\partial \Delta^k)_n.$
Therefore, we have a pushout of monoids

\begin{equation}
\begin{tikzcd}
\coprod\limits_{x\in X_k^\circ} \coprod\limits_{f\in (\partial \Delta^k)_n}   \NN 
\ar[r]
\ar[d]
& 
A^{(k-1)}_n 
\ar[d]
\\
\coprod\limits_{x\in X^\circ_k}  \coprod\limits_{f\in  (\Delta^k)_n} \NN 
\ar[r,"\varphi"]
&
A_n^{(k)},
\end{tikzcd}
\end{equation}
where $\varphi_{x,f}:\NN\to  A_n^{(k)}$ is defined by the element $f^*(x)$.
Since the pushout of simplicial monoids is defined component-wise, we obtain the required pushout. 
\end{proof}

\subsection{Anodyne free maps}

Let $\alpha:B\to A$ be a free map with a basis $(X_n)_{n\in \mathbb{N}}$. A collection of subsets $(H_n)_{n\in \mathbb{N}}$ in the non-degenerate part of the basis $H_n\subseteq X^\circ_n$ is called a collection of \emph{horn generators} if we can choose numbers $0\leq i(h)\leq n,$ for all $n\geq 1$ and $h\in H_n$, such that:
\begin{itemize}
\item $d_{i(h)}(h)\in X^\circ_{n-1} \setminus H_{n-1};$
\item for any $x\in X_{n-1}^\circ\setminus H_{n-1}$  there exists a unique $h \in H_{n}$ such that $x=d_{i(h)}(h)$. 
\end{itemize}
In other words, the inclusion $H_n \hookrightarrow X_n^\circ$ and the map $H_{n+1}\to X_n^\circ,$ $h \mapsto d_{i(h)}(h)$ define a bijection
\begin{equation}
H_{n+1} \sqcup  H_n \cong  X_n^\circ   
\end{equation}
for any $n\geq 0.$
The numbers $i(h)$ will be referred as  \emph{horn indices} of the horn generators. 

A free map is called an \emph{anodyne free map} if it has a basis that has a collection of horn generators. An anodyne free map is called an \emph{inner (resp. left; right) anodyne free map}, if it has a basis $(X_n)$ that has horn generators $(H_n)$ whose horn indices satisfy $0<i(h)<n$ (resp. $0\leq i(h)<n$; $0<i(h)\leq n$) for $h\in H_n.$ 

A free monoid $A$ is called anodyne free (resp. inner anodyne free; left anodyne free; right anodyne free) if the map $1\to A$ is so. In this case a collection of horn generators of $1\to A$ is called a collection of horn generators of $A$.

\begin{example}
Consider the quotient simplicial set $\Delta^n/\Lambda^n_i$, which is defined as the pushout.
\begin{equation}
\begin{tikzcd} 
\Lambda^n_i
\ar[r] \ar[d] 
& 
\Delta^n 
\ar[d]
\\
\Delta^0 
\ar[r]
&
\Delta^n/\Lambda^n_i 
\end{tikzcd}
\end{equation} 
The simplicial set $\Delta^n/\Lambda^n_i$ can be viewed as a simplicial model of an $n$-disk with $3$ non-degenerate simplices: the base-point,  the $n$-simplex that we denote by  $h$, and its $i$-th face $d_i(h)$. If we denote by $X_k$ the set of all $k$-simplices of $\Delta^n/\Lambda^n_i,$ except the degeneracy of the base-point, we obtain that $(X_k)_k$ is a basis of 
$(\Delta^n/\Lambda^n_i)\sm \NN$. The non-degenerated part of this basis consists of two simplices: $X_n^\circ = \{h\}$ and $X^\circ_{n-1} =\{d_i(h)\},$ and $X^\circ_k=\emptyset$ for $k\notin \{n,n-1\}.$  Therefore, the collection $(H_k)_k$ defined by  $H_n=\{h\}$ and $H_k=\emptyset, k\neq n$ is a collection of horn generators of $(\Delta^n/\Lambda^n_i)\sm \NN$ with the horn index $i(h)=i$. It follows that  $(\Delta^n/\Lambda^n_i)\sm \NN$ is an inner anodyne free (resp. left anodyne free; right anodyne free) simplicial monoid if $0<i<n$ (resp. $0\leq i<n$; $0<i\leq n$). 
\end{example}

The following lemma says that an anodyne free map $\alpha:B\to A$ can be obtained as an infinite composition of maps obtained by attaching simplices along horns.

\begin{lemma}\label{lemma:inner-pushout}
Let $\alpha : B\to A$ be an anodyne  free map with a basis $(X_n)_{n}$ and its horn generators $(H_n)_{n}$. Then there is an exhausting filtration of $A$ by  simplicial submonoids 
\begin{equation}
\tilde A^{(-1)} \subseteq \tilde A^{(0)}\subseteq \tilde A^{(1)} \subseteq \dots \subseteq A 
\end{equation}
such that $\alpha(B) = \tilde A^{(-1)},$  and, for any $n\geq 0$, we have $H_n\subseteq \tilde A_n^{(n)}$, and there is a pushout of the form 
\begin{equation}
\begin{tikzcd}
\coprod\limits_{h\in H_n} \Lambda^n_{i(h)} \otimes \NN  
\ar[r]
\ar[d,hookrightarrow]
&
\tilde A^{(n-1)}
\ar[d,hookrightarrow]
\\
\coprod\limits_{h\in H_n} \Delta^n \otimes \NN 
\ar[r]
&
\tilde A^{(n)},
\end{tikzcd}
\end{equation}
where the maps $\Delta^n \otimes \NN \to \tilde A^{(n)}$ correspond to elements  of $h\in H_n,$ and $i(h)$ denotes the horn index of $h.$
\end{lemma}
\begin{proof} 
In this proof we use the notations from the proof of Lemma \ref{lemma:free:pushout}. So we have a filtrations $X_n^{(k)}$ and $A^{(k)}$  and want to construct new filtrations $\tilde X_n^{(k)}$ and $\tilde A^{(k)}.$ Set 
\begin{equation}
H_{n+1}'=X_n^\circ\setminus H_n.    
\end{equation}
Then there is a bijection $H_{n+1} \cong H_{n+1}'$ defined by $h\mapsto d_{i(h)}(h)$. For $-1\leq k\leq n$, we set 
\begin{equation}
\tilde X_n^{(k)} \  = \left( \coprod_{\substack{ 0\leq l\leq k\\ \sigma: [n] \twoheadrightarrow [l]}} \sigma^*(H_l) \right) 
\sqcup
\left( \coprod_{\substack{ 1\leq l\leq k\\ \tau: [n] \twoheadrightarrow [l-1]}} \tau^*(H'_l) \right).
\end{equation}
Recursively, it can be presented as
\begin{equation}
\tilde X_n^{(k)} = \tilde X_n^{(k-1)} \sqcup  \left( \coprod_{ \sigma: [n] \twoheadrightarrow [k]} \sigma^*(H_k) \right) 
\sqcup
\left( \coprod_{ \tau: [n] \twoheadrightarrow [k-1]} \tau^*(H'_k) \right),
\end{equation}
where $\tilde X_n^{(-1)}=\emptyset$. Since $X^\circ_l=H_l \sqcup H'_{l+1},$ we can also present it as 
\begin{equation}
\tilde X^{(k)}_n = \left( \coprod_{ \sigma: [n] \twoheadrightarrow [k]} \sigma^*(H_k) \right) 
\sqcup X^{(k-1)}_n.
\end{equation}
In particular, we have the inclusions $ 
X^{(k-1)}_n \subseteq  \tilde X_n^{(k)} \subseteq  X^{(k)}_n$ and the equality  $\tilde X^{(n)}_n = H_n \sqcup X_n^{(n-1)}.$  For $k > n$, we set  
$\tilde X_n^{(k)} = X_n.$ Using the inclusions \eqref{eq:d_ix^k}, it is easy to check that, for any $n,k\geq 0$, we have  
\begin{equation}
d_i(\tilde X_{n+1}^{(k)}) \subseteq B_{n} \sqcup \langle \tilde X_{n}^{(k)} \rangle, \hspace{1cm} s_i( \tilde X_n^{(k)}) \subseteq \tilde X_{n+1}^{(k)}. 
\end{equation}
Therefore, if we define $\tilde A^{(k)}_n = B_n \sqcup \langle \tilde X_n^{(k)} \rangle,$ we obtain a filtration by simplicial submonoids $\tilde A^{(k)}.$ We also have a recursive formula 
\begin{equation}
\tilde A_n^{(k)} = \tilde A_n^{(k-1)} \sqcup  \left( \coprod_{ \sigma: [n] \twoheadrightarrow [k]} \langle \sigma^*(H_k) \rangle \right) 
\sqcup
\left( \coprod_{ \tau: [n] \twoheadrightarrow [k-1]} \langle \tau^*(H'_k) \rangle \right),
\end{equation}
where $\tilde A^{(-1)}_n=B_n.$ Note that 
\begin{equation}
(\Delta^k)_n\setminus(\Lambda^k_i)_n = \{ \sigma\mid \sigma:[n]\twoheadrightarrow [k]\} \sqcup \{ d^i\tau \mid \tau:[n]\twoheadrightarrow [k-1] \} 
\end{equation}
and, for any $h\in H_k$ and $f\in (\Lambda^k_{i(h)})_n$, we have $f^*(h) \subseteq \tilde A^{(k-1)}_n.$ Therefore, we have a pushout of monoids.
\begin{equation}
\begin{tikzcd}
\coprod\limits_{h\in H_k} \coprod\limits_{f\in (\Lambda^k_{i(h)})_n}   \NN  
\ar[r]
\ar[d,hookrightarrow]
& 
\tilde A^{(k-1)}_n 
\ar[d,hookrightarrow]
\\
\coprod\limits_{h\in H_k} \coprod\limits_{f\in (\Delta^k)_n}  \NN
\ar[r,"\varphi"]
&
\tilde A_n^{(k)},
\end{tikzcd}
\end{equation}
where $\varphi_{h,f}:\NN\to  \tilde A_n^{(k)}$ is defined by the element $f^*(h)$. Since pushout of simplicial monoids is taken component-wise, we obtain the required pushout of simplicial monoids. 
\end{proof}

\begin{proposition}\label{prop:lifting:anodyne} Let $\varphi:M\to N$ be a map of simplicial monoids such that $U\varphi$ is a Kan (resp. inner; left; right) fibration and $\alpha:B\to A$ an anodyne (resp. inner anodyne; left anodyne; right anodyne) free map. Then $\alpha$ has the left lifting property with respect to $\varphi$. 
\end{proposition}
\begin{proof}
The class of maps having left lifting property with respect to $\varphi$ is saturated. Therefore, it is closed with respect to coproducts and stable under pushouts and transfinite compositions.  Combining this with Lemma \ref{lemma:inner-pushout} we see that the problem is reduced to the case when $\alpha$ is equal to the inclusion $\Lambda^n_i\otimes \NN \to \Delta^n \otimes \NN$ for any $0\leq i\leq n$ (resp. $0< i< n$, $0\leq i< n$, $0<i\leq n$). By the adjunction the map of simplicial monoids $\Lambda^n_i\otimes \NN \to \Delta^n \otimes \NN$ has the left lifting property with respect to $\varphi$ if and only if the map of simplicial sets  $\Lambda^n_i \to \Delta^n$ has the left lifting property with respect to $U\varphi$, which follows from the assumption.  
\end{proof}

\subsection{Pullback power}

\begin{proposition}\label{prop:pullback_power}
Let $\alpha:B\to A$ be a cofibration of simplicial monoids, $\varphi :M\to N$ be a morphism of simplicial monoids and
\begin{equation}
\varphi^{\Box \alpha} : \Map(A,M) \longrightarrow \Map(B,M) \times_{\Map(B,N)} \Map(A,N)
\end{equation}
be the pullback power. Then the following holds. 
\begin{itemize}
    \item[(a)] If $U\varphi$ is a Kan (resp. inner; left; right) fibration, then $\varphi^{\Box \alpha}$ is a Kan (resp. inner; left; right) fibration.  
    \item[(b)] If, in addition, $\alpha$ is an anodyne (resp. inner anodyne; left anodyne; right anodyne) free map, then $\varphi^{\Box \alpha}$ is a trivial fibration. 
\end{itemize} 
\end{proposition}
\begin{proof} (a) Take an anodyne (resp. inner anodyne; left anodyne; right anodyne) map of simplicial sets $i: Y\to X$. We need to prove that the lifting problem 
\begin{equation}\label{eq:lift1}
\begin{tikzcd}
Y \ar[r]\ar[d,"i"] & \Map(A,M) \ar[d,"\varphi^{\Box \alpha}"] \\
X\ar[ru,dashed] \ar[r] & \Map(B,M) \times_{\Map(B,N)} \Map(A,N) 
\end{tikzcd}
\end{equation}
has a solution. This lifting problem is equivalent to the lifting problem 
\begin{equation}\label{eq:lift2}
\begin{tikzcd}
B \ar[r]\ar[d,"\alpha"] & M^{X} \ar[d,"\varphi^{\Box i}"] \\
A\ar[ru,dashed] \ar[r] & M^{Y} \times_{N^{Y} } N^{X}
\end{tikzcd}
\end{equation}
(see \cite[Lemma 9.3.6]{hirschhorn2003model}). Since $U(M^X)=(UM)^X$ and $U$ is a right adjoint, we obtain $U(\varphi^{\Box i}) = (U\varphi)^{\Box i}$. The map $(U\varphi)^{\Box i}$ is a trivial fibration by \cite[Th.1.3.37]{land2021introduction}. Therefore, the lifting problem \eqref{eq:lift2} has a solution. 

(b). The proof is similar to (a). We need to take any monomorphism $i: Y\to X$ and prove that the lifting problem  \eqref{eq:lift1} has a solution. The lifting problem is equivalent to the lifting problem   \eqref{eq:lift2}. By \cite[Th.1.3.37]{land2021introduction} we obtain that $(U\varphi)^{\Box i}=U( \varphi^{\Box i})$ is a Kan (resp. inner; left; right) fibration. Then the lifting problem  \eqref{eq:lift2} has a solution by Proposition \ref{prop:lifting:anodyne}.    
\end{proof}

In the next section we use the following special case of Proposition \ref{prop:pullback_power} to prove that the trivial $\infty$-monoid is the initial object in the $\infty$-category of $\infty$-monoids  (Corollary \ref{cor:trivial_monoid_is_initial}).

\begin{corollary}\label{cor:left_anodyne_free_monoid}
If $A$ is a left anodyne free simplicial monoid and $\varphi:M\to N$ is a morphism of simplicial monoids such that $U\varphi$ is a left fibration, then 
\begin{equation}
\varphi_* : \Map(A,M) \longrightarrow \Map(A,N)
\end{equation}
is a trivial fibration. 
\end{corollary}

\subsection{Strict monoidal \texorpdfstring{$\infty$}{}-categories} 

We define a \emph{strict monoidal $\infty$-category} as a simplicial monoid $\EE$ such that $U\EE$ is an $\infty$-category.  A strict monoidal $\infty$-category can be viewed as a quasi-categorically enriched category with one object, which were considered in the paper of Riehl-Verity \cite{riehl2016homotopy}. We use the usual abuse of notation identifying a strict $\infty$-category $\EE$ with the underlying $\infty$-category $U\EE$. For example $\map_\EE(\cdot,\cdot)$ denotes the mapping space in $U\EE.$

Let us give some examples of strict monoidal $\infty$-categories. The main example for us is the category of endofunctors $\End(\CC)=\Fun(\CC,\CC)$ of an $\infty$-category $\CC$, whose strict monoidal structure is defined by the composition. Another example is the nerve $N(E)$ of an ordinary strict monoidal category $E$. More generally, if $\EEE$ is a strict monoidal simplicial category, its homotopy coherent nerve $\Ner(\EEE)$ has a natural structure of a simplicial monoid (see Subsection \ref{subsection:monoida_and_monads_in _simplicial_cats} for details). So, if $\EEE$ is, in addition, locally Kan, then $\Ner(\EEE)$ is a strict monoidal $\infty$-category.

Let us prove some properties of of strict monoidal $\infty$-categories. An inner fibration between $\infty$-categories $F:\mathcal{C}\to \mathcal{D}$ is called an isofibration if any equivalence $\alpha : F(c)\to d'$ in $\mathcal{D}$ can be lifted to an equivalence $\tilde \alpha:c\to c'$. A functor $F:\mathcal{C}\to \mathcal{D}$ is called conservative, if for a morphism $\alpha$ in $\CC$, the morphism $F(\alpha)$ is an equivalence only if $\alpha$ is an equivalence.

\begin{proposition}\label{prop:hom_from_free_monoid1}
 Let $\alpha:B\to A$ be a free map simplicial monoids and $\EE$ be a strict monoidal $\infty$-category. Consider the pre-composition map 
 \begin{equation}\label{eq:pre-comp}
\alpha^* : \Map(A,\EE)\longrightarrow \Map(B,\EE).
\end{equation}
 Then the following holds. 
\begin{itemize}
    \item[(a)] The map $\alpha^*$ is an inner fibration. In particular, if $A$ is a free simplicial monoid, $\Map(A,\EE)$ is an $\infty$-category. 
    \item[(b)] If $A$ and $B$ are free simplicial monoids and $\alpha$ induces an isomorphism of monoids of $0$-simplices $A_0 \cong B_0$, then $\alpha^*$ is a conservative isofibration. 
    \item[(c)] If $\alpha$ is an inner anodyne free map, then $\alpha^*$ is a trivial fibration.  
\end{itemize}
\end{proposition}
\begin{proof}
(a) and (c) follow from Proposition \ref{prop:pullback_power}. Let us prove (b). 

(b) First, we introduce some notation. We treat the ordinal $[n]=\{0,\dots,n\}$ as a category. In particular, $[1]$ is the category with two objects and one non-identity morphism. Then  $[1]^{\sf grpd}$ is the localization of this category by this one morphism.  Its nerve is denoted by $J=\Ner([1]^{\sf grpd})$. We consider the class of maps of simplicial sets $\mathcal{CI}\subseteq {\sf Mor}(\sSet),$ which are inner fibrations  having the right lifting property with respect to the maps $\Delta^1 \to J$ and $\Delta^0\to J.$ A functor between $\infty$-categories is a conservative isofibration if and only if it is from the class $\mathcal{CI}$ (see \cite[Cor. 2.1.18, Prop. 2.1.21]{land2021introduction}). Therefore, the maps from the class $\mathcal{CI}$ are generalizations of conservative isofibrations to the case of all simplicial sets. 

It is easy to see that the class $\mathcal{CI}$ is cosaturated (saturated class in $\sSet^{\sf op}$), because it is defined as a class of maps having the right lifting property with respect to some set of maps.  Therefore, the class of maps $\alpha$ such that $\alpha^*$ is from $\mathcal{CI}$ is saturated. Then Lemma \ref{lemma:free:pushout} allows us to reduce the statement to the case of $\alpha$ equal to the inclusion $\partial\Delta^n \otimes \NN \to \Delta^n \otimes \NN$ for $n\geq 1.$ Using the adjunction $\Map(S\otimes \NN,\EE)\cong (U\EE)^S$ we obtain that that we need to prove that $\alpha^*: (U\EE)^{\Delta^n} \to (U\EE)^{\partial \Delta^n}$ is a conservative isofibration. This follows from \cite[Th.2.2.1, Prop.2.2.5]{land2021introduction}. 
\end{proof}

Let us finish the section by the following technical lemma about terminal objects in strict monoidal $\infty$-categories that we use in the following section.

\begin{lemma}\label{lemma:contractible_fiber}
Let $A$ be a free simplicial monoid, $B\subseteq A$ be its free submonoid such that $B_0=A_0$, and let  $\EE$  be a strict monoidal $\infty$-category with a terminal object $t$ and $\varphi: B\to \EE$ be a map of simplicial monoids. Assume that the inclusion $\alpha:B\hookrightarrow A$ is a free map with a basis $(X_n)_n$, and that $\varphi_0(d_0^n(x))=t$ for any $x\in X_n$ and $n\geq 0.$ Then the fiber of the conservative isofibration  
\begin{equation}
\alpha^* : \Map(A,\EE)\longrightarrow \Map(B,\EE) 
\end{equation}
over $\varphi$ is a contractible Kan complex. 
\end{lemma}
\begin{proof}
By Lemma \ref{lemma:free:pushout}, we obtain that there is an exhaustive filtration by free  simplicial monoids 
\begin{equation}
B =A^{(0)} \subseteq A^{(1)} \subseteq \dots \subseteq A
\end{equation}
such that the inclusions $A^{(n-1)}\hookrightarrow A^{(n)}$ are free and  there is a pushout of the form
\begin{equation}
\begin{tikzcd}
\coprod\limits_{X_n^\circ} \partial \Delta^n \otimes \NN  
\ar[r]
\ar[d]
&
A^{(n-1)}
\ar[d] 
\\
\coprod\limits_{X_n^\circ} \Delta^n \otimes \NN 
\ar[r]
&
A^{(n)},
\end{tikzcd}
\end{equation}
where the maps $\Delta^n\otimes \NN\to A^{(n)}$ are defined by the elements of $X_n^\circ$. We denote by  $F^{(n)}_\varphi \subseteq \Map(A^{(n)},\EE)$ the fiber of the map $\Map(A^{(n)},\EE)\to \Map(B,\EE)$ over $\varphi$. Then 
\begin{equation}
F_\varphi=\lim F^{(n)}_\varphi, 
\end{equation}
where $F_\varphi$ is the fiber of the map $\Map(A,\EE)\to \Map(B,\EE).$ 

Using the isomorphism $\Map(X\otimes \NN,\EE)\cong \EE^X,$ we obtain that for each $n$ and a point $\theta$ of $F_\varphi^{(n-1)}$ there is a diagram consisting of three pullbacks
\begin{equation}\label{eq:three_pullbacks}
\begin{tikzcd}
G^{(n)}_{\varphi,\theta} \ar[r,rightarrowtail] \ar[d] & F^{(n)}_\varphi \ar[r,rightarrowtail] \ar[d] & \Map(A^{(n)},\EE) \ar[r] \ar[d] & \prod\limits_{X_n^\circ}  \EE^{\Delta^n} \ar[d] \\
\Delta^0 \ar[r,"\theta"] & F^{(n-1)}_\varphi \ar[r,rightarrowtail] & \Map(A^{(n-1)},\EE) \ar[r] & \prod\limits_{X_n^\circ} \EE^{\partial \Delta^n}.
\end{tikzcd}
\end{equation}
Since the map 
$\Map(A^{(n)} , \EE) \to \Map(A^{(n-1)},\EE) $ is a conservative isofibration (Proposition \ref{prop:hom_from_free_monoid1}), the maps $F^{(n)}_\varphi\to F^{(n-1)}_\varphi$ are conservative isofibrations as well. Using that $F^{(0)}_\varphi\cong \Delta^0$,  we obtain by induction that $F^{(n)}_\varphi$ is a Kan complex and $F^{(n)}_\varphi\to F^{(n-1)}_\varphi$ is a Kan fibration for any $n$ \cite[Prop.2.1.20]{land2021introduction}. Hence, it is sufficient to prove that $F^{(n)}_\varphi$ is contractible. 

Let us prove that $F^{(n)}_\varphi$ is contractible by induction. The Kan complex $F^{(0)}_\varphi\cong \Delta^0$ is obviously contractible. Assume that $n\geq 1,$ and the Kan complex $F^{(n-1)}_\varphi$ is contractible.  
It is sufficient to prove that for any point $\theta$ of $F^{(n-1)}_\varphi$ the fiber $G^{(n)}_{\varphi,\theta}$ over $\theta$ of the map $F^{(n)}_\varphi\to F^{(n-1)}_\varphi$ is contractible. 
Since the composite square \eqref{eq:three_pullbacks} is a pullback, the fiber $G^{(n)}_{\varphi,\theta}$ coincides with the fiber of the map $\prod_{X_n^\circ} \EE^{\Delta^n}\to \prod_{X_n^\circ} \EE^{\partial \Delta^n}$ over the image of $\varphi$. Denote by $\varphi_x \in \EE^{\partial \Delta^n}, x\in X_n^\circ$ the component of the image of $\varphi$ in the product $ \prod_{X_n^\circ} \EE^{\partial \Delta^n}$. Then it is sufficient to prove that the fiber of the map $\EE^{\Delta^n}\to \EE^{\partial \Delta^n}$ over $\varphi_x$ is contractible for any $x\in X_n^\circ$.  

Since $\varphi_0(d^n_0(x))=t$ for any $x\in X_n^\circ$, the composition $\Delta^{\{n\}} \to \Delta^n \overset{\varphi_x}\to \EE$ is defined by the terminal object $t$. Therefore, the fiber of $\EE^{\Delta^n}\to \EE^{\partial \Delta^n}$ over $\varphi_x$ is a contractile Kan complex 
(Lemma 
\ref{lemma:fiber:terminal} in the appendix).
\end{proof}

\section{\bf \texorpdfstring{$\infty$}{}-monoids}

In this section, we introduce the notion of an $\infty$-monoid in a strict monoidal $\infty$-category $\EE$. We study the forgetful functors from the category of $\infty$-monoids in $\EE$ to the original category $\EE$ and to the category of coaugmented objects $\EE^{\1/}$. The two main results of this section state that a terminal object of $\EE$ (respectively, of $\EE^{\1/}$) admits an essentially unique structure of an $\infty$-monoid (unique up to a contractible space of choices) and that this structure defines a terminal object in the category of $\infty$-monoids (Theorem \ref{th:Structures_of_monad_on_terminal}, Theorem \ref{th:terminal:coaug}). The main example of a strict $\infty$-category for us is the $\infty$-category of endofunctors $\End(\CC)$, and its $\infty$-subcategories closed under composition. The main examples of $\infty$-monoids are $\infty$-monads, which are defined as $\infty$-monoids in $\End(\CC)$. 
Later we use the results of this section to show that the codensity $\infty$-monad, which is initially just a functor, has a canonical structure of an $\infty$-monad defined uniquely up to contractible space of choices. 

Our approach to $\infty$-monoids is essentially a reformulation of the approach of Riehl and Verity \cite{riehl2016homotopy}, rewritten in terms that are more convenient for our purposes. Essentially, a strict monoidal $\infty$-category is a quasi-categorically enriched category with one object. Then an $\infty$-monoid is a homotopy coherent monad in the sense of Riehl and Verity \cite[Def. 6.1.1]{riehl2016homotopy}. The equivalence of this approach with that of Lurie \cite{luriehigheralgebra} was established by Haugseng \cite[\S 8]{haugseng2020lax}.

\subsection{\texorpdfstring{$\infty$}{}-monoids}
Consider the augmented simplex category $\Delta_+,$ whose objects are finite ordinals $[n]=\{0,\dots , n\}$,
including the empty ordinal $[-1]=\emptyset,$
and morphisms are monotone maps. We also treat the well ordered set $[n]$ as a posetal  category, and $\Delta_+$ as a full subcategory of the category of small categories ${\sf Cat}.$ Using the identification 
\begin{equation}
[n] \star [m] \cong [n+m+1],
\end{equation}
we obtain that the join defines a structure of strict monoidal category
\begin{equation}
  \star :\Delta_+ \times \Delta_+ \longrightarrow \Delta_+  
\end{equation}
where $[-1]$ is the unit object. 
This strict monoidal structure defines a structure of a strict monoidal $\infty$-category on its nerve $\Ner(\Delta_+)$.

\begin{lemma}\label{lemma:Delta_is_free} 
The simplicial monoid $\Ner(\Delta_+)$ is free, with the basis  given by functors 
 $f:[n]\to \Delta_+$ such that $f(n)=[0].$ Moreover, $\Ner(\Delta_+)$ is left anodyne free with horn generators given by functors $f:[n]\to \Delta_+$ which are non-degenerate, and satisfy:  $f(n)=[0]$ and $f(0)=[-1]$; and all the horn indices are equal to $0$.   
\end{lemma}
\begin{proof}  First we prove that $\Ner(\Delta_+)_n$ is a free monoid generated by functors $f:[n]\to \Delta_+$ such that $f(n)=[0]$. 
For $n=0$ we have $\Ner(\Delta_+)_0=\NN,$ where $\NN$ is the monoid of natural numbers, and the statement is obvious. Assume that $n\geq 1.$ Then a functor $f:[n]\to \Delta_+$ is defined by a sequence of ordinals and monotone maps 
\begin{equation}
f(0) \to f(1) \to \dots \to f(n).
\end{equation}
Now we just take all one-point subsets of the last set $f(n)$ and decompose $f$ into their preimages. For each $k\in f(n)$ and $0\leq i\leq n,$  we denote by $A_{i,k}\subseteq f(i)$ the preimage of $k$ with respect to the map $f(i)\to f(n).$ 
Then $f(i)=\coprod_{k\in f(n)} A_{i,k}$ and a map $f(i)\to f(i+1)$ sends $A_{i,k}$ to $A_{i+1,k}.$ 
Therefore $f=\bigstar_{k\in f(n)} f_{k},$ where $f_k(i)\cong A_{i,k},$ and it is easy to see that this decomposition is unique. Hence, the functors $f:[n]\to \Delta_+$ such that $f(n)=[0]$ form a basis of $\Ner(\Delta_+)$. Let us denote it by $(X_n).$ 

The non-degenerate part of the basis $X_n^\circ$ is given by non-degenerate functors $f:[n]\to \Delta_+$ such that $f(n)=[0].$ Consider a subset $H_n\subseteq X_n^\circ$ consisting of functors $f$ such that $f(0)=[-1].$ Note that $H_0=\emptyset$ and for $n\geq 1$ and $f\in H_n$, since $f$ is non-degenerate, we have $f(1)\neq [-1].$ Using this, it is easy to check that $d_0$ defines a bijection $H_{n+1} \to X_n^\circ\setminus H_n.$ Therefore, $H_n$ is a collection of horn generators  with all horn indices equal to $0$.   
\end{proof}

Let $\EE$ be a strict monoidal $\infty$-category. An $\infty$-monoid in $\EE$ is defined as a morphism of simplicial monoids
\begin{equation}\label{eq:def:mon}
\MM : \Ner(\Delta_+) \longrightarrow \EE.
\end{equation}
The underlying object is denoted by $\MM_0=\MM([0]).$
The simplicial set of $\infty$-monoids in $\EE$ is defined as the simplicial mapping space in the category of simplicial monoids 
\begin{equation}
\mon_\infty(\EE) = \Map( \Ner(\Delta_+), \EE).
\end{equation}
Since $\Ner(\Delta_+)$ is a free simplicial monoid, 
Proposition   \ref{prop:hom_from_free_monoid1} implies that $\mon_\infty(\EE)$ is an $\infty$-category.

Since $\Ner(\Delta_+)_0=\mathbb{N}$ is the monoid of natural numbers, we have a map of simplicial monoids 
\begin{equation}
\mathbb{N} \longrightarrow \Ner(\Delta_+),   
\end{equation}
which is the identity on $0$-simplices, where $\mathbb{N}$ is treated as a constant simplicial monoid. Since $\Map(\mathbb{N},\EE)\cong \EE,$  we obtain a functor  
$\FF_\EE : \mon_\infty(\EE) \to \EE$ that we call the  forgetful functor. For an object $x$ of $\EE$, we define the $\infty$-category of $\infty$-monoid structures on $x$ as  the fiber  
\begin{equation}
\begin{tikzcd}
\mon_\infty(\EE,x) \ar[r] \ar[d] & \mon_\infty(\EE) \ar[d]\\
\Delta^0\ar[r,"x"] & \EE.
\end{tikzcd}
\end{equation}

\begin{proposition}\label{prop:structures_of_monads_is_Kan} For a strict monoidal $\infty$-category $\EE,$ the forgetful functor 
\begin{equation}\label{eq:forg_funct}
\FF_\EE : \mon_\infty(\EE) \longrightarrow \EE
\end{equation}
is a conservative isofibration of $\infty$-categories. 
In particular, for any object $x$ of $\EE$, the simplicial set $\mon_\infty(\EE,x)$ is a Kan complex.     
\end{proposition}
\begin{proof} Lemma \ref{lemma:Delta_is_free} implies that  $\NN \to \Ner(\Delta_+)$ is a free map  inducing an isomorphism on $0$-simplices. Hence Proposition  \ref{prop:hom_from_free_monoid1} implies that the forgetful functor 
$\mon_\infty(\EE) \longrightarrow \EE$
is a conservative isofibration.
Since conservative isofibrations are stable under pullbacks, we obtain that  $\mon_\infty(\EE,x)\to \Delta^0$ is also a conservative isofibration. Therefore $\mon_\infty(\EE,x)$ is a Kan complex. 
\end{proof}

The following theorem is a basis of what follows. It ensures the uniqueness of the $\infty$-monad structure on the codensity $\infty$-monad.

\begin{theorem}\label{th:Structures_of_monad_on_terminal} Assume that $\EE$ is a strict $\infty$-category with a terminal object $t$. Then any structure of an  $\infty$-monoid on $t$ defines a terminal object of $\mon_\infty(\EE)$ and the Kan complex of $\infty$-monoid structures   $\mon_\infty(\EE,t)$ is (non-empty and) contractible. 
\end{theorem}
\begin{proof}  
The forgetful functor $\mon_\infty(\EE)\to \EE$ is induced by the free map $\alpha:\NN\to \Ner(\Delta_+)$. By Lemma \ref{lemma:Delta_is_free}, the sets $X_n=\{f:[n]\to \Delta_+\mid f(n)=[0], f\not\equiv [0] \}$  form a basis of $\alpha$.  Denote by $\varphi:\NN\to \EE$ the map of simplicial monoids  defined by $\varphi(1)=t,$ the terminal object. Here we identify $1=[0]$. Therefore, for any $f\in X_n,$ we have  $d^n_0(f)=[0]$ and $\varphi(d^n_0(f))=t$. Hence  Lemma  \ref{lemma:contractible_fiber} 
implies that $\mon_\infty(\EE,t)$ is contractible. In particular, $\mon_\infty(\EE,t)$ is non-empty. 

Now assume that $T$ is an object of  $\mon_\infty(\EE,t)$, we prove that $T$ is a terminal object of $\mon_\infty(\EE)$. Take an object $\MM$ of $\mon_\infty(\EE)$ we prove that the mapping space  $\map_{\mon_\infty(\EE)}(\MM,T)$ is contractible. The mapping space is defined as the fiber of the map 
$\mon_\infty(\EE)^{\Delta^1} \to \mon_\infty(\EE)^2$
over $(\MM,T)$. Since we have 
\begin{equation}
\mon_\infty(\EE)^{\Delta^1} \cong \Map(\Delta^1\otimes \Ner(\Delta_+),\EE),     
\end{equation}
and 
\begin{equation}
 \mon_\infty(\EE)^2 \cong \Map( \partial \Delta^1\otimes \Ner(\Delta_+) ,\EE),   
\end{equation}
we need to prove that the fiber of the map 
\begin{equation}
\Map(\Delta^1\otimes \Ner(\Delta_+),\EE) \to \Map(\partial \Delta^1\otimes \Ner(\Delta_+),\EE)
\end{equation}
over $(\MM,T)$ is contractible. Therefore, it is sufficient to check the assumptions of Lemma  \ref{lemma:contractible_fiber}. 

The $n$-th component of the basis of $\Delta^1\otimes \Ner(\Delta_+)$ consists of pairs $(\alpha,f),$ where $\alpha:[n]\to [1]$ and $f\in X_n.$ The basis of $\partial \Delta^1\otimes \Ner(\Delta_+)$ consists of such pairs $(\alpha,f),$ where $\alpha \equiv 0$ or $\alpha\equiv 1.$ Therefore, the basis of the inclusion $\partial\Delta^1\otimes \Ner(\Delta_+) \hookrightarrow \Delta^1\otimes \Ner(\Delta_+)$ is defined by pairs $(\alpha,f)$ such that $\alpha:[n]\to [1]$ is surjective. It follows that $d_0^n(\alpha,f)=(1,[0]).$ Therefore $(\MM,T)(d_0^n(\alpha,f))=T([0])=t$ and the conditions of Lemma  \ref{lemma:contractible_fiber} are satisfied.   
\end{proof}

\begin{proposition}\label{prop:monoids_left_fibr}
Let $\varphi:\EE\to \EE'$ be a morphism of strict monoidal $\infty$-categories such that $U\varphi$ is a left fibration. Then the induced map 
\begin{equation}
\mon_\infty(\EE) \longrightarrow \mon_\infty(\EE')
\end{equation}
is a trivial fibration. 
\end{proposition}
\begin{proof} By Lemma \ref{lemma:Delta_is_free} we know that $\Ner(\Delta_+)$ is a left anodyne simplicial monoid. Then  Corollary \ref{cor:left_anodyne_free_monoid} implies that $\Map(\Ner(\Delta_+), \EE) \to \Map(\Ner(\Delta_+),\EE')$ is a trivial fibration. The statement follows. 
\end{proof}

\subsection{Coaugmented objects} 

A coaugmentation of an object $x$ in a strict monoidal $\infty$-category $\EE$ is a morphism $\1  \to x$ from the unit. Each $\infty$-monoid has a natural  coaugmentation (the unit map). In this subsection we study the forgetful functor from the $\infty$-category of $\infty$-monoids to the $\infty$-category of coaugmented objects. 

Let $\EE$ be a strict monoidal $\infty$-category with the unit $\1$. The fat slice $\infty$-category $\EE^{\1/}$ is called the $\infty$-category of coaugmented objects. It can be defined as a pointed cotensoring with $\Delta^1$
\begin{equation}
\EE^{\1/} = \map_*(\Delta^1,\EE),
\end{equation}
where $\Delta^1$ is treated as a pointed simplicial set with the basepoint $0.$ In particular, we see that $\EE^{\1/}$ is also a strict monoidal $\infty$-category. It turns  out that the $\infty$-category of $\infty$-monoids in $\EE^{\1/}$ is equivalent to the $\infty$-category of $\infty$-monoids in $\EE.$ 

\begin{proposition}\label{prop:monoids_of_coaug}
The map 
\begin{equation}
\mon_\infty(\EE^{\1/}) \longrightarrow \mon_\infty(\EE)
\end{equation}
induced by the map  $\EE^{\1/}\to \EE$ 
is a trivial fibration.
\end{proposition}
\begin{proof}
It follows from Proposition \ref{prop:monoids_left_fibr} and the fact that $\EE^{\1/}\to \EE$ is a left fibration \cite[Lemma 2.5.24]{land2021introduction}. 
\end{proof}

We define the trivial $\infty$-monoid in $\EE$ as a constant map $\tilde \1:\Ner(\Delta_+)\to \EE$ defined by $\1$. 

\begin{corollary}\label{cor:trivial_monoid_is_initial}
The trivial monoid $\tilde \1$ is the initial object of $\mon_\infty(\EE)$. 
\end{corollary}
\begin{proof}
We need to prove that the map $\mon_\infty(\EE)^{\1/} \to \mon_\infty(\EE)$ is a trivial fibration. Using the isomorphisms   \eqref{eq:adjunction_pointed}, we obtain  
\begin{align}
\mon_\infty(\EE)^{\tilde \1/}&= \map_*(\Delta^1,\Map(\Ner(\Delta_+),\EE))  \\
& \cong \Map(\Ner(\Delta_+),\map_*(\Delta^1,\EE)) \\
& = \mon_\infty(\EE^{\1/}).
\end{align}
Then the result follows from Proposition \ref{prop:monoids_of_coaug}. 
\end{proof}

We denote by $\Delta_+^{\sf inj}$ the wide subcategory of $\Delta_+$ whose morphisms are injective morphisms of $\Delta_+.$  Then $\Delta_+^{\sf inj}$ is a strict monoidal subcategory of $\Delta_+.$ Therefore its nerve is a simplicial submonoid of $\Ner(\Delta_+)$
\begin{equation}
\Ner(\Delta_+^{\sf inj}) \subseteq \Ner(\Delta_+).
\end{equation}

\begin{lemma}\label{lemma:Delta^inj}
The simplicial monoid $\Ner(\Delta^{\sf inj}_+)$ is free with the basis consisting of functors $f_{k,n}:[n] \to \Delta^{\sf inj}_+, 0\leq k\leq n$ defined by the  formula 
\begin{equation}\label{eq:f_{k,n}}
f_{k,n}(i) = \begin{cases}
[-1], & 0\leq i< k,\\
[0], &  k\leq i\leq n.
\end{cases}
\end{equation}
Moreover, the map $\Delta^1 \to \Ner(\Delta^{\sf inj}_+)$ corresponding to the $1$-simplex $f_{1,1}$ defines an isomorphism of simplicial monoids 
\begin{equation} \label{eq:F(Delta^1)}
\Delta^1 \sm \NN  \cong \Ner(\Delta^{\sf inj}_+),
\end{equation}
where $\Delta^1$ is treated as a pointed simplicial with the basepoint $0$. 
\end{lemma}
\begin{proof}
The description of the basis of $\Ner(\Delta^{\sf inj}_+)$ follows from the description of the basis of $\Ner(\Delta_+)$ (Lemma \ref{lemma:Delta_is_free}). Since $d_1(f_{1,1})=0,$ the $1$-simplex $f_{1,1}$ defines a morphism of pointed simplicial sets $\Delta^1 \to U\Ner(\Delta^{\sf inj}_+)$. By adjunction, we obtain a morphism of simplicial monoids $\Delta^1\sm \NN \to \Ner(\Delta^{\sf inj}_+).$ It is easy to check that this morphism defines a bijection on the bases.  
\end{proof}

\begin{lemma}\label{lemma:coaug-Delta^inj}
The isomorphism \eqref{eq:F(Delta^1)} induces an isomorphism  
\begin{equation}
  \EE^{\1/} \cong \Map(\Ner(\Delta^{\sf inj}_+),\EE). 
\end{equation}
\end{lemma}
\begin{proof}
It follows from the isomorphisms $\EE^{\1/} \cong  \map_*(\Delta^1,\EE)$ and    $\Map(\Delta^1\sm \NN,\EE)\cong  \map_*(\Delta^1,\EE)$, and Lemma \ref{lemma:Delta^inj}.
\end{proof}

\begin{proposition}\label{prop:F'-cons-isofib}
For a strict monoidal $\infty$-category $\EE$, the inclusion  $\Ner(\Delta^{\sf inj}_+) \to \Ner(\Delta_+)$ defines a conservative isofibration 
\begin{equation}\label{eq:forg_to_coaug}
\FF'_\EE :\mon_\infty(\EE) \longrightarrow \EE^{\1/}.
\end{equation}  
\end{proposition}
\begin{proof}
Combining Lemmas  \ref{lemma:Delta^inj} and \ref{lemma:Delta_is_free}, we see that the map $\Ner(\Delta^{\sf inj}_+) \to \Ner(\Delta_+)$ is a free map inducing an isomorphism on monoids of $0$-simplices. Then the statement follows from Proposition \ref{prop:hom_from_free_monoid1}. 
\end{proof}

For an object $x$ of $\EE$, the mapping space $\map_\EE(\1,x)$ is a fiber of the map $\EE^{\1/}\to \EE$ over $x$. Therefore, we can restrict the forgetful functor \eqref{eq:forg_to_coaug} to a map $\mon_\infty(\EE,x)\to \map_{\EE}(\1,x)$ of Kan  complexes. For a coaugmentation $\eta: \1 \to x$, we denote by $\mon_{\infty}(\EE,x,\eta)$ the fiber of this map over $\eta$ 
\begin{equation}
\mon_\infty(\EE,x,\eta) \longrightarrow \mon_\infty(\EE,x) \longrightarrow \map_\EE(\1,x). 
\end{equation}
In object of $\mon_\infty(\EE,x,\eta)$ is called an $\infty$-monad structure on the coaugmented object of $(x,\eta)$. 

\begin{corollary}\label{cor:Kan_fib_mon_coaug}
For an object $x$ of a strict monoidal $\infty$-category $\EE,$ the forgetful map 
\begin{equation}
 \mon_\infty(\EE,x) \longrightarrow \map_{\EE}(\1,x)    
\end{equation}
is a Kan fibration of Kan complexes. In particular, for any $\eta:\1\to x$, the fiber  $\mon_\infty(\EE,x,\eta)$ is a Kan complex. 
\end{corollary}
\begin{proof} Note that the square 
\begin{equation}
\begin{tikzcd}
 \mon_\infty(\EE,x) \ar[r] \ar[d] & \mon_\infty(\EE) \ar[d] \\ 
\map_{\EE}(\1,x) \ar[r] & \EE^{\1/}    
\end{tikzcd}
\end{equation}
is a pullback. Hence, the statement follows from the fact that isofibrations are stable under pullbacks \cite[Cor.2.1.22]{land2021introduction} and the fact that an isofibration between Kan complexes is a Kan fibration \cite[Prop.2.1.20]{land2021introduction}. 
\end{proof}  

For an $\infty$-category $\CC$, we denote by $\term(\CC)$ the full subcategory of its terminal objects.  

\begin{lemma}\label{lemma:term(C)}
Let $\FF:\CC\to \DD$ be a conservative isofibration of $\infty$-categories, and assume that there exists a terminal object $t$ of $\CC$  such that $\FF(t)$ is a terminal object of $\DD$. Then an object of $\CC$ is terminal if and only if its image in $\DD$ is terminal and $\FF$ restricts to a trivial fibration of contractible Kan complexes
\begin{equation}
\term(\CC) \longrightarrow \term(\DD).
\end{equation}
In particular, the fiber over any terminal object of $\DD$ is a contractible Kan complex.
\end{lemma}
\begin{proof}
Take a terminal object $s'$ of $\DD$ and any object $t'$ of $\CC$ such that $\FF(t')=s'$. Since $t$ is terminal, there is a morphism $t'\to t$. Its image $s'\to \FF(t)$ is a morphism between terminal objects, so, it is an equivalence. Since $\FF$ is conservative, $t'\to t$ is also an equivalence. Therefore, $t'$ is terminal. So a preimage of any terminal object is terminal. On the other hand, the images of equivalent objects are equivalent, and hence, the image of any terminal object of $\CC$ is a terminal object of $\DD$. It follows that the map $\term(\CC)\to \term(\DD)$ is well defined, and there is a pullback square.
\begin{equation}
\begin{tikzcd}
\term(\CC)\ar[r] \ar[d]  & \CC \ar[d,"\FF"] \\
\term(\DD) \ar[r] & \DD
\end{tikzcd}
\end{equation}
Since $\FF$ is an isofibration, the functor $\term(\CC)\to \term(\DD)$ is also an isofibration \cite[Cor.2.1.22]{land2021introduction}. An isofibration between Kan complexes is a Kan fibration. Since both of the Kan complexes are contractible, the Kan fibration is a trivial Kan fibration. 
\end{proof}

\begin{theorem}\label{th:terminal:coaug}
Let $\EE$ be a strict monoidal $\infty$-category and $(x,\eta_t)$ be a terminal object of the $\infty$-category of coaugmented objects $\EE^{\1/}$. Then any  $\infty$-monoid structure on the coaugmented object $(x,\eta_t)$ defines a terminal object of $\mon_\infty(\EE)$ and the Kan complex of  $\infty$-monoid structures  $\mon_\infty(\EE,x,\eta_t)$ is (non-empty and) contractible.   
\end{theorem}
\begin{proof}
Combining Lemma \ref{lemma:term(C)} and Proposition \ref{prop:F'-cons-isofib} we see that it is sufficient to prove that there is a terminal object  of $\mon_\infty(\EE)$ whose image in $\EE^{\1/}$ is also terminal.

First, let us make some general remarks about a general $\infty$-category $\CC$ and its object $c$. First of all the functor $\varphi_{\CC} : \CC^{c/}\to \CC$ is a left fibration \cite[Lemma 2.5.24]{land2021introduction}, and hence conservative \cite[Prop. 2.1.3]{land2021introduction}. Therefore, if we have a morphism $\alpha : (\eta:c\to x) \to (\eta':c\to x')$  in $\CC^{c/}$ that induces an equivalence $x\to x'$ in $\CC$, then $\alpha$ is an equivalence. Using this one can show that if two morphisms $\eta,\eta':c\to x$ are equivalent morphisms in $\CC$, then they are equivalent  objects in $\CC^{c/}$.  

The identity map ${\sf id}_c$ is an initial object of the slice category $\CC^{c/}$   \cite[Exercise 135]{land2021introduction}.  Consider the double slice category $(\CC^{c/})^{{\sf id}_c/}.$ There are two different functors
\begin{equation}
\varphi_{\CC^{c/}}, (\varphi_{\CC})_* : (\CC^{c/})^{{\sf id}_c/} \longrightarrow \CC^{c/}. 
\end{equation}
We claim that the images $\varphi_{\CC^{c/}}(x)$ and $(\varphi_{\CC})_*(x)$ of any object $x$ of $(\CC^{c/})^{{\sf id}_c/}$ are equivalent in $\CC^{c/}$. We can present the category $(\CC^{c/})^{{\sf id}_c/}$ as 
\begin{equation}
(\CC^{\1/})^{{\sf id}_c/}=\map_*(\Delta^1,\map_*(\Delta^1,\CC))\cong \map_*(\Delta^1\sm \Delta^1,\CC). 
\end{equation}
Then the functors $\varphi_{\CC^{c/}}$ and $(\varphi_{\CC})_*$ are induced by the maps  $i_1: \Delta^1 \times \{1\} \to  \Delta^1 \sm \Delta^1$ and $i_2:\{1\}\times \Delta^1 \to \Delta^1 \sm \Delta^1.$ It is easy to see that for any  $x:\Delta^1 \sm \Delta^1 \to \CC$  the morphisms $\varphi_{\CC^{c/}}(x),(\varphi_{\CC})_*(x):x(0)\to x(1,1)$ are equivalent morphisms in $\CC$
\begin{equation}
\begin{tikzcd}
x(0)\ar[r,"{\sf id}"] \ar[d,"{\sf id}"] \ar[rd] & x(0) \ar[d,"\varphi_{\CC^{c/}}(x)"] \\
x(0)\ar[r,"(\varphi_{\CC})_*(x)"'] & x(1,1).
\end{tikzcd} 
\end{equation}
Therefore, $\varphi_{\CC^{c/}}(x)$ and $(\varphi_{\CC})_*(x)$ are equivalent objects in $\CC^{c/}.$ 

Now we assume that $\EE$ is a strict monoidal $\infty$-category. Then $\EE^{\1/}$ is also a strict monoidal category with a the unit object ${\sf id}_\1$ and we have a morphism of strict monoidal $\infty$-categories  
$\varphi_\EE:\EE^{\1/}\to \EE$. 
It induces a functor $\varphi_*:\mon_\infty(\EE^{\1/})\to \mon_\infty(\EE)$. Therefore, we obtain two functors 
\begin{equation}
\FF_{\EE^{\1/}} , (\varphi_\EE)_* \circ \FF'_{\EE^{\1/}}: \mon_\infty(\EE^{\1/}) \longrightarrow \EE^{\1/}. 
\end{equation} 
It is easy to check that  $\FF_{\EE^{\1/}} = \varphi_{\EE^{\1/}} \circ \FF'_{\EE^{\1/}}$. Therefore, we obtain that the images of any object  under these functors are equivalent. 

Denote by $\theta \in \EE^{\1/}$ a terminal object. Theorem \ref{th:Structures_of_monad_on_terminal} says that we can lift $\theta$ to a terminal object $\tilde \theta \in \mon_\infty(\EE^{\1/})$ along $\FF_{\EE^{\1/}}$. The object  $(\varphi_\EE)_* ( \FF'_{\EE^{\1/}}(\tilde \theta))$ is equivalent to $\theta$, so it is also terminal. Then using the commutative diagram 
\begin{equation}
\begin{tikzcd}
\mon_\infty(\EE^{\1/}) \ar[rr,"\mon_\infty(\varphi_\EE)"] \ar[d,"\FF'_{\EE^{\1/}}"] & & \mon_\infty(\EE) \ar[d,"\FF'_\EE"] \\
(\EE^{\1/})^{{\sf id}_c/}\ar[rr,"(\varphi_\EE)_*"]  & & \EE^{\1/}
\end{tikzcd}
\end{equation}
and the fact that $\mon_\infty(\varphi_\EE)$ is a trivial fibration (Proposition \ref{prop:monoids_of_coaug}), we obtain that $\mon_\infty(\varphi_\EE)(\tilde \theta)$ is a terminal object of $\mon_\infty(\EE)$ whose image in $\EE^{\1/}$  is terminal. 
\end{proof}

\subsection{Monoids and monads in simplicial categories} 
\label{subsection:monoida_and_monads_in _simplicial_cats}

In this subsection we show how to pass from monoids and monads in simplicial categories to $\infty$-monoids and $\infty$-monads in $\infty$-categories. First, we recall the definition of the internal hom in the category of simplicial categories ${\sf Cat}_\Delta$ and the fact that the homotopy coherent nerve $\Ner:{\sf Cat}_\Delta \to \sSet$ is a closed monoidal functor.

The category of simplicial categories ${\sf Cat}_\Delta$ is a cartesian closed category. The internal hom $\Fun_\Delta(\CCC,\CCC')$ is a simplicial category, whose objects are simplicial functors from $\CCC$ to $\CCC'$, and the simplicial set of natural transformations $\Nat({\bf F},{\bf G})$  is defined as the enriched end of the mapping spaces $\int_{c\in \CCC} \map_{\CCC'}({\bf F}(c),{\bf G}(c))$ (see \cite[\S 2.2]{kelly1982basic},  \cite[Prop.6.6.9]{borceux1994handbook}, \cite[\S 7.3]{riehl2014categorical}).  Then we have a natural adjunction  
\begin{equation}
\Fun_\Delta(\CCC \times \CCC', \CCC'') \cong  \Fun_\Delta(\CCC,\Fun_{\Delta}(\CCC',\CCC''))
\end{equation} 
(see \cite[\S 2.3]{kelly1982basic}). Moreover, for any tree simplicial categories $\CCC,\CCC',\CCC''$, there is a natural associative composition functor 
\begin{equation}
\label{eq:comp_simp}
\circ : \Fun_\Delta (\CCC',\CCC'') \times  \Fun_\Delta(\CCC,\CCC') \longrightarrow \Fun_\Delta(\CCC,\CCC'')
\end{equation}
(see \cite[\S II.3]{eilenberg1966closed}). Since Kan complexes are closed under limits, they are also closed under ends. Therefore,  if $\CCC'$ is locally Kan, then $\Fun_\Delta(\CCC,\CCC')$ is also locally Kan.

The homotopy coherent nerve 
\begin{equation}
\Ner : \Cat_\Delta \longrightarrow \sSet
\end{equation}
is a monoidal functor, where both categories are considered as monoidal categories with respect to the product. Therefore, it is a closed monoidal functor  \cite[\S II, Prop. 4.3]{eilenberg1966closed}. It follows that, for any two simplicial categories $\CCC,\CCC'$, the map 
\begin{equation}\label{eq:Ner_CC'}
\Ner_{\CCC,\CCC'} : \Ner( \Fun_\Delta(\CCC,\CCC') ) \longrightarrow \map(\Ner(\CCC),\Ner(\CCC'))
\end{equation}
induced by the nerve of the map $\CCC\times \Fun_\Delta(\CCC,\CCC') \to \CCC',$ respects the composition functor \eqref{eq:comp_simp}. 

For simplicial categories $\CCC$ and $\CCC'$ equipped with some chosen base-objects $c_0$ and $c_0',$ we define the simplicial category of simplicial functors preserving the base-object $\Fun_\Delta^*(\CCC,\CCC')$ as the equaliser of two simplicial functors
\begin{equation}
\Fun_\Delta^*(\CCC,\CCC') \to \Fun_\Delta(\CCC,\CCC') \rightrightarrows  \CCC',
\end{equation} 
where the first one is defined by the evaluation at $c_0$, and the second one is the constant functor defined by the object $c_0'.$ Using that $\Ner:\Cat_\Delta \to \sSet$ commutes with equalizers, we obtain that $\Ner_{\CCC,\CCC'}$ induces a map 
\begin{equation}
\Ner^*_{\CCC,\CCC'} : \Fun_\Delta^*(\CCC,\CCC') \longrightarrow \map_*(\Ner(\CCC),\Ner(\CCC')).
\end{equation}

A strict monoidal simplicial category $\EEE$ is a simplicial category equipped with a simplicial functor of product and the unit object
\begin{equation}
\otimes : \EEE\times \EEE \to \EEE, \hspace{1cm} \1\in {\sf Ob}(\EEE)
\end{equation}
satisfying the usual strict associativity and unit axioms. Then the underlying ordinary category $\EEE_0$ is an ordinary strict monoidal category. Monoidal functors of strict monoidal simplicial categories are defined in the obvious way. The simplicial category of monoidal functors $\Fun^\otimes_\Delta(\EEE,\EEE')$ 
 is 
defined as the equaliser 
\begin{equation}
\Fun^\otimes_\Delta(\EEE,\EEE') \to \Fun^*_\Delta(\EEE,\EEE') \rightrightarrows \Fun^*_\Delta(\EEE^2,\EEE'),
\end{equation}
where the first map is induced by the map $\otimes : \EEE^2\to \EEE,$ and the second map is defined as the composition of the square map $\Fun^*_\Delta(\EEE,\EEE')\to \Fun^*_\Delta(\EEE^2,(\EEE')^2)$ and the map induced by $\otimes':(\EEE')^2 \to \EEE'$ (compare with Lemma \ref{lemma:Map_eq}). Here, we assume that the base-objects of $\EEE$ and $\EEE'$ are the unit objects $\1$ and $\1'$. Note that, if $\EEE'$ is locally Kan, then $\Fun^\otimes_\Delta(\EEE,\EEE')$ is also locally Kan. 

The homotopy coherent nerve $\Ner(\EEE)$ of a strict monoidal simplicial category $\EEE$ has a natural structure of a simplicial monoid defined by the map $\Ner(\EEE) \times \Ner(\EEE) \cong \Ner(\EEE\times \EEE) \to  \Ner(\EEE).$ 

\begin{proposition}
For any strict monoidal simplicial categories $\EEE$ and $\EEE'$, the map $\Ner_{\EEE,\EEE'}$ induces a map 
\begin{equation}
\Ner^\otimes_{\EEE,\EEE'} : \Ner(\Fun_\Delta^\otimes(\EEE,\EEE')) \longrightarrow \Map(\Ner(\EEE),\Ner(\EEE')), 
\end{equation} 
where $\Map(-,=)$ denotes the mapping space in the category of simplicial monoids. 
\end{proposition}
\begin{proof}
Consider the morphism of parallel pairs, 
\begin{equation}
\begin{tikzcd}[column sep = 20mm]
\Ner(\Fun^*_\Delta(\EEE,\EEE')) 
\ar[r,"\Ner^*_{\EEE,\EEE'}"]
\ar[d,shift right =2mm]
\ar[d,shift left =2mm]
& 
\map_*(\Ner(\EEE),\Ner(\EEE'))
\ar[d,shift right = 2mm]
\ar[d,shift left =2mm]
\\
\Ner(\Fun^*_\Delta(\EEE^2,\EEE'))
\ar[r,"\Ner^*_{\EEE^2,\EEE}"]
& 
\map_*(\Ner(\EEE)^2,\Ner(\EEE'))
\end{tikzcd}
\end{equation}
where in each pair  the first map is induced by the map $\otimes: \EEE^2\to \EEE$, and the second map is the composition of the square map and the map induced by $\otimes' : (\EEE')^2\to \EEE'$. 
The equalizer of the left-hand pair of vertical maps is $\Ner(\Fun^\otimes_\Delta(\EEE,\EEE'))$. By Lemma \ref{lemma:Map_eq}, the equalizer of the right-hand pair of vertical maps is $\Map(\Ner(\EEE),\Ner(\EEE'))$. 
Then the map induced on the equalizers is the required map $\Ner^\otimes_{\EEE,\EEE'}$.  
\end{proof}

By a monoid $M$ in a strict monoidal simplicial category $\EEE$ we mean a monoid in $\EEE_0$. The cobar construction of $M$ defines a strict monoidal functor $\tilde M: \Delta_+ \to \EEE_0.$ If we treat $\Delta_+$ as a simplicial category with discrete hom-sets, we obtain a morphism of strict monoidal simplicial categories
\begin{equation}
\tilde M : \Delta_+ \longrightarrow \EEE.
\end{equation}
This motivates the following definition of the simplicial category of monoids in $\EEE$
\begin{equation}
\mon_\Delta(\EEE) = \Fun_\Delta^\otimes(\Delta_+,\EEE). 
\end{equation} 
Objects of this simplicial category are identified with monoids in $\EEE$.

\begin{corollary}
For a strict monoidal locally Kan simplicial category $\EEE$, there is a  functor of $\infty$-categories 
\begin{equation}
\Ner^\otimes_{\Delta_+,\EEE} :\Ner(\mon_\Delta(\EEE)) \longrightarrow \mon_\infty(\Ner(\EEE)). 
\end{equation}
that sends a monoid $M$ in $\EEE$ to the $\infty$-monoid $\Ner(\tilde M)$ in $\Ner(\EEE).$
\end{corollary}

A monad $M$ on a simplicial category $\CCC$ is defined as a monoid in the strict monoidal simplicial category $\End_\Delta(\CCC)$. The simplicial category of monads is defined as follows 
\begin{equation}
\Mon_\Delta(\CCC) = \mon_\Delta(\End_\Delta(\CCC)).
\end{equation}
For an $\infty$-category $\CC$, we denote the $\infty$-category of $\infty$-monads by 
\begin{equation}
\Mon_\infty(\CC) = \mon_\infty(\End(\CC)).
\end{equation}

\begin{corollary} Let $\CCC$ be a locally Kan simplicial category. Then the composition of $\Ner^\otimes_{\Delta_+,\End(\CCC)}$ and  
$\mon_\infty(
\Ner_{\CCC,\CCC})
$
is a functor 
\begin{equation}
\Ner(\Mon_\Delta(\CCC)) \longrightarrow \Mon_\infty(\Ner(\CCC)),
\end{equation}
that sends a monad $M$ on $\CCC$ to the $\infty$-monad on $\Ner(\CCC)$ defined by the composition 
\begin{equation}
\Ner(\Delta_+) \xrightarrow{ \Ner(\tilde M)} \Ner(\End_\Delta(\CCC)) \xrightarrow{\ \Ner_{\CCC,\CCC}\ } \End(\Ner(\CCC)).
\end{equation}
\end{corollary}

\section{\bf Actions of \texorpdfstring{$\infty$}{}-monoids} 
\label{section:actions}

The goal of this subsection is to introduce the notion of an $\infty$-monoid acting on an object, and an algebra over an $\infty$-monad, and prove some of their properties. 

\subsection{Actions of monoids in ordinary categories}
\label{subsec:actions}
First, we give the definition in the setting of ordinary categories and reformulate it in a form suitable for generalization.

Let $E = (E, \otimes, \1)$ be an ordinary strict monoidal category, and let $C$ be an ordinary category.
A strict $E$-action on $C$ is a functor
\begin{equation} \oslash : E \times C \to C \end{equation}
such that the functor $\tilde \oslash : E \to \End(C)$ obtained by the adjunction is a strict monoidal functor.
A strict action category is a pair $(E, C)$, where $E$ is a strict monoidal category and $C$ is a category equipped with a strict $E$-action.
A morphism of strict action categories $(E, C) \to (E', C')$ is a pair $(\varphi, \psi)$, where $\varphi: E \to E'$ is a strict monoidal functor and $\psi: C \to C'$ is a functor such that the diagram
\begin{equation}\label{diag:morphism_action}
\begin{tikzcd}
E\times C\ar[rr,"\varphi\times \psi"] \ar[d,"\oslash"] && E' \times C' \ar[d,"\oslash'"] \\
C \ar[rr,"\psi"]  && C' 
\end{tikzcd}
\end{equation}
commutes. 

Let $(E,C)$ be a strict action category,   $M = (M, \mu, \eta)$ be a monoid in $E$, and let $x$ be an object of $C$.
An $M$-action on $x$ is a morphism $a: M \oslash x \to x$ such that the following diagrams commute.
\begin{equation}\label{eq:action:axioms}
\begin{tikzcd}
M\otimes M \oslash \ar[rr,"\mu\oslash x"] \ar[d,"M \oslash a "] x && M \oslash x \ar[d,"a"] \\
M \oslash x \ar[rr,"a"] && x
\end{tikzcd}
\hspace{1cm}
\begin{tikzcd}
x \ar[r,"\eta\oslash x"] \ar[rd,equal] & \MM\oslash x\ar[d,"a"] \\ 
& x
\end{tikzcd}
\end{equation}
An object with an $M$-action is called an $M$-action object. A morphism $M$-action objects $f:(x,a)\to (y,b)$ is defined as a map $f:x\to y$ such that $f\circ a=b \circ ({\sf id}\oslash f).$  

\begin{example}[Free $M$-action]
Let $(E,C)$ be a strict action category, $M$ be a monoid in $E$ and $c$ be an object in $C$. Then the $M$-action on $M \oslash c$ defined by the map 
$\mu\oslash c : (M \otimes M) \oslash c \to M \oslash c$ 
is called the free $M$-action.
\end{example}

An action object in a strict action category $(E,C)$ is a pair $(M,x),$ where $M=(M,\mu,\eta)$ is a monoid in $E$ and $x=(x,a)$ is an object with an $M$-action. Note that any strict monoidal functor $\varphi:E\to E'$ sends a monoid $M=(M,\mu,\eta)$ to a monoid $\varphi(M)=(\varphi(M),\varphi(\mu),\varphi(\eta))$. Moreover, any morphism of strict action categories $(\varphi,\psi):(E,C)\to (E',C')$ sends an action object $(M,x)$ to an action object $(\varphi(M),\psi(x)),$ where $\psi(x)=(\psi(x),\psi(a)).$

It is well known that $[0]\in \Delta_+$ is the walking monoid in a strict monoidal category i.e. for any monoid $M$ in a strict monoidal category $E$, there exists a unique strict monoidal functor $ \Delta_+ \to E$, called the cobar construction of $M$, sending the monoid $[0]$ to the monoid $M$  \cite[\S VII.5]{mac1998categories}. Let us define the walking action object. 

Consider a wide subcategory of the simplex category 
\begin{equation}
 \Delta_\max \subseteq \Delta,   
\end{equation}
whose morphisms $f:[n]\to [m]$ are the monotone maps preserving the largest element $f(n)=m$ (see \cite[\href{https://kerodon.net/tag/04S6}{Subsection 04S6}]{kerodon}). The join operation defines a strict $\Delta_+$-action on $\Delta_\max$ 
\begin{equation}\label{eq:Delta-action}
\star : \Delta_+ \times \Delta_\max \longrightarrow \Delta_\max.
\end{equation}
The object $[0]$ of the category $\Delta_+$ has a unique structure of a monoid defined by the unique maps $\mu_{[0]}: [1]\to [0]$ and $\eta_{[0]}: [-1]\to [0]$. Consider the $[0]$-action of the monoid $[0]\in \Delta_+$ on the object $[0]\in \Delta_\max$ defined by the unique map in $\Delta_\max$
\begin{equation}
a_{[0]} : [0]\star [0] \longrightarrow [0].
\end{equation}

\begin{proposition}\label{prop:walking_action}
Let $(E,C)$ be a strict action category and $(M,x)$ be an action object in $(E,C)$. Then there exists a unique morphism of strict action categories 
\begin{equation}
(\Delta_+,\Delta_\max)\longrightarrow (E,C)   
\end{equation}
sending the action object $([0],[0])$ to $(M,x)$.
\end{proposition}

The proof of Proposition \ref{prop:walking_action} is moved to the appendix (Subsection \ref{subsection:walking_monoid}).

\subsection{Actions of \texorpdfstring{$\infty$}{}-monoids}

We now generalize the definition of a strict action category to the setting of $\infty$-categories.

\begin{definition}[Strict action $\infty$-category]
Let $\EE$ be a strict monoidal $\infty$-category and $\CC$ be an $\infty$-category. A strict (left) $\EE$-action on $\CC$ is a functor of $\infty$-categories
\begin{equation}
\oslash : \EE \times \CC \to \CC 
\end{equation}
such that the functor $\EE\to \End(\CC)$ obtained by the adjunction is a morphism of simplicial monoids. In other words, $\oslash : \EE \times \CC \to \CC$ is a functor such that for any $n$ the map $\oslash_n : \EE_n \times \CC_n \to \CC_n$ defines a left action of the monoid $\EE_n$ on the set $\CC_n.$

A strict action $\infty$-category is a pair $(\EE,\CC)$, where $\EE$ is a strict monoidal $\infty$-category and $\CC$ is an $\infty$-category equipped with a strict $\EE$-action. A morphism of strict action $\infty$-categories $(\EE,\CC)\to (\EE',\CC')$ is a pair $(\varphi,\psi),$ where $\varphi:\EE\to \EE'$ is a morphism of simplicial monoids and $\psi:\CC\to \CC'$ is a functor such that the diagram \begin{equation}\label{diag:morphism_action2}
\begin{tikzcd}
\EE\times \CC\ar[rr,"\varphi\times \psi"] \ar[d,"\oslash"] && \EE' \times \CC' \ar[d,"\oslash'"] \\
\CC \ar[rr,"\psi"]  && \CC' 
\end{tikzcd}
\end{equation} 
commutes. 
\end{definition}

\begin{example}
Any strict monoidal $\infty$-category $\EE$ has an action on itself defined by $\oslash=\otimes$, and a morphism of strict monoidal $\infty$-categories $\varphi: \EE \to \EE'$ defines a morphism of strict action categories 
\begin{equation}
(\varphi,\varphi): (\EE,\EE) \to (\EE',\EE').
\end{equation}
Moreover, if $(\EE,\CC)$ is a strict action $\infty$-category, and $c$ is an object of $\CC$,  we have a functor $\cdot \oslash c:\EE\to \CC$ given by the composition $\EE\to \End(\CC) \xrightarrow{{\sf ev}_c} \CC.$ This functor defines a morphism of strict action categories 
\begin{equation}
(\Id_\EE, \cdot \oslash c): (\EE,\EE) \longrightarrow (\EE,\CC). 
\end{equation}
\end{example}

 The strict $\Delta_+$-action \eqref{eq:Delta-action} induces a strict  $\Ner(\Delta_+)$-action 
\begin{equation}
    \Ner(\Delta_+) \times \Ner(\Delta_\max) \longrightarrow \Ner(\Delta_\max).
\end{equation}
Therefore, we obtain that $(\Ner(\Delta_+),\Ner(\Delta_\max))$ is a strict action $\infty$-category. Then Proposition \ref{prop:walking_action} motivates the following definition. 

\begin{definition}[Action of an $\infty$-monoid]\label{def:action}
Let $(\EE,\CC)$ be a strict action $\infty$-category and $\MM:\Ner(\Delta_+)\to \EE$ be an $\infty$-monoid. An $\MM$-action on an object $x$ of $\CC$ is a functor 
\begin{equation}
\AA : \Ner(\Delta_\max) \longrightarrow \CC
\end{equation}
such that $\AA([0])=x$ and $(\MM,\AA):(\Ner(\Delta_+),\Ner(\Delta_\max))\to (\EE,\CC)$ is a morphism of strict action $\infty$-categories.
\end{definition}

\begin{example}[Free action of an $\infty$-monoid]
\label{example:free_action}
For a strict action $\infty$-category $(\EE,\CC)$, an $\infty$-monoid $\MM$ in $\EE$ and an object $c\in \CC$, we can define a structure of $\MM$-action on $\MM_0 \oslash c$ given by the composition 
\begin{equation}
\AA^{\sf free}_{\MM,c} : \Ner(\Delta_\max) \hookrightarrow \Ner(\Delta_+) \xrightarrow{ \ \MM \ } \EE \xrightarrow{\ \cdot \oslash c\ } \CC.
\end{equation}
This $\MM$-action is called free $\MM$-action on $\MM_0\oslash c$. 
The corresponding morphism of strict action $\infty$-categories is defined by the composition
\begin{equation}
(\Ner(\Delta_+),\Ner(\Delta_\max)) \hookrightarrow (\Ner(\Delta_+),\Ner(\Delta_+)) \xrightarrow{(\MM,\MM)} (\EE,\EE) \xrightarrow{(\Id,\cdot \oslash c)} (\EE,\CC).
\end{equation}
\end{example}

\subsection{Splitting of the cosimplicial resolution} 

A coaugmented cosimplicial object $X:\Delta_+\to C$ in an ordinary category $C$ is \emph{right split} if it has extra degeneracy maps $s^n:X^n\to X^{n-1}$ for $n\geq 0$ satisfying the simplicial relations. Dually, it is called \emph{left split} if it has extra degeneracy maps $s^{-1}:X^n\to X^{n-1}$ satisfying the cosimplicial relations. It is easy to see that $X$ is left split if and only if $X^\op$ is right split. So we consider only right split cosimplicial objects. 

Let us reformulate the property of being right split in terms of the category $\Delta_\max$.  The inclusion functor $ \Delta_\max  \hookrightarrow \Delta_+$ admits a left adjoint 
\begin{equation}
{\sf c}_+ : \Delta_+ \longrightarrow \Delta_\max, \hspace{1cm} [n] \mapsto [n] \star [0]  
\end{equation}
\cite[\href{https://kerodon.net/tag/04S8}{Rem.04S8}]{kerodon}. Then a coaugmented cosimplicial object $X:\Delta_+\to C$ is right split if and only if it can be presented as a composition 
\begin{equation}
 \Delta_+ \xrightarrow{{\sf c}_+} \Delta_\max \to C.   
\end{equation}

Assume $(E,C)$ is an ordinary strict action category and $M=(M,\eta,\mu)$ is a monoid in $E$. Then for any object $c$ of $C$, we obtain a coaugmented cosimplicial object $M^\bullet_c : \Delta_+ \to C$ 
\begin{equation} M^\bullet_c: \hspace{1cm}
\begin{tikzcd}[column sep = 12mm]
c \ar[r,"\eta_c"] & 
M \oslash  c 
\ar[r,"\eta_{M\oslash c} ", shift left = 4mm] 
\ar[r,"M \oslash \eta_c",shift left = -5mm ] 
& 
M^{\otimes 2} \oslash c 
\ar[l, "\mu_c"' ] 
& 
\dots 
\end{tikzcd}
\end{equation}
If $(x,a)$ is an $M$-action object, $M^\bullet_x$ is right split 
\begin{equation}
\begin{tikzcd}[column sep = 12mm]
x \ar[r,"\eta_x"]  
& 
M \oslash x 
\ar[l,"a"',shift left = 3mm] 
\ar[r,"\eta_{M\oslash x} ", shift left = 3mm] 
\ar[r,"M\oslash \eta_x",shift left = -5mm ] 
& 
M^{\otimes 2} \oslash x 
\ar[l, "\mu_x"' ] 
\ar[l, "M\oslash a"', shift left =9mm ] 
& 
\dots 
\end{tikzcd}
\end{equation}
In particular, if $x$ admits an $M$-action, $x$ is a retract of $M(x)$. Let us generalize this picture to the setting of $\infty$-categories.

Let $\CC$ be an $\infty$-category. A coaugmented cosimplicial object $X:\Ner(\Delta_+) \to \CC$ is \emph{right split}, if it factors as a composition 
\begin{equation}
\Ner(\Delta_+) \xrightarrow{\Ner({\sf c}_+)} \Ner(\Delta_\max) \longrightarrow \CC. 
\end{equation}

\begin{proposition}[{\cite[\href{https://kerodon.net/tag/04SH}{Prop. 04SH}]{kerodon}}]
\label{prop:split_cosimplicial}
Let $X:\Ner(\Delta_+)\to \CC$ be a right split coaugmented cosimplicial object and let $\bar X:\Ner(\Delta) \to \CC$ be its restriction. Then $X$ is a limit diagram of $\bar X$ and the natural map to the  totalization of $\bar X$
\begin{equation}
X(0) \xrightarrow{\sim} \lim \bar X
\end{equation}
is an equivalence. 
\end{proposition}

Assume that $(\EE,\CC)$ is a strict action $\infty$-category and $\MM$ is an $\infty$-monoid in $\EE$.  For an object $c$ of $\CC$, the composition of $\MM$ with the  functor $ - \oslash  c:\EE \to \CC$  is denoted by 
\begin{equation}
\MM_c : \Ner(\Delta_+) \longrightarrow \CC.
\end{equation} 
Consider the following full subcategories of $\CC:$ 
\begin{itemize}
    \item $\II(\MM)$ is spanned by objects of the form $\MM_0\oslash c$;  
    \item $\AA(\MM)$ is spanned by objects that admit an $\MM$-action; 
    \item $\SS(\MM)$ is spanned by objects $c$ such that the coaugmented cosimplicial object $\MM_c:\Ner(\Delta_+)\to \CC$ is right split; 
    \item $\RR(\MM)$ is spanned by retracts of objects of $\II(\MM)$.
\end{itemize}

Here we say that an object of an $\infty$-category $\CC$ is a retract of another object if it is a retract in the homotopy category $h\CC$.

\begin{lemma}\label{lemma:R(M)} For any object $c$ of  
$\RR(\MM)$, the morphism $\eta_c:c\to \MM_0 \oslash c$ is left invertible. 
\end{lemma}
\begin{proof}
Assume that $c$ is an object of $\RR(\MM)$. Then there is an object $d=\MM \oslash c'$ of and maps $\iota :c \to d$ and $\rho : d \to c$ such that $\rho \circ \iota =\id_c$ in the homotopy category $h\CC$. There is also a map $\mu: \MM_0 \oslash d \to d$, defined as $\mu=\MM_{c'}([1]\to [0]),$ such that 
$\mu \circ \eta_{d} = \id_d$ in the homotopy category. Therefore we have  $ \rho \circ \mu \circ (\MM \oslash \iota) \circ \eta_c = \rho \circ \mu  \circ \eta_d \circ \iota = \rho \circ \iota = \id_c $.    
\end{proof}

\begin{proposition}\label{prop:ISAR}
For a strict action $\infty$-category $(\EE,\CC)$ and an $\infty$-monoid $\MM$ in $\EE$, there are inclusions
\begin{equation}
    \II(\MM) \subseteq \AA(\MM) \subseteq \SS(\MM) \subseteq  \RR(\MM) \subseteq \CC. 
\end{equation}
\end{proposition}
\begin{proof} The inclusion 
$\II(\MM) \subseteq \AA(\MM)$ follows from the fact that any object of the form $\MM_0 \otimes c$ has the free  $\MM$-action.

Prove $\AA(\MM)\subseteq \SS(\MM)$. Assume that $c$ is an object of $\AA(\MM)$ and  $\AA:\Ner(\Delta_\max)\to \CC$ is an $\MM$-action on $c$. Then $c=\AA ([0])$ and the diagram 
\begin{equation}
\begin{tikzcd}
\Ner(\Delta_+) \times \Ner(\Delta_\max) 
\ar[rr,"{(\MM,\AA)}"] \ar[d,"\star"] && \EE \times \CC \ar[d,"\oslash"] \\
\Ner(\Delta_\max) \ar[rr,"{\AA}"] && \CC
\end{tikzcd}
\end{equation}
commutes. 
If we compose the vertical arrows with $\Ner(\Delta_+) \to \Ner(\Delta_+) \times \Ner(\Delta_\max), n \mapsto (n,[0])$ and $\EE \to \EE \times \CC, e \mapsto (e,c)$, we obtain a commutative diagram 
\begin{equation}\label{diag:m+0_1}
\begin{tikzcd}
 \Ner(\Delta_+)  
\ar[rr,"{\MM}"] \ar[d,"{\Ner({\sf c}_+)}"] 
&& \EE  \ar[d,"- \oslash c"] \\
\Ner(\Delta_\max) \ar[rr,"{\AA}"] && \CC
\end{tikzcd}
\end{equation}
Therefore, $\MM_c$ is right split.

Prove $\SS(\MM)\subseteq \RR(\MM).$ Assume that $\MM_c : \Ner(\Delta_+)\to \CC$ can be presented as a composition $\Ner(\Delta_+)\to  \Ner(\Delta_\max) \to \CC.$ Since $[0]$ is a retract of $[1]$ in $\Delta_\max,$ we obtain that $c$ is a retract of $\MM_0 \oslash  c$. 
\end{proof}

\section{\bf Codensity \texorpdfstring{$\infty$}{}-monads of full subcategories} 

In this section we introduce the notion of codensity $\infty$-monad $T_\DD$ associated with a full $\infty$-subcategory $\DD\subseteq \CC$.  We use the above theory of $\infty$-monoids in strict monoidal $\infty$-categories to define the canonical $\infty$-monad structure on $T_\DD$, and to prove that $T_\DD$ is the terminal $\DD$-preserving $\infty$-monad (Theorem \ref{th:D-preserving-coaugmented}).    We also show that if $\DD$ is a reflective $\infty$-subcategory, then $T_\DD$ is the localization functor associated with $\DD$ (Proposition  \ref{prop:localization}).

\subsection{Right Kan extensions} 

Let us recall the definition of a right Kan extension of a functor from a full $\infty$-subcategory of an $\infty$-category, as given in \cite[Def.~4.3.2.2]{lurie2009higher}.
Let $\DD$ be a full $\infty$-subcategory of an $\infty$-category $\CC$, and let $f\colon \DD \to \CC'$ be a functor to an $\infty$-category $\CC'$. A functor $\FF:\CC\to \CC'$ is called right Kan extension of $f$, if $\FF\mid_{\DD} = f$ and for any object $c$ of $\CC$ the induced triangle 
\begin{equation}
\begin{tikzcd}
\DD_{c/} \ar[r,"f_c"] \ar[d,hookrightarrow]   & \CC' \\ 
(\DD_{c/})^{\triangleleft} \ar[ru]  & 
\end{tikzcd}
\end{equation}
exhibits $\FF(c)$ as a limit of $f_c$, where $f_c$ is defined as the composition $\DD_{c/} \to \DD \xrightarrow{f} \CC',$ and, in addition, the diagonal functor in the diagram above is obtained from the composition $\DD_{c/} \hookrightarrow \CC_{c/} \to \CC'_{\FF(c)/}$ by the join-slice adjunction. 

\begin{proposition}\label{proposition:Kan_extension_along_inclusion}
Let $\DD$ be a full $\infty$-subcategory of an $\infty$-category $\CC$ and $\FF:\CC\to \CC'$ be a right Kan extension of $f:\DD\to \CC'$ to $\CC$. Then $\FF$ is a terminal object of the $\infty$-category $\Fun_{\DD/}(\CC,\CC')$ of functors under $\DD$. 
\end{proposition}
\begin{proof}
The statement is equivalent to the fact that the lifting problem 
\begin{equation}
\begin{tikzcd}
\Delta^{\{n\}} \ar[r,hookrightarrow] \ar[rr,bend left=8mm,"\FF"] & \partial \Delta^n \ar[r] \ar[d] & \Fun(\CC,\CC') \ar[d] \\
& \Delta^n \ar[r,"\tilde f"] \ar[ru,dashed] & \Fun(\DD,\CC')
\end{tikzcd}
\end{equation}
has a solution
for $n\geq 1,$ where $\tilde f$ is the composition $\Delta^n\to \Delta^0 \overset{f}\to \Fun(\CC,\CC')$. 
This lifting problem has a solution by \cite[Lemma 4.3.2.12]{luriehigheralgebra}. 
\end{proof}

\subsection{Codensity \texorpdfstring{$\infty$}{}-monads} 
Assume that $\CC$ is an $\infty$-category and $\DD\subseteq \CC$ is its full $\infty$-subcategory. 
Then a functor
\begin{equation}
T_\DD:\CC \to \CC
\end{equation}
is called codensity $\infty$-monad associated with $\DD$, if it is the right Kan extension of the inclusion $\DD\hookrightarrow \CC$ along itself. 
Proposition \ref{proposition:Kan_extension_along_inclusion} implies that $T_\DD$ is a terminal object of the $\infty$-category of endofunctors under $\DD$
\begin{equation}
\End_{\DD/}(\CC)=\Fun_{\DD/}(\CC,\CC). 
\end{equation}

For an $\infty$-category $\CC$ and a functor $\FF:\CC\to \CC,$ we use the following notations 
\begin{equation}
\Mon_\infty(\CC) = \mon_\infty(\End(\CC)), \hspace{5mm} \Mon_\infty(\CC,\FF) = \mon_\infty(\End(\CC),\FF)
\end{equation}
for the $\infty$-category of $\infty$-monads on $\CC$, and the Kan complex of $\infty$-monad structures on $\FF.$
Let us define a Kan complex of canonical $\infty$-monad structures on $T_\DD$.

It is easy to see that  $\End_{\DD/}(\CC)$ is a simplicial submonoid of $\End(\CC)$, so it is a strict monoidal $\infty$-category. 
Proposition  \ref{proposition:Kan_extension_along_inclusion} says that $T_\DD$ is a terminal object of $\End_{\DD/}(\CC)$. 
Therefore Theorem  \ref{th:Structures_of_monad_on_terminal} implies that the Kan complex of structures of $\infty$-monoid on $T_\DD$ in the strict monoidal category $\End_{\DD/}(\CC)$ is contractible. 
The inclusion $\End_{\DD/}(\CC)\subseteq \End(\CC)$ defines an inclusion of this contractible Kan complex to the Kan complex of all $\infty$-monad structures on $T_\DD:$
\begin{equation}
* \sim  \mon_\infty(\End_{\DD/}(\CC), T_\DD) \subseteq  \Mon_\infty(\CC, T_\DD).
\end{equation}  
Objects of this contractible Kan complex are called canonical $\infty$-monad structures on the codensity monad $T_\DD$.  

Sometimes we do not need a structure of an $\infty$-monad on $T_\DD$ but we consider it as a coaugmented functor. 
Since $T_\DD$ is a terminal object of $\End_{\DD/}(\CC)$, we have the similar situation with the coaugmentations  
\begin{equation}
* \sim \map_{\End_{\DD/}(\CC)}({\sf Id}_\CC,T_\DD) \subseteq {\sf Nat}({\sf Id}_\CC,T_\DD).
\end{equation}
So, we have a contractible Kan complex of natural transformations $\eta:{\sf Id}_\CC \to T_\DD$ that we call canonical coaugmentations of $T_\DD$. 
For any canonical coaugmentation $\eta$ the coaugmented object  $(T_\DD,\eta)$ is a terminal object of $\End_{\DD/}(\CC)^{\Id_\CC/}$ (Lemma  \ref{lemma:terminal_object_of_slice} in the appendix). 
Then Theorem \ref{th:terminal:coaug}  implies that any such a canonical coaugmentation $\eta:{\sf Id}_\CC\to T_\DD$ can be lifted to a canonical structure of $\infty$-monad on $T_\DD$ uniquely up to a contractible space of choices.    

It turns out that the codensity $\infty$-monad $T_\DD$ depends only on the ``retract closure'' of $\DD.$ More precisely, we prove the following.

\begin{proposition}
\label{prop:coden:retracts}
Let $\CC$ be an $\infty$-category, $\DD\subseteq \DD'$ be its full $\infty$-subcategories such that any object of $\DD'$ is a retract of an object of $\DD$. Then  $T_\DD$ exists if and only if $T_{\DD'}$ exists, and in this case there is a natural equivalence   
\begin{equation}
    T_{\DD'} \xrightarrow{\simeq} T_{\DD}.
\end{equation}
\end{proposition}
\begin{proof} 
By \cite[\href{https://kerodon.net/tag/03YQ}{Prop.03YQ}]{kerodon} we know that any functor $\FF:\DD' \to \CC$ is a right Kan extension of the restriction $\FF|_\DD.$ Then the inclusion $\iota':\DD'\hookrightarrow \CC$ is a right Kan extension of the inclusion $\iota:\DD \hookrightarrow \CC$ along the inclusion $\DD \hookrightarrow \DD'$. 
It is also known that a functor $\FF:\CC\to \CC$ is a right Kan extension of $\FF|_{\DD}$ if and only if it is a right Kan extension of $\FF|_{\DD'}$ and $\FF|_{\DD'}$ is a right Kan extension of $\FF|_{\DD}$ \cite[\href{https://kerodon.net/tag/0314}{Cor.0314}]{kerodon}. It follows that, if $T_{\DD'}$ exists, then it is equal to $T_\DD$. On the other hand, if $T_\DD$ exists, then it is a right Kan extension $T_\DD|_{\DD'}$, and $T_\DD|_{\DD'}$ is a right Kan extension of $\iota:\DD\hookrightarrow \CC.$ Then $T_\DD|_{\DD'}$ is equivalent to $\iota':\DD'\hookrightarrow \CC$, and $T_\DD$ is equivalent to $T_{\DD'}$. 
\end{proof}

\subsection{\texorpdfstring{$\infty$}{}-monads preserving a full subcategory}

For a full $\infty$-subcategory $\DD$ of an $\infty$-category $\CC$, we can consider two associated $\infty$-categories: the $\infty$-category of coaugmented functors that preserve all objects of $\DD$, and the $\infty$-category of $\infty$-monads that preserve all objects of $\DD$. This subsection is devoted to the statement that the codensity $\infty$-monad $T_\DD$ is the terminal object in both of these $\infty$-categories (Theorem~\ref{th:D-preserving-coaugmented}). In the setting of ordinary categories, this result appears in \cite[Prop.~5.4]{leinster2013codensity}.

The $\infty$-category of 
coaugmented functors is defined as the fat under-category of 
endofunctors under the identity functor
$\End(\CC)^{\Id_\CC/}.$
The category of 
$\DD$-preserving coaugmented functors $\Coaug^\DD(\CC)$ is defined as a pullback.
\begin{equation}
\begin{tikzcd}
\Coaug^\DD(\CC)\ar[r,hookrightarrow] \ar[d] & \End(\CC)^{\Id_\CC/} \ar[d] \\
(\Fun(\DD,\CC)^{\simeq})^{\iota/}\ar[r,hookrightarrow] & \Fun(\DD,\CC)^{\iota/}
\end{tikzcd}
\end{equation} 
Since equivalences in $\Fun(\DD,\CC)$ satisfy 2-out-of-3 property, 
we obtain that $(\Fun(\DD,\CC)^{\simeq})^{\iota/}$ is a full $\infty$-subcategory of $\Fun(\DD,\CC)^{\iota/}$, whose objects are equivalences $\iota \overset{\simeq}\to \FF$.
Therefore, $\Coaug^\DD(\CC)$ is also a full $\infty$-subcategory of $\End(\CC)^{\Id_\CC/}$ consisting of coaugmented functors $(\FF,\eta)$ such that $\eta_d:d\to \FF (d)$ is an equivalence for any object $d$ of $\DD$.

Note that $\End_{\DD/}(\CC)^{\Id_\CC/}$ and $\Coaug^\DD(\CC)$ are two simplicial subsets of $\End(\CC)^{\Id_\CC/}.$ It is easy to see that $\End_{\DD/}(\CC)^{\Id_\CC/} $ is a simplicial subset of  $ \Coaug^\DD(\CC).$

\begin{proposition}\label{prop:coaug_equiv}
The monomorphism 
\begin{equation}
\End_{\DD/}(\CC)^{\Id_\CC/} \overset{\simeq}\longrightarrow \Coaug^\DD(\CC)   
\end{equation}
is an equivalence of $\infty$-categories. 
\end{proposition}
\begin{proof} 
Consider the following diagram,  
\begin{equation}
\begin{tikzcd}
\Fun_{\DD/}(\CC,\CC)^{\Id_\CC/}\ar[r] \ar[d] & \Coaug^\DD(\CC) \ar[r]\ar[d] & \Fun(\CC,\CC)^{\Id_\CC/}\ar[d] \\
\Delta^0 \ar[r,"{\sf id}_\iota"] & (\Fun(\DD,\CC)^{\simeq})^{\iota/} \ar[r] & \Fun(\DD,\CC)^{\iota/}
\end{tikzcd}
\end{equation}
where $\iota : \DD \hookrightarrow \CC$ is the inclusion functor. 
The right hand square is a pullback by the definition of $\Coaug^\DD(\CC)$. 
The composite square is a pullback because the functor of fat slice is a right adjoint, and hence, it preserves pullbacks.  
So, the left-hand square is also a pullback. 
Note that the functor $\Fun(\CC,\CC)\to \Fun(\DD,\CC)$ is an isofibration \cite[Prop.2.2.5]{land2021introduction}. 
Then the functor $\Fun(\CC,\CC)^{\Id_\CC/} \to \Fun(\DD,\CC)^{\iota/}$ is also an isofibration (Lemma \ref{lemma:slice_of_isofibrations} in the appendix). Hence the functor ${\sf Coaug}^\DD(\CC) \to (\Fun(\DD,\CC)^{\simeq})^{\iota/}$ is also an isofibration. 
Therefore, all these pullbacks are Joyal homotopy pullbacks. 

We claim that the simplicial set $(\Fun(\DD,\CC)^\simeq)^{\iota/}$ is a contractible Kan complex. Indeed, $\Fun(\DD,\CC)^\simeq$ is a Kan complex, and for any Kan complex $K$ and any its object $x$, $K_{x/}$ is the path complex of $K$ \cite[(2.9)]{curtis1971simplicial}, and $K^{x/}$ is equivalent to $K_{x/}$. Therefore ${\sf id}_\iota : \Delta^0\to (\Fun(\DD,\CC)^\simeq)^{\iota/}$ is a categorical equivalence. It follows that the functor $ \Fun_{\DD/}(\CC,\CC)^{\Id_\CC/}\to {\sf Coaug}^\DD(\CC)$ is also a categorical equivalence. 
\end{proof}

Recall that if $\CC$ is an $\infty$-category, then a simplicial subset $\CC'$ of $\CC$ is called an $\infty$-subcategory, if its $1$-simplices are closed with respect to the equivalences and compositions, and an $n$-simplex of $\CC$ is in $\CC'$ if and only if all its $1$-simplices are in $\CC'$. In other words, an $\infty$-subcategory is a preimage of an ordinary subcategory of the homotopy category $h\CC$ (see \cite[Lemma 1.2.77]{land2021introduction}).  

Let $\EE$ be a strict monoidal $\infty$-category. A simplicial submonoid $\EE'$ of $\EE$ is called a strict monoidal $\infty$-subcategory, if  $U\EE'$ is an $\infty$-subcategory of $U\EE$.  

\begin{lemma}\label{lemma:subcategory-pushout}
Let $\EE'$ be a strict monoidal $\infty$-subcategory of a strict monoidal $\infty$-category $\EE.$ Assume that morphisms of $\EE'$ satisfy the following $2$-out-of-$3$ property: for any $2$-simplex  $\sigma$ of $\EE$, if any two of its edges are in $\EE'$, then $\sigma$ is in $\EE'$. Then the square 
\begin{equation}
\begin{tikzcd}
\mon_\infty(\EE') \ar[r] \ar[d,"\FF'_{\EE'}"] & \mon_\infty(\EE) \ar[d,"\FF'_\EE"] \\
(\EE')^{\1/}\ar[r] & \EE^{\1/}
\end{tikzcd}
\end{equation}
is a pullback.
\end{lemma}
\begin{proof} Since $\EE'$ satisfies the above $2$-out-of-$3$ property, $(\EE')^{\1/}$ is a full $\infty$-subcategory of $\EE^{\1/}$, whose objects are coaugmented objects $\eta:\1\to x$ such that $\eta$ is a morphism of $\EE'$. Therefore, the pullback of $(\EE')^{\1/}$ along $\FF'_\EE$ is a full $\infty$-subcategory of 
$\mon_\infty(\EE)$ spanned by $\infty$-monoids $\MM$, whose unit $\eta^\MM: \1 \to \MM$ is in $(\EE')^{\1/}$. So, it is sufficient to prove that if $\MM:\Ner(\Delta_+)\to \EE$ is an $\infty$-monoid in $\EE$ such that its unit $\eta^\MM :\1\to \MM_0$ is a morphism of $\EE',$ then $\MM(\Ner(\Delta_+))\subseteq \EE'$. 

Take an $\infty$-monoid $\MM:\Ner(\Delta_+)\to \EE$ such that  $\eta^\MM:\1 \to \MM_0$ is in $\EE'$. Since $\EE'$ is an $\infty$-subcategory, it is sufficient to prove that for any morphism $\alpha:[n]\to [m]$ of $\Delta_+$, we have $\mu(\alpha)\in \EE'.$ Note that the map $[-1]\to [n]$  in $\Delta_+$ is the join of $n+1$ copies of the map $\eta^{0}:[-1] \to [0]$. It follows that $\MM([-1]\to [n])$ is in $\EE'$. For any morphism $\alpha:[n]\to [m]$, the triangle 
\begin{equation}
\begin{tikzcd}
& {[-1]}
\ar[rd] 
\ar[ld] & \\
{[n]}
\ar[rr,"\alpha"] & & 
{[m]}
\end{tikzcd}
\end{equation}
is commutative. Then, using the $2$-out-of-$3$ property of $\EE',$ we obtain that $\MM(\alpha)$ is in $\EE'$. 
\end{proof}

The $\infty$-category of $\DD$-preserving $\infty$-monads is defined as the pullback 
\begin{equation}\label{diag:def-mon^D}
\begin{tikzcd}
\Mon^\DD_\infty(\CC) \ar[r,hookrightarrow] \ar[d] & \Mon_\infty(\CC) \ar[d] \\ 
\Coaug^\DD(\CC)\ar[r,hookrightarrow] & \End(\CC)^{\Id_\CC/}. 
\end{tikzcd}
\end{equation}
This pullback is the full $\infty$-subcategory of $\Mon_\infty(\CC)$ consisting of $\infty$-monads, whose underlying coaugmented functor is $\DD$-preserving.

\begin{theorem}\label{th:D-preserving-coaugmented}
Let $\DD$ be a full $\infty$-subcategory of an $\infty$-category $\CC$ and $T_\DD$ be the  associated codensity $\infty$-monad. Then the following holds. 
\begin{itemize}
    \item[(a)] Any canonical coaugmentation on $T_\DD$ defines a terminal object  of the $\infty$-category of $\DD$-preserving coaugmented functors;
    \item[(b)] Any canonical $\infty$-monad structure on $T_\DD$  defines a terminal object of the $\infty$-category of $\DD$-preserving $\infty$-monads.
\end{itemize}
\end{theorem}
\begin{proof} (a). Since $T_\DD$ is a terminal object of $\End_{\DD/}(\CC)$ (Proposition \ref{proposition:Kan_extension_along_inclusion}), the coaugmented endofunctor $(T_\DD,\eta)$ is a terminal object of  $\End_{\DD/}(\CC)^{\Id_\CC/}$ (Lemma \ref{lemma:terminal_object_of_slice}). Then the statement follows from Proposition \ref{prop:coaug_equiv}.

(b).  If we have two functors $\FF,\GG:\CC\to \CC$ and a natural transformation $\varphi:\Delta^1 \times \CC \to \CC$ from $\FF$ to $\GG$, for a morphism $\alpha:c\to c'$ in $\CC$, we denote by $\varphi_\alpha:\FF (c)\to \GG(c')$ the diagonal morphism defined by $\varphi_\alpha=\varphi(\delta,\alpha),$ where $\delta:0\to 1$ is the only non-identity morphism of $\Delta^1.$ The ordinary components $\varphi_c:\FF(c)\to \GG(c)$ can be described as $\varphi_c=\varphi_{{\sf id}_c}$. 
\begin{equation}
\begin{tikzcd}
\FF(c)\ar[r,"\FF(\alpha)"] \ar[d,"\varphi_c"'] \ar[rd,"\varphi_\alpha"] & \FF(c') \ar[d,"\varphi_{c'}"] \\
\GG(c) \ar[r,"\GG(\alpha)"'] & \GG(c')
\end{tikzcd}
\end{equation}
It is easy to check that the product on $\End(\CC)_1$ satisfies the equation $(\psi\cdot \varphi)_\alpha=\psi_{\varphi_\alpha}.$ 

Denote by $\tilde \DD$ the full $\infty$-subcategory of $\CC$ consisting of objects equivalent to objects of $\DD.$ Consider an $\infty$-subcategory $\EE$ of $\End(\CC)$ whose objects are endofunctors $\FF$ such that $\FF(\tilde \DD)\subseteq \tilde \DD$  and morphisms are natural transformations $\varphi:\FF\to \GG$ such that the component $\varphi_d:\FF(d)\to \GG(d)$ is an equivalence for any object $d$ of $\tilde \DD$. It is easy to see that, for any morphism $\varphi$ of $\EE$, and any equivalence $\alpha:d\to d'$ in $\tilde \DD$, the morphism $\varphi_\alpha:\FF(d)\to \GG(d')$ is also an equivalence in $\tilde \DD.$ Using this, it is easy to see that $\EE_0$ is a submonoid of $\End(\CC)_0$ and $\EE_1$ is a submonoid of $\End(\CC)_1$. Therefore, $\EE$ is a strict monoidal $\infty$-subcategory of $\End(\CC).$ Also note that the class of morphisms of $\EE$ satisfies $2$-out-of-$3$ property in $\End(\CC)$.

We claim that 
\begin{equation}
\Coaug^\DD(\CC) = \EE^{\Id_\CC/}. 
\end{equation}
It is easy to see that $\EE^{\Id_\CC/} \subseteq \Coaug^\DD(\CC)$. On the other hand, 
any $n$-simplex  $\Delta^n\to \Coaug^\DD(\CC)$ corresponds to a map 
$\varphi : \Delta^1\times  \Delta^n \to \End(\CC)$ such that $\{0\}\times \Delta^n$ goes 
to the degeneracy of  $\Id_\CC$ and, for $0\leq i\leq n,$ $\Delta^1\times \{i\}$ 
goes to a natural transformation 
$\varphi|_{\Delta^1\times \{i\}} :\Id \to \FF_i$ such that $(\varphi|_{\Delta^1\times \{i\}})_d:d\to \FF(d)$ is an equivalence for any object $d$ of $\DD$. In particular all the $1$-simplices of the form $\{0\}\times \Delta^{\{i,j\}}, 0\leq i\leq j\leq n$ go to the identity natural transformation $\Id_\CC\to \Id_\CC.$ Using the $2$-out-of-$3$ property of equivalences, one can check that this implies that all the $1$-simplices of $\Delta^1\times \Delta^n$ go to natural transformations whose $d$-components are equivalences for any object $d$ of $\DD.$ And hence $\varphi$ lies in $\EE^{\Id_\CC/}.$

We also claim that 
\begin{equation}
\Mon^\DD_\infty(\CC) = \mon_\infty(\EE). 
\end{equation}
Indeed, it follows from the definition \eqref{diag:def-mon^D} and Lemma \ref{lemma:subcategory-pushout}. Therefore, any canonical $\infty$-monad structure $\theta$ on $T_\DD$ defines a object in $\mon_\infty(\EE),$ whose image in $\EE^{\Id_\CC/}$ is terminal by (a). Then Theorem \ref{th:terminal:coaug} implies that $\theta$ is terminal in $\mon_\infty(\EE)$.   
\end{proof}

\subsection{Localizations as codensity 
\texorpdfstring{$\infty$}{}-monads} 

A full $\infty$-subcategory $\DD$ of an $\infty$-category $\CC$ is called reflective, if the inclusion functor $\DD \hookrightarrow \CC$ admits a left adjoint functor. The composition of the left adjoint functor with the inclusion $L_\DD : \CC \to \CC$ is called the localization functor associated with $\DD$.

\begin{proposition}\label{prop:localization}
Let $\DD$ be a reflective $\infty$-subcategory of an $\infty$-category $\CC$ with the associated localization functor $L:\CC\to \CC$. Then $L$ is equivalent to the codensity $\infty$-monad associated with $\DD$
\begin{equation}
L_\DD \simeq T_\DD. 
\end{equation}
\end{proposition}
\begin{proof}
The unit of the adjunction defines a coaugmentation $\eta:\Id \to L_\DD$ such that for any object $c$ of $\CC$ and an object $d$ of $\DD$ the map 
\begin{equation}
\eta^*_c : \map_{\CC}(L_\DD(c),d) \xrightarrow{\sim} \map_\CC(c,d)
\end{equation}
is a homotopy equivalence. In particular, for any object $d$ of $\DD$, the map $\eta_d:d \to L_\DD(d)$ is an equivalence in $\DD$. By Proposition \ref{prop:coaug_equiv}, we can assume (replacing $L_\DD$ by an equivalent coaugmented functor) that $L_\DD(d)=d$ and $\eta_d={\sf id}_d$ for any object $d$ of $\DD$. For any object $c$ of $\CC$, the unit $\eta_c: c\to L_\DD(c)$ defines an initial object of $\DD_{c/}$ (Lemma \ref{lemma:local_initial} in the appendix). Therefore, the limit of the functor $\DD_{c/}\to \DD \hookrightarrow \CC$ is given by $L_\DD(c)$ and $\eta_c:c\to L_\DD(c)$ defines the limit cone. It follows that $L_\DD$ is the right Kan extension of the inclusion $\DD \hookrightarrow \CC$ along itself.   
\end{proof}

\section{\bf Functorial over- and under-categories}  
\label{section:functorial-over-cat}

This section is a preparation to the next section, where we describe the $\MM$-completion as a codensity $\infty$-monad. In the proof of this description we use models for the over-category $\CC_{/c}$ and the under-category $\CC_{c/}$ which are functorial by $c$. We define them using the twisted arrow category $\Tw(\CC)$ \cite[Def. 4.2.3]{land2021introduction}, \cite[\S 5.2.1]{luriehigheralgebra}. 

In this section, we recall the definition of the twisted arrow category $\Tw(\CC)$, define the functorial versions of the over- and under-categories that we denote by $\CC_{|c}$ and $\CC_{c|}$ and prove several technical results about them. At the end of the section, we apply the developed technique to prove a sufficient condition for a functor to be left cofinal (Proposition \ref{prop:cofinal}). The results of this section, including the sufficient condition for left cofinality, are used in the next section, in the proof of Lemma \ref{lemma:BK-cofinal} that we call ``a lemma of Bousfield-Kan''. 

\subsection{Twisted arrow category} 

For an ordinary category $C$, the twisted arrow category $\Tw(C)$ is a category, whose objects are morphisms of $C$, and a morphism $\alpha \to \alpha'$ consists of two morphisms $(\gamma_0,\gamma_1)$ such that $\alpha = \gamma_1 \circ \alpha' \circ \gamma_0.$
\begin{equation}
\begin{tikzcd}
x \ar[r,"\alpha"] \ar[d,"\gamma_0"] 
& 
y \
\\
x' \ar[r,"\alpha'"] 
& 
y' \ar[u,"\gamma_1"]
\end{tikzcd}
\end{equation}
The obvious functor from $\Tw(C)$ to $C\times C^\op$ is denoted by 
\begin{equation}
P : \Tw(C) \longrightarrow C\times C^\op.
\end{equation}

Now we generalize the definition to the $\infty$-categorical setting. A functor $\Delta\to \Delta$ defined by $J \mapsto J \star J^\op$, defines a cosimplicial simplicial set 
\begin{equation}
\Delta \to \sSet, \hspace{1cm}  J \mapsto \Delta^J \star (\Delta^J)^\op.   
\end{equation}
For the sake of shortness, we use the notation 
$\Delta^{J\star J^\op} = \Delta^J \star (\Delta^{J})^\op.$
The twisted arrow category $\Tw(\CC)$ of an $\infty$-category $\CC$ is a simplicial set defined by 
\begin{equation}\label{eq:Tw}
  \Tw(\CC)_J = \Hom( \Delta^{J\star J^\op} , \CC).
\end{equation} 
This construction is natural by $\CC$:  a functor $\FF:\CC\to \DD$ defines the obvious functor 
$\Tw(\FF) : \Tw(\CC) \to \Tw(\DD).$ It follows from \eqref{eq:Tw}  that $\Tw$ preserves products 
$\Tw(\prod_i \CC_i) \cong \prod_i \Tw(\CC_i).$

The inclusion $\Delta^J \sqcup (\Delta^J)^\op \hookrightarrow \Delta^{J\star J^\op}$ defines a map 
\begin{equation}
P: \Tw(\CC) \longrightarrow \CC \times \CC^\op,
\end{equation}
which is a right fibration \cite[Prop.4.2.4]{land2021introduction}, \cite[Prop.5.2.1.3]{luriehigheralgebra}. In particular, $\Tw(\CC)$ is an $\infty$-category. If $C$ is an ordinary category, there is an isomorphism 
$\Tw(\Ner(C)) \cong \Ner(\Tw(C)) $
\cite[Ex.5.2.1.2]{luriehigheralgebra}.

The map $P$ is defined by two maps $P=(p,q)$
\begin{equation}
p : \Tw(\CC)\longrightarrow \CC, \hspace{1cm} q :\Tw(\CC)\longrightarrow \CC^\op.
\end{equation}
Since a composition of cartesian fibrations is a cartesian fibration, the maps $p$ and $q$ are cartesian fibrations. Note that $p$ and $q$ are induced by the inclusions 
$\Delta^J \hookrightarrow \Delta^{J\star J^\op}$ and  $\Delta^{J^\op} \hookrightarrow \Delta^{J\star J^\op}$
respectively. 

Note that a morphism $\alpha$ of $\Tw(\CC)$ is $q$-cartesian if and only if $p(\alpha)$ is an equivalence. It follows from the fact that all morphisms of $\Tw(\CC)$ are $P$-cartesian, and the fact that a morphism of $\CC\times \CC^\op$ is cartesian with respect to the projection $\CC\times \CC^\op\to \CC^\op$ if and only if its first component is an equivalence.  

\subsection{Functorial over- and under-categories}

In this subsection we define the functorial model for over- and under-categories, and list some basic properties of them that we use in the next section. 

The category $\CC_{|c}$ is defined as  the fiber of the cartesian fibration $q: \Tw(\CC)\to \CC^\op$ over an object $c$. The map $\Delta^{J\star J^\op} \to \Delta^J \star \Delta^0$ defines an equivalence 
\begin{equation}\label{eq:C_c-equiv}
\CC_{/c} \xrightarrow{\simeq} \CC_{|c}
\end{equation}
\cite[Lemma 4.2.7]{land2021introduction}.
The straightening of $q$ defines a functor to the $\infty$-category of $\infty$-categories 
\begin{equation}
\CC \longrightarrow \Cat_\infty, \hspace{1cm} c \mapsto \CC_{|c}.
\end{equation}
In particular, any morphism $f:c\to c'$ in $\CC$ defines a functor that we denote by 
\begin{equation}
f_* : \CC_{|c} \to \CC_{|c'}.
\end{equation}
It is easy to see that $f_*$ sends a morphism $\alpha : x\to c$ of $\CC$  treated as an object of $\CC_{|c}$ to a composition of $\alpha$ and $f$ treated as an object of $\CC_{c'|}$.

Since the square  
\begin{equation}
\begin{tikzcd}[column sep = 15mm]
\CC_{|c} 
\ar[r]
\ar[d]
& 
\Tw(\CC)
\ar[d,"P"]
\\
\CC
\ar[r,"\id\times \{c\}"]
& 
\CC \times \CC^\op
\end{tikzcd}
\end{equation}
is a pullback, the map $\CC_{|c} \to \CC$ is a right fibration. The map $P$ defines a morphism of cartesian fibrations. 
\begin{equation}
    \begin{tikzcd}
        \Tw(\CC)\ar[rr,"P"] \ar[rd,"q"'] && \CC \times \CC^\op \ar[ld] \\
        & \CC^\op &
    \end{tikzcd}
\end{equation}
The straightening of this morphism shows that the map $\CC_{|c} \to \CC$ is natural by $c$ in the following sense: for any $f:c\to c'$ the triangle. 
\begin{equation}\label{eq:f_*triangle}
    \begin{tikzcd}
        \CC_{|c} \ar[rr,"f_*"] \ar[rd] && \CC_{|c'} \ar[ld] \\
        &\CC &
    \end{tikzcd}
\end{equation}
is commutative up to a natural equivalence. 

For a functor $\FF:\CC\to \DD$ and an object $d$ of $\DD$, we denote by $\CC_{|d}$ the pullback 
\begin{equation}
\begin{tikzcd}
\CC_{|d} 
\ar[r]
\ar[d]
& 
\DD_{|d} 
\ar[d]
\\
\CC 
\ar[r,"\FF"]
& 
\DD
\end{tikzcd}
\end{equation}
Using that the right-hand vertical map is a right fibration, we obtain that this is a Joyal homotopy pullback. Therefore, the equivalence \eqref{eq:C_c-equiv} defines an equivalence 
\begin{equation}\label{eq:C_d-equiv}
\CC_{/d} \xrightarrow{\simeq} \CC_{|d}.
\end{equation}
The $\infty$-category $\CC_{|d}$ is also natural by $d$. Indeed, we can take the pullback 
\begin{equation}
\begin{tikzcd}[column sep=15mm]
T_\FF
\ar[r]
\ar[d,"q_\FF"]
& 
\Tw(\DD)
\ar[d,"P"]
\\
\CC \times \DD^\op 
\ar[r,"\FF\times \Id"]
& 
\DD \times \DD^\op
\end{tikzcd}
\end{equation}
and obtain a right fibration $q_\FF$. Then the straightening of the composition  $T_\FF \to \CC\times \DD^\op \to \DD^\op$ defines a functor 
\begin{equation}
\DD^\op \longrightarrow \Cat_\infty, \hspace{1cm} d \mapsto \CC_{|d}. 
\end{equation}
Similar to \eqref{eq:f_*triangle}, we can show that for a morphism $f:d\to d'$ of $\DD$ the diagram 
\begin{equation}\label{eq:f_*triangle'}
    \begin{tikzcd}
        \CC_{|d} \ar[rr,"f_*"] \ar[rd] && \CC_{|d'} \ar[ld] \\
        &\CC &
    \end{tikzcd}
\end{equation}
is commutative up to a natural equivalence. 

A functor $\FF:\CC\to \DD$ induces a functor 
\begin{equation}\label{eq:F|c}
\FF_{|c} : \CC_{|c} \longrightarrow \DD_{|\FF(c)},
\end{equation}
which is equal to the restriction of $\Tw(\FF)$. 
Consider a pullback 
\begin{equation}
\begin{tikzcd}[column sep = 15mm]
S_\FF 
\ar[r]
\ar[d,"r_\FF"]
& 
\Tw(\DD)
\ar[d,"q"]
\\
\CC^\op 
\ar[r,"\FF^\op"]
& 
 \DD^\op
\end{tikzcd}
\end{equation}
Then the straightening of $r_\FF$ defines the composition  functor 
\begin{equation}
    \CC \overset{\FF}\longrightarrow \DD \longrightarrow {\sf Cat}_\infty, \hspace{1cm} c\mapsto \DD_{|\FF(c)}.
\end{equation}
The map $\Tw(\FF) : \Tw(\CC)\to \Tw(\DD)$ defines a morphism of cartesian fibrations $\Tw(\CC)\to S_\FF$ over $\CC^\op.$ It follows that the map \eqref{eq:F|c} is is a natural transformation of functors $\CC\to {\sf Cat}_\infty$. In particular, for any $f:c\to c'$ the square 
\begin{equation}\label{eq:F|c:nat}
\begin{tikzcd}
\CC_{|c} 
\ar[d,"f_*"]
\ar[r,"\FF_{|c}"]
& 
\DD_{|\FF(c)} 
\ar[d,"\FF(f)_*"]
\\
\CC_{|c'} 
\ar[r,"\FF_{|c'
}"]
& 
\DD_{|\FF(c')}
\end{tikzcd}
\end{equation}
is commutative up to a natural equivalence. 

\begin{proposition}\label{prop:functorial:over-cat}
Let $f:c\to c'$ be a morphism of $\CC,$ $K$ be a simplicial set and 
$
\FF:K\to \CC_{|c}$ and  $\GG: K \to \CC_{|c'}$
be two functors. Assume that there exists a natural transformation $\varphi$ from $\iota_{c'} \circ \GG$ to $\iota_{c} \circ \FF$ 
\[\varphi : \Delta^1 \times K \longrightarrow \Tw(\CC) \]
such that for any object $k$ of $K$ the morphism $p(\varphi_k)$ is an equivalence, and $q(\varphi_k)=f$. Then the diagram 
\[
\begin{tikzcd}
& K\ar[rd,"\GG"] \ar[ld,"\FF"'] & \\
\CC_{|c} \ar[rr,"f_*"] & & \CC_{|c'}
\end{tikzcd}
\]
is commutative up to a natural equivalence.   
\end{proposition}
\begin{proof}
The result follows from Lemma \ref{lemma:straightening} in the appendix and the fact that a morphism $\alpha$ of $\Tw(\CC)$  is $q$-cartesian if and only if $p(\alpha)$ is an equivalence. 
\end{proof}

Dually we can define 
\begin{equation}
\CC_{c|} = (\CC^\op_{|c})^\op    
\end{equation}
and the functors $f^* : \CC_{c'|}\to \CC_{c|}$
are obtained from the straightening of the cocartesian fibration 
\begin{equation}
q^\op : \Tw(\CC^\op)^\op\longrightarrow \CC^\op.
\end{equation}
Obviously there is a dual version of Proposition \ref{prop:functorial:over-cat}. 

\begin{proposition}
Let $f:c\to c'$ be a morphism of $\CC,$ $K$ be a simplicial set and 
$
\FF:K\to \CC_{c'|}$ and  $\GG: K \to \CC_{c|}$
be two functors. Assume that there exists a natural transformation $\varphi$ from $\iota_{c} \circ \GG$ to $\iota_{c'} \circ \FF$ 
\[\varphi : \Delta^1 \times K \longrightarrow \Tw(\CC) \]
such that for any object $k$ of $K$ the morphism $q(\varphi_k)$ is an equivalence, and $p(\varphi_k)=f$. Then the diagram 
\[
\begin{tikzcd}
& K\ar[rd,"\GG"] \ar[ld,"\FF"'] & \\
\CC_{c'|} \ar[rr,"f^*"] & & \CC_{c|}
\end{tikzcd}
\]
is commutative up to a natural equivalence.   
\end{proposition}

\subsection{Natural transformations induced by natural transformations}

Assume that $\FF,\GG:\CC\to \DD$ are functors and $\varphi$ is a natural transformation from $\FF$ to $\GG$
\begin{equation}
\varphi : \Delta^1 \times \CC \longrightarrow \DD. 
\end{equation}
Since $\Tw(\Delta^1\times \CC)\cong \Tw(\Delta^1)\times \Tw(\CC),$ the natural transformation induces a functor 
\begin{equation}
 \Tw(\Delta^1) \times \Tw(\CC) \longrightarrow \Tw(\DD).
\end{equation}
Therefore, any object of $\Tw([1])$ defines a functor $\Tw(\CC)\to \Tw(\DD)$ and any morphism between objects of $\Tw([1])$ defines a natural transformation. There are three objects in $\Tw([1])$: the identity morphisms $\id_0,\id_1$ and the only non-identity morphism that we denote by $\delta:0\to 1.$ The objects $\id_0,\id_1$ define $\Tw(\FF)$ and $\Tw(\GG)$ and the object $\delta$ defines a functor that we denote by
\begin{equation}
 \varphi^\diag : \Tw(\CC) \longrightarrow \Tw(\DD).
\end{equation}
The value of $\varphi^{\sf diag}$ on an object $\alpha:c'\to c$ of $\Tw(\CC)$ can be depicted as follows. 
\begin{equation}
\begin{tikzcd}[column sep = 10mm, row sep = 10mm] 
\FF(c') 
\ar[r]
\ar[d]
\ar[rd,"\varphi^{\sf diag}(\alpha)" description]
&
\GG(c') 
\ar[d]
\\ 
\FF(c)
\ar[r]
&
\GG(c)
\end{tikzcd}
\end{equation}
It is easy to check that for a map $\alpha:\Delta^{J\star J^\op}\to \CC$ treated as a $J$-simplex of $\Tw(\CC)$ we have 
\begin{equation}\label{eq:varphi^diag-u}
\varphi^\diag(\alpha) = \varphi\circ (u_J, \alpha),
\end{equation}
where $u_J:\Delta^{J\star J^\op}\to \Delta^1$ is the map sending $\Delta^J$ to $0$ and $\Delta^{J^\op}$ to $1$.

There are two non-identity morphisms in $\Tw([1])$
\begin{equation}
\id_0\leftarrow \delta \to  \id_1
\end{equation}
defined by the diagram.
\begin{equation}
\begin{tikzcd}
0 
\ar[d,"\id_0"]
& 
0 
\ar[r] 
\ar[l] 
\ar[d,"\delta"] 
& 
1 
\ar[d,"\id_1"]
\\
0 
\ar[r]
& 
1 
& 
1
\ar[l]
\end{tikzcd}
\end{equation}
They define natural transformations. 
\begin{equation}
\Tw(\FF) \overset{\sigma}\Longleftarrow \varphi^\diag \overset{\tau}\Longrightarrow  \Tw(\GG)
\end{equation}
For an object $\alpha:c'\to c$ of $\Tw(\CC)$ the $\alpha$-components of these natural transformations can be depicted as follows 
\begin{equation}\label{eq:sigmatau} 
\begin{tikzcd}[column sep=15mm, row sep =15mm]
\FF(c')
\ar[d,"\FF(\alpha)"']
\ar[rd,phantom,"\sigma_\alpha"]
&
\FF(c')
\ar[d,"\varphi^\diag(\alpha)" description]
\ar[l,"\id"']
\ar[r,"\varphi_{c'}"]
\ar[rd,phantom,"\tau_\alpha"]
&
\GG(c')
\ar[d,"\GG(\alpha)"]
\\
\FF(c)
\ar[r,"\varphi_c"]
&
\GG(c)
&
\GG(c)
\ar[l,"\id"']
\end{tikzcd}
\end{equation}

For a functor $\FF: \CC \to \DD$, we denote by 
\begin{equation}\label{eq:F_[//c]}
\FF_{|c} : \CC_{|c} \longrightarrow \DD_{|\FF(c)}  
\end{equation}
the restriction of $\Tw(\FF)$. Using the formula \eqref{eq:varphi^diag-u}, we see that, for a natural transformation $\varphi$ from $\FF$ to $\GG$, the functor $\varphi^\diag$ restricts to a functor $\CC_{|c}\to \DD_{|\GG(c)}$ that we denote by 
\begin{equation} 
(\varphi_c)_* \: \tilde \circ\: \FF_{|c} : \CC_{|c} \longrightarrow \DD_{|\GG(c)}. 
\end{equation}

\begin{lemma}\label{lemma:phi-diag}
There is a natural equivalence 
\begin{equation}
(\varphi_c)_* \: \tilde \circ\: \FF_{|c} \simeq \varphi_c  \circ \FF_{|c}.
\end{equation}
\end{lemma}
\begin{proof} The proof is based on Proposition  \ref{prop:functorial:over-cat}. So, it is sufficient to construct a natural transformation
$\sigma' : \Delta^1 \times \CC_{|c} \to \Tw(\DD),$
from $\iota_{\GG(c)}\circ ((\varphi_c)_* \: \tilde \circ\: \FF_{|c})$ to  $\iota_{\FF(c)} \circ  \FF_{|c}$, such that for any object $\alpha:c'\to c$ of  $\CC_{|c}$ the morphism $p(\sigma'_\alpha)$ is an identity and $q(\sigma'_\alpha)=\varphi_c$. We define $\sigma'$ as the restriction of $\sigma : \varphi^\diag \Rightarrow \Tw(\FF)$. 
It is easy to see that this restriction is a natural transformation from $\iota_{\GG(c)}\circ ((\varphi_c)_* \: \tilde \circ\: \FF_{|c})$ to  $\iota_{\FF(c)} \circ  \FF_{|c}$ and  $P( \sigma'_\alpha ) = (\id_{\FF(c')}, \varphi_c)$  (see \eqref{eq:sigmatau}) because in $\Tw([1])$ we have  $P(\delta \to \id_0) = (\id_0, \delta)$.  
\end{proof}

\begin{proposition}\label{prop:natual-to-natural}
For a pair of functors $\FF,\GG:\CC\to \DD$ and a natural transformation $\varphi : \FF\Rightarrow \GG$, we can associate a natural transformation $\tilde \varphi : (\varphi_c)_* \circ \FF_{|c} \Rightarrow \GG_{|c}$ such that the diagram 
\begin{equation}\label{eq:tildevaphi-proj}
\begin{tikzcd}
\Delta^1 \times \CC_{|c} 
\ar[r]
\ar[d,"\tilde \varphi"]
& 
\Delta^1 \times \CC 
\ar[d,"\varphi"]
\\ 
\DD_{|\GG(c)} 
\ar[r]
& \DD 
\end{tikzcd}
\end{equation}
commutes up to a natural equivalence. Moreover, if we have a commutative diagram 
\begin{equation}\label{eq:diag-S-T}
\begin{tikzcd}[column sep = 15mm]
\Delta^1 \times \CC 
\ar[d,"\varphi"]
\ar[r,"\Delta^1\times S"]
& 
\Delta^1 \times \CC' 
\ar[d,"\varphi'"]
\\
\DD 
\ar[r,"T"]
& 
\DD'
\end{tikzcd}
\end{equation}
where $\varphi : \FF \Rightarrow \GG,$ and $\varphi': \FF' \Rightarrow \GG'$ are natural transformations, then the diagram 
\begin{equation}\label{eq:diag-S-T'}
\begin{tikzcd}[column sep = 15mm]
\Delta^1 \times \CC_{|c} 
\ar[d,"\tilde \varphi"]
\ar[r,"\Delta^1\times S_{|c}"]
& 
\Delta^1 \times \CC'_{|c'} 
\ar[d,"\tilde \varphi'"]
\\
\DD_{|\GG(c)} 
\ar[r,"T_{|\GG(c)}"]
& 
\DD'_{|\GG(c')}
\end{tikzcd}
\end{equation}
commutes up to a natural equivalence, where $c'=S(c).$ 
\end{proposition}
\begin{proof}Consider the restriction $\tau' : (\varphi_c)_* \: \tilde \circ \: \FF_{|c} \Rightarrow \GG_{|c}$ of  
the natural transformation $\tau : \varphi^\diag \Rightarrow \Tw(\GG)$.  We define $\tilde \varphi$ as a composition of $\tau'$ and the equivalence $(\varphi_c)_* \: \tilde \circ \: \FF_{|c} \simeq (\varphi_c)_*  \circ  \FF_{|c}$ from Lemma \ref{lemma:phi-diag}. It is easy to check that the diagram 
\begin{equation}
\begin{tikzcd}[column sep=15mm]
\Delta^1 \times \Tw(\CC) 
\ar[r,"\Delta^1\times p"]
\ar[d,"\tau"]
& 
\Delta^1 \times \CC 
\ar[d,"\varphi"]
\\ 
\Tw(\DD) 
\ar[r,"p"]
& \DD 
\end{tikzcd}
\end{equation}
commutes. It follows that the diagram \eqref{eq:tildevaphi-proj} is commutative up to a natural equivalence. If we have a commutative diagram \eqref{eq:diag-S-T}, it is easy to check that the diagram 
\begin{equation}
\begin{tikzcd}[column sep = 15mm]
\Delta^1 \times \Tw(\CC) 
\ar[d,"\tau"]
\ar[r,"\Delta^1\times \Tw(S)"]
& 
\Delta^1 \times \Tw(\CC') 
\ar[d,"\tau'"]
\\
\Tw(\DD) 
\ar[r,"\Tw(T)"]
& 
\Tw(\DD')
\end{tikzcd}
\end{equation}
also commutes. It follows that the diagram \eqref{eq:diag-S-T'} commutes up to a natural equivalence. 
\end{proof}

The following proposition is dual to Proposition \ref{prop:natual-to-natural}.

\begin{proposition}\label{prop:natual-to-natural-dual}
For every pair of functors $\FF,\GG:\CC\to \DD$ and a natural transformation $\varphi : \FF\Rightarrow \GG$, we can associate a natural transformation $\bar  \varphi : (\varphi_c)^* \circ \GG_{c|} \Rightarrow \FF_{c|}$ such that the diagram 
\begin{equation}
\begin{tikzcd}
\Delta^1 \times \CC_{c|} 
\ar[r]
\ar[d,"\bar \varphi"]
& 
\Delta^1 \times \CC 
\ar[d,"\varphi"]
\\ 
\DD_{\GG(c)|} 
\ar[r]
& \DD 
\end{tikzcd}
\end{equation}
commutes up to a natural equivalence. Moreover, if we have a commutative diagram 
\begin{equation}
\begin{tikzcd}[column sep = 15mm]
\Delta^1 \times \CC 
\ar[d,"\varphi"]
\ar[r,"\Delta^1\times S"]
& 
\Delta^1 \times \CC' 
\ar[d,"\varphi'"]
\\
\DD 
\ar[r,"T"]
& 
\DD',
\end{tikzcd}
\end{equation}
where $\varphi : \FF \Rightarrow \GG,$ and $\varphi': \FF' \Rightarrow \GG'$ are natural transformations, then the diagram 
\begin{equation}
\begin{tikzcd}[column sep = 15mm]
\Delta^1 \times \CC_{c|} 
\ar[d,"\bar \varphi"]
\ar[r,"\Delta^1\times S_{c|}"]
& 
\Delta^1 \times \CC'_{c'|} 
\ar[d,"\tilde \varphi'"]
\\
\DD_{\FF(c)|} 
\ar[r,"T_{\FF(c)|}"]
& 
\DD'_{\FF(c')|}
\end{tikzcd}
\end{equation}
commutes up to a natural equivalence, where $c'=S(c).$ 
\end{proposition}

\subsection{Application: a sufficient condition for cofinality} 

Recall that a functor of $\infty$-categories $\FF : \CC \to \DD$ is called left cofinal (or coinitial), if for any functor to an $\infty$-category $p: \DD \to  \AA$, the induced functor $\AA_{/p}\to \AA_{/p\FF}$ is an equivalence. In particular, the composition with $\FF$ sends a limit cone in $\AA$ to a limit cone in $\AA$. The $\infty$-categorical version of Quillen's theorem A says that a functor of $\infty$-categories $\FF:\CC\to \DD$ is left cofinal if and only if the relative over-category $\CC_{/d}$ is weakly contractible for any object $d$ of $\DD$ \cite[Th.4.4.20]{land2021introduction}.  

\begin{proposition}\label{prop:cofinal}
Let $\CC$ and $\DD$ be $\infty$-categories and $\FF:\CC \to \DD$ be a functor. Assume that:
\begin{enumerate}
    \item there are functors $\Phi: \CC\to \CC$ and $\Psi:\DD \to \DD$ such that the diagram 
    \[ 
    \begin{tikzcd}
        \CC 
        \ar[r,"\FF"]
        \ar[d,"\Phi"]
        & 
        \DD 
        \ar[d,"\Psi"]
        \\
        \CC 
        \ar[r,"\FF"]
        & 
        \DD
    \end{tikzcd}
    \]
    commutes up to  a natural equivalence; 
    \item there is an object $c_0$ of $\CC$ together with natural transformations 
    \[{\sf Const}(c_0) \overset{\sigma} \Longrightarrow \Phi \overset{\xi} \Longleftarrow \Id_\CC, \hspace{5mm} 
    {\sf Const}(d_0) \overset{\tau} 
    \Longrightarrow
    \Psi 
    \overset{\zeta}
    \Longleftarrow
    \Id_\DD
    \]
    where $d_0=\FF(c_0),$ such that the diagrams 
    \[
    \begin{tikzcd}
        \Delta^1 \times \CC 
        \ar[r,"\Delta^1 \times  \FF"]
        \ar[d,"\sigma"]
        & 
        \Delta^1 \times \DD 
        \ar[d,"\tau"]
        \\
         \CC 
        \ar[r,"\FF"]
        & 
        \DD
    \end{tikzcd}
    \hspace{1cm}
    \begin{tikzcd}
    \Delta^1 \times \CC 
    \ar[r,"\Delta^1 \times  \FF"]
    \ar[d,"\xi"]
    & 
    \Delta^1 \times \DD 
    \ar[d,"\zeta"]
    \\
    \CC 
    \ar[r,"\FF"]
    & 
    \DD
    \end{tikzcd}
    \]
    commute up to natural equivalences;
    \item  for each object $d$ of $\DD$, the morphism $\zeta_d : d \to \Psi(d)$ is left invertible.
\end{enumerate}
Then $\FF$ is left cofinal.
\end{proposition}
\begin{proof}  We fix an object $d$ of $\DD$ and prove that $\CC_{/d}$ is weakly contractible. 
The equivalence $\CC_{/d} \simeq  \CC_{|d} $   \eqref{eq:C_d-equiv} implies that it is sufficient to prove that $\CC_{|d}$ is weakly  contractible. In order to prove this, we  construct a functor 
$
\hat \Phi  : \CC_{|d} \to \CC_{| d},
$
an object $\hat c_0$ of $\CC_{|d}$ and natural transformations 
\begin{equation}\label{eq:nat_main}
{\sf Const}(\hat c_0) \overset{\hat \sigma}
\Longrightarrow
\hat \Phi 
\overset{\hat \xi} 
\Longleftarrow
\Id_{\CC_{|d}}.
\end{equation} 
Let us construct this data.

By Proposition \ref{prop:natual-to-natural}, 
we have a functor 
\begin{equation}
\Psi_{|d}:\DD_{|d}\longrightarrow \DD_{|\Psi(d)}   
\end{equation}
and natural transformations 
\begin{equation}
{\sf Const}(\tau_{d}) 
\overset{\tilde \tau} 
\Longrightarrow
\Psi_{|d} 
\overset{\tilde \zeta}
\Longleftarrow
(\zeta_d)_*. 
\end{equation}
Consider a morphism $r:\Psi(d)\to d$ such that $\id_d$ is a composition of $r$ and $\zeta_d$. Taking the composition with the functor 
$r_* : \DD_{|\Psi(d)} \longrightarrow \DD_{|d}$, and using that 
$r_* \circ (\zeta_d)_* \simeq \Id_{\DD_{|d}}$, 
we obtain a functor 
\begin{equation}
\hat \Psi = r_* \circ \Psi_{|d} : \DD_{|d} \longrightarrow \DD_{|d}    
\end{equation}
and natural transformations 
\begin{equation}
{\sf Const}(\hat  d_0) \overset{\hat \tau} \Longrightarrow \hat \Psi \overset{\hat \zeta}
\Longleftarrow \Id_{\DD_{|d}},
\end{equation}
where $\hat d_0 = r_*(\tau_d).$  Using the commutative triangle \eqref{eq:f_*triangle}, we can check that the diagram 
\begin{equation}
\begin{tikzcd}
\CC
\ar[r,"\FF"]
\ar[d,"\Phi"]
& 
\DD 
\ar[d,"\Psi"]
& 
\DD_{| d} 
\ar[l]
\ar[d,"\hat \Psi"]
\\
\CC\ar[r,"\FF"] 
& 
\DD 
& 
\DD_{| d} \ar[l]
\end{tikzcd}
\end{equation}
is commutative up to a natural equivalence. Since $\CC_{|d}$ is the Joyal homotopy pullback of the diagram $\CC\to \DD \leftarrow \DD_{|d} $, this defines a functor 
\begin{equation}
\hat \Phi: \CC_{|d} \longrightarrow \CC_{|d}.
\end{equation}
Similarly, using the commutative diagram \eqref{eq:tildevaphi-proj}, we obtain that the diagrams
\begin{equation}
\begin{tikzcd}
\Delta^1 \times \CC
\ar[r]
\ar[d,"\sigma"]
& 
\Delta^1 \times \DD 
\ar[d,"\tau"]
& 
\Delta^1 \times \DD_{| d} 
\ar[l]
\ar[d,"\hat \tau"]
\\
\CC\ar[r] 
& 
\DD 
& 
\DD_{| d} \ar[l]
\end{tikzcd}
\end{equation}
and
\begin{equation}
\begin{tikzcd}
\Delta^1 \times \CC
\ar[r]
\ar[d,"\xi"]
& 
\Delta^1 \times 
\DD 
\ar[d,"\zeta"]
& 
\Delta^1 \times
\DD_{| d} 
\ar[l]
\ar[d,"\hat \zeta"]
\\
\CC\ar[r] 
& 
\DD 
& 
\DD_{| d} \ar[l]
\end{tikzcd}
\end{equation}
commute up to natural equivalences. 
Using that $\Delta^1 \times \CC_{|d}$ is the Joyal homotopy pullback of the diagram $\Delta^1 \times \CC \to \Delta^1 \times \DD \leftarrow \Delta^1 \times \DD_{|d}$, we obtain the natural transformations \eqref{eq:nat_main}, where $\hat c_0=(c_0,\hat d_0).$
\end{proof}

\section{\bf \texorpdfstring{$\infty$}{}-monad completion} 

\subsection{Definition} Let $\CC$ be an $\infty$-category. A cosimplicial object in $\CC$ is a functor $X:\Ner(\Delta)\to \CC$. The totalization of $X$ is the limit of $X$ in $\CC$ \cite[\href{https://kerodon.net/tag/04QZ}{04QZ}]{kerodon}.
\begin{equation}
\Tot\: X = \lim X.
\end{equation}
We say that $\CC$ admits totalizations if it admits all $\Ner(\Delta)$-indexed limits. Since the inclusion of the semi-simplex category    
$\Ner(\Delta^{\sf inj}) \hookrightarrow \Ner(\Delta)$
is left cofinal \cite[Lemma 3.8]{dror1977long}, \cite[Lemma 6.5.3.7]{lurie2009higher} any totalization can be replaced by the limit over $\Ner(\Delta^{\sf inj})$. 

Let $\CC$ be an $\infty$-category that admits totalizations and 
\begin{equation}
\MM:\Ner(\Delta_+) \longrightarrow \End(\CC)  
\end{equation}
be an $\infty$-monad i.e. an $\infty$-monoid in $\End(\CC)$. We denote by $\MM_0:\CC\to \CC$ the underling endofunctor.
Note that 
$\Ner(\Delta_+) \cong  \{[-1]\} \star \Ner(\Delta).$
By the join-slice adjunction, $\MM$ induces a cosimplicial coaugmented endofunctor 
\begin{equation}
\MM^\bullet : \Ner(\Delta) \longrightarrow \End(\CC)_{\Id_\CC/}.
\end{equation}
Since $\CC$ admits totalizations, $\End(\CC)$ also admits totalizations, and they are computed object-wise \cite[Cor.5.1.2.3]{lurie2009higher}. Hence $\End(\CC)_{\Id_\CC/}$ admits totalizations, which are preserved by the forgetful functor  $\End(\CC)_{\Id_\CC/} \to \End(\CC)$ \cite[Prop.1.2.13.8]{lurie2009higher}. We define the $\MM$-completion as the totalization of $\MM^\bullet$
\begin{equation}
(\widehat{\MM},\hat \eta) = \Tot \:   \MM^\bullet.
\end{equation}

For an object $c$ of $\CC$, we consider the coaugmented cosimplicial object 
\begin{equation}
\MM_c : \Ner(\Delta_+) \longrightarrow \CC
\end{equation}
defined as the composition of $\MM$ and the evaluation functor ${\sf ev}_c: \End(\CC)\to \CC$. 
Then the cosimplicial object 
\begin{equation}
\MM^\bullet_c : \Ner(\Delta) \longrightarrow \CC_{c/} 
\end{equation}
can be defined in two equivalent ways: as the functor corresponding to $\MM_c$ by join-slice adjunction, and as the composition of $\MM^\bullet$ with the evaluation functor $\End(\CC)_{\Id_\CC/}\to \CC_{c/}.$ Using the fact that the limits in $\End(\CC)$ are computed object-wise and the description of limits in under-categories \cite[Prop.1.2.13.8]{lurie2009higher}, we see that the evaluation functor $\End(\CC)_{\Id_\CC/}\to \CC_{c/}$ commutes with limits. Therefore the completion can be computed on objects as follows 
\begin{equation}
(\widehat{\MM}(c),\hat \eta_c) = \lim \MM^\bullet_c. 
\end{equation}

Consider the cosimplicial endofunctor (without coaugmentation)
\begin{equation}
\MM^\circ : \Ner(\Delta) \longrightarrow \End(\CC)
\end{equation}
that can be defined in two equivalent ways: as a composition of $\MM^\bullet$ with the forgetful functor $\End(\CC,\CC)_{\Id_\CC/} \to \End(\CC,\CC)$; or as the restriction of $\MM$ to $\Ner(\Delta) \subseteq \Ner(\Delta_+)$. Since $\End(\CC,\CC)_{\Id_\CC/} \to \End(\CC,\CC)$ preserves small limits, we obtain that the endofunctor $\widehat{\MM}$ can be computed as the totalization
\begin{equation}
\widehat{\MM} = \Tot\: \MM^\circ.
\end{equation}
For an object $c$, we can also consider the composition of $ \MM^\circ$ with the evaluation functor  ${\sf ev}_c:\End(\CC)\to \CC$ 
\begin{equation}
\MM^\circ_c : \Ner(\Delta) \longrightarrow \CC 
\end{equation}
and obtain that 
$\widehat{\MM}(c) = \Tot\: \MM^\circ_c$.

\subsection{A lemma of Bousfield-Kan} \label{section:a_lemma_of_BK}

In this subsection, we prove a key technical lemma for the main theorem of this section (Theorem~\ref{th:monad_completion}). This lemma is a version of a proposition from the book of Bousfield and Kan \cite[§XI.10.5]{bousfield1972homotopy}, and we refer to it as a lemma of Bousfield–Kan. The proof is long and technical, so we devote a separate subsection to it. Before presenting the full proof, we first outline the argument in the setting of ordinary category theory.

Recall that for an $\infty$-monad $\MM$ on an $\infty$-category $\CC$, we denote by $\II(\MM)$ the full $\infty$-subcategory of $\CC$ generated by objects from the image of $\MM_0$, and by  $\RR(\MM)$ the full $\infty$-subcategory of $\CC$ spanned by retracts of objects from $\II(\MM)$. 
Consider a full subcategory $\DD$ such that 
\begin{equation}
\II(\MM)\subseteq \DD \subseteq \RR(\MM).
\end{equation}
Since the image of $\MM^\bullet_c : \Ner(\Delta) \to \CC_{c/}$ lies in $\DD_{c/},$ we can consider the induced functor 
$
\MM^\bullet_{c,\DD} : \Ner(\Delta) \longrightarrow \DD_{c/}.    
$

\begin{lemma}
\label{lemma:BK-cofinal} For any full subcategory $\DD$ of $\CC$ such that $\II(\MM) \subseteq \DD \subseteq \RR(\MM)$, and any object $c$ of $\CC$, the functor
\begin{equation}
  \MM^\bullet_{c,\DD} : \Ner(\Delta) \longrightarrow \DD_{c/} 
\end{equation}
is left cofinal. 
\end{lemma}

For the reader's convenience, before we give the actual proof of the lemma, we give a sketch of the proof in the setting of ordinary category theory. Assume that $\CC=C$ is an ordinary category and $\MM :\Delta_+ \to \End(C)$ is the cobar construction of an ordinary monad $M:C\to C$. The unit of $M$ is denoted by $\eta : \Id \to M.$ The cobar construction $\MM$ applied to the map $[-1]\to [n]$ defines a natural transformation that we denote by 
\begin{equation}
\eta^{[n]}:\Id \longrightarrow M^{n+1}.
\end{equation}
In particular $\eta = \eta^{[0]}$. We denote by $\II(M)$ the full subcategory of $C$ spanned by the objects from the image of $M$, by $\RR(M)$ the full subcategory of $C$ spanned by objects $c$ such that $\eta_c : c\to M(c)$ is left invertible (Lemma \ref{lemma:R(M)}), and by $D$ a full subcategory of $C$ such that $\II\subseteq D \subseteq \RR.$ We need to prove that the functor 
\begin{equation}
M^\bullet_c : \Delta \longrightarrow D_{c/}, \hspace{1cm} [n] \mapsto \eta^{[n]}_c 
\end{equation}
is left cofinal. 

We prove the lemma using Proposition \ref{prop:cofinal}. First we construct functors 
\begin{equation}
\varphi : \Delta \to \Delta, \hspace{1cm} \psi: D_{c/} \to D_{c/} . 
\end{equation}
 The functor $\varphi$ is defined as the join with one point $\varphi=[0]\star -$. The functor $\psi$ is defined on an object $\alpha:c\to x$ of $D_{c/}$ as the composition 
\begin{equation}
\psi(\alpha)  :  c \xrightarrow{\ \eta_c\ } M(c) \xrightarrow{\ M(\alpha)\ } M(x)
\end{equation}
with the obvious definition on morphisms. Since $\eta$ is a natural transformation, $\psi(\alpha)$ can be equivalently defined as the composition 
\begin{equation}
\psi(\alpha): c \xrightarrow{\ \alpha\ } x \xrightarrow{\ \eta_x\ } M(x).
\end{equation}
It is easy to see that 
$\psi(\eta^{[n]}_c) = \eta^{[0]\star [n]}_c.$
Using this, we obtain that the diagram
\begin{equation}
\begin{tikzcd}
\Delta 
\ar[r,"M^\bullet_c"]
\ar[d,"\varphi"]
& 
D_{c/} 
\ar[d,"\psi"]
\\ 
\Delta 
\ar[r,"M^\bullet_c"]
& 
D_{c/} 
\end{tikzcd}
\end{equation}
commutes. Now we construct natural transformations 
\begin{equation}
{\sf Const}([0]) \overset{\sigma}\Longrightarrow \varphi \overset{\xi}\Longleftarrow \Id_\Delta, 
\hspace{5mm} 
{\sf Const}(\eta_c) \overset{\tau}\Longrightarrow \psi \overset{\zeta}\Longleftarrow \Id_{D_{c/}}. 
\end{equation}
The components of the natural transformations $\sigma$ and $\xi$ are defined by the inclusions 
\begin{equation}
[0] \longrightarrow [0] \star [n] \longleftarrow [n].
\end{equation}
The $\alpha$-components of the natural transformations $\tau$ and $\zeta$ are defined by the commutative diagram
\begin{equation}
\begin{tikzcd}[column sep = 15mm,row sep =15mm]
c 
\ar[d,"\eta_c"]
\ar[r,equal]
\ar[rd,phantom, "\tau_\alpha"]
&
c 
\ar[d,"\psi(\alpha)"]
&
c 
\ar[l,equal]
\ar[d,"\alpha"]
\ar[dl,phantom,"\zeta_\alpha"]
\\
M(c)
\ar[r,"M(\alpha)"']
&
M(x) 
& 
x
\ar[l,"\eta_x"]
\end{tikzcd}
\end{equation}
If $\alpha=\eta^{[n]}_c$, we obtain that the components of $\tau$ and $\zeta$ are defined by the diagram
\begin{equation}
\begin{tikzcd}[column sep = 15mm, row sep =15mm]
c 
\ar[r,equal]
\ar[d]
\ar[dr,phantom,"\tau_{\eta^{[n]}_c}"]
&
c
\ar[d]
\ar[dr,phantom,"\zeta_{\eta^{[n]}_c}"]
&
c
\ar[l,equal]
\ar[d]
\\
M(c)
\ar[r,"M(\eta^{[n]}_c)"']
& 
M^{n+2}(c)
&
M^{n+1}(c)
\ar[l,"\eta_{M^{n+1}(c)}"] 
\end{tikzcd}
\end{equation} 
It follows that 
\begin{equation}
\tau M^\bullet_c = M^\bullet_c \sigma, \hspace{1cm} \zeta M^\bullet_c =  M^\bullet_c \xi. 
\end{equation}

Thus, to apply the Proposition \ref{prop:cofinal}, we only need to check that $\zeta_\alpha : \alpha \to \psi(\alpha)$ is left invertible for any $\alpha:c\to x$ in $D_{c/}$. 
Since $x$ is in $\RR(M)$, there exists $r : M(x)\to x$ such that $r \circ \eta_x = \id_x.$  
Therefore, the fact that $\zeta_\alpha$ is left invertible follows from the commutativity of the following diagram. 
\begin{equation}
\begin{tikzcd}[column sep =15mm,row sep=15mm]
c
\ar[r,equal]
\ar[d,"\alpha"] 
\ar[rd,phantom,"\zeta_\alpha"]
& 
c 
\ar[r,equal]
\ar[d,"\psi(\alpha)"]
&
c 
\ar[d,"\alpha"]
\\
x 
\ar[r,"\eta_x"]
\ar[rr,bend right = 1cm,"\id_x"]
& 
M(x) 
\ar[r,"r"]
&
x 
\end{tikzcd}
\end{equation}
Now let us give an actual proof of Lemma \ref{lemma:BK-cofinal}.

\begin{proof}[Proof of Lemma \ref{lemma:BK-cofinal}]
Consider the equivalence  $\DD_{c/} \to  \DD_{c|}$ (see \eqref{eq:C_d-equiv}) and replace $\MM^\bullet_{c,\DD}$ with the composition 
$\MM_c' : \Ner(\Delta) \to \DD_{c|}.$
We prove that $\MM'_c$ is left cofinal using Proposition \ref{prop:cofinal}. The proof consists of the following steps.
\begin{itemize}
    \item[(A)] A construction of functors $\Phi :\Ner(\Delta) \to \Ner(\Delta)$ and $\Psi: \DD_{|c}\to \DD_{|c}$ satisfying the assumption (1) of Proposition \ref{prop:cofinal}. 
    \item[(B)] A construction of natural transformations $\sigma : {\sf Const}([0]) \Rightarrow \Phi$ and $\tau:{\sf Const}(\eta_c^\MM ) \Rightarrow \Psi$ satisfying the  assumption (2) of Proposition \ref{prop:cofinal}. 
    \item[(C)] A construction of natural transformations $\xi : \Id_{\Ner(\Delta)} \Rightarrow \Phi$ and $\zeta: \Id_{\DD_{c|}} \Rightarrow \Psi$ satisfying the assumption (2) of Proposition \ref{prop:cofinal}.
    \item[(D)] Proof of the fact that, for any $f:c\to x$ in $\DD_{c|}$, the map $\zeta_f : f \to \Psi(f)$ is left invertible.
\end{itemize}

Let us complete these steps.

(A) We denote by $\eta^{0} : [-1]\to [0]$ the only map from $[-1]$ to $[0]$. Consider functors of join with one point
\begin{equation}
\varphi: \Delta \to \Delta, 
\hspace{5mm}
\varphi_+: \Delta_+ \to \Delta_+, \hspace{5mm}
\bar \varphi: \Delta_+ \to (\Delta_+)_{[0]/}
\end{equation}
where $\varphi = [0]\star -$ and $\varphi_+ = [0]\star -$, and $\bar \varphi = ([0]\star - )_{[-1]/}.$ Here we  use the identification
\begin{equation}
 \Delta_+\cong (\Delta_+)_{[-1]/}.   
\end{equation}
Then the diagram  
\begin{equation}\label{eq:phi-phi+}
\begin{tikzcd}[column sep = 15mm]
\Delta 
\ar[r,hookrightarrow]
\ar[dd,"{\varphi}"]
& 
\Delta_+ 
\ar[d,"\bar \varphi"']
\ar[dd,bend left = 30mm,"\varphi_+"]
\\
& 
(\Delta_+)_{[0]/}
\ar[d,"{\eta^0_*}"']
\\
\Delta 
\ar[r,hookrightarrow]
& 
\Delta_+
\end{tikzcd}
\end{equation}
is commutative. Further we set 
\begin{equation}
\Phi=\Ner(\varphi), \hspace{5mm} \Phi_+ = \Ner(\varphi_+), \hspace{5mm}  \bar \Phi = \Ner(\bar \varphi).
\end{equation}
Since $\MM:\Ner(\Delta_+)\to \End(\CC)$ is a morphism of simplicial monoids, the diagram
\begin{equation}\label{eq:M-homo}
\begin{tikzcd}[column sep = 12mm]
\Ner(\Delta_+)\times \Ner(\Delta_+)
\ar[d,"{\star}"]
\ar[r,"\MM \times \MM"]
& 
\End(\CC) \times  \End(\CC)
\ar[d,"\circ "]
\ar[r,"{\id \times {\sf ev}_c}"]
& 
\End(\CC) \times \CC 
\ar[d,"{\sf ev}"] 
\\ 
\Ner(\Delta_+) 
\ar[r,"\MM"]
&
\End(\CC)
\ar[r,"{\sf ev}_c"]
& 
\CC
\end{tikzcd}
\end{equation}
is commutative. Using this, we obtain that the diagram 
\begin{equation}\label{eq:0*--M-0}
\begin{tikzcd}[column sep=15mm]
\Ner(\Delta_+)_{[-1]|} 
\ar[r,"{(\MM_c)_{[-1]|}}"]
\ar[d,"{\bar \Phi}"]
&
\CC_{c|}
\ar[d,"{(\MM_0)_{c|}}"]
\\
\Ner(\Delta_+)_{[0]|}
\ar[r,"{(\MM_c)_{[-1]|}}"]
&
\CC_{\MM_0(c)|}
\end{tikzcd}
\end{equation}
is commutative. The diagram  \eqref{eq:F|c:nat} implies that the diagram 
\begin{equation}\label{eq:eta^0-eta^M}
\begin{tikzcd}[column sep = 15mm]
\Ner(\Delta_+)_{[0]|} 
\ar[r,"(\MM_c)_{[0]|}"]
\ar[d,"\eta^0_*"]
& 
\CC_{\MM_0(c)|} 
\ar[d,"(\eta^\MM_c)_*"]
\\
\Ner(\Delta_+)_{[-1]|} 
\ar[r,"(\MM_c)_{[-1]|}"]
& 
\CC_{c|}
\end{tikzcd}
\end{equation}
is commutative up to a natural equivalence. Composing the diagrams \eqref{eq:phi-phi+}, \eqref{eq:0*--M-0} and \eqref{eq:eta^0-eta^M}, we obtain that the diagram 
\begin{equation}
\begin{tikzcd}[column sep=15mm]
\Ner(\Delta) 
\ar[r,"{\MM'_c}"]
\ar[d,"{\Phi}"]
& 
\CC_{c|}
\ar[d,"\Psi"]
\\ 
\Ner(\Delta)
\ar[r,"{\MM'_c}"]
& 
\CC_{c|} 
\end{tikzcd}
\end{equation}
commutes up to a natural equivalence, where 
\begin{equation}
\Psi=(\eta^\MM_c)_* \circ (\MM_0)_{c|}.
\end{equation}
It is easy to see that $\Psi$ restricts to a functor that we denote by the same letter  
\begin{equation}
 \Psi:  \DD_{c|} \to \DD_{c|}.    
\end{equation}
Using that $\DD_{c|}\subseteq \CC_{c|}$ is a full $\infty$-subcategory, we obtain that the diagram
\begin{equation}
\begin{tikzcd}[column sep=15mm]
\Ner(\Delta) 
\ar[r,"{\MM'_c}"]
\ar[d,"{\Phi}"]
& 
\DD_{c|}
\ar[d,"\Psi"]
\\ 
\Ner(\Delta)
\ar[r,"{\MM'_c}"]
& 
\DD_{c|} 
\end{tikzcd}
\end{equation}
commutes up to a natural equivalence. 

(B) The inclusion $[0] \to [0]\star J$ defines natural transformations 
\begin{equation}
\sigma: {\sf Const}([0]) \Rightarrow \varphi, 
\hspace{5mm}
\sigma_+ : {\sf Const}([0]) \Rightarrow \varphi_+, \hspace{5mm} 
\bar \sigma : {\sf Const}(\id_{[0]}) \Rightarrow \bar \varphi.
\end{equation}
It is easy to check that the diagram
\begin{equation}\label{eq:sigma-sigma-plus}
\begin{tikzcd}[column sep = 15mm]
[1] \times \Delta 
\ar[r,"I"]
\ar[dd,"{\sigma}"]
& 
{[1]\times \Delta_+} 
\ar[d,"\bar \sigma"']
\ar[dd,bend left = 25mm,"\sigma_+"]
\\
& 
(\Delta_+)_{[0]/}
\ar[d,"{\eta^0_*}"']
\\
\Delta 
\ar[r,"I"]
& 
\Delta_+
\end{tikzcd}
\end{equation}
is commutative. The identity map $\id_{\MM_0(c)}$ is the initial object of $\CC_{\MM_0(c)|}.$ Therefore, ${\sf Const}(\id_{\MM_0(c)})$ is the initial object of the $\infty$-category of functors $\CC_{c|}\to \CC_{\MM_0(c)|}$.
 Now, consider the natural transformation 
\begin{equation}
    \bar \tau : {\sf Const}(\id_{\MM_0(c)}) \Longrightarrow (\MM_0)_{c|} 
\end{equation} 
which is unique up to a contractible ambiguity. The diagram 
\begin{equation}\label{eq:bar-sigma-bar-tau}
\begin{tikzcd}[column sep=20mm]
\Delta^1 \times \Ner(\Delta_+) 
\ar[r,"\Delta^1\times (\MM_{c})_{[-1]|}"]
\ar[d,"\bar \sigma"]
& 
\Delta^1 \times \CC_{c|} 
\ar[d,"\bar \tau"]
\\
\Ner(\Delta_+)_{[0]|}
\ar[r,"(\MM_c)_{[0]|}"]
&
\CC_{\MM_0(c)|}
\end{tikzcd}
\end{equation}
is commutative up to a natural equivalence, because
\begin{equation}
 {\sf Const}(\id_{\MM_0(c)}) = (\MM_c)_{[0]|} \circ {\sf Const}(\id_{[0]})   
\end{equation}
is the initial object in the category of functors $\Ner(\Delta_+) \to  \CC_{\MM_0(c)|}$. We set $\tau =  (\eta^\MM_c)^* \circ \bar \tau$. Combining the diagrams
\eqref{eq:sigma-sigma-plus}, \eqref{eq:eta^0-eta^M} and 
\eqref{eq:bar-sigma-bar-tau} we obtain that the diagram 
\begin{equation}
\begin{tikzcd}[column sep = 15mm]
\Delta^1 \times \Ner(\Delta_+) 
\ar[r,"\Delta^1\times \MM'_c"]
\ar[d,"\sigma"]
&
\Delta^1 \times \CC_{c|} 
\ar[d,"\tau"]
\\
 \Ner(\Delta_+) 
\ar[r,"\MM'_c"]
&
\CC_{c|}
\end{tikzcd}
\end{equation}
commutes up to a natural equivalence. Using that $\DD_{c|}$ is a full $\infty$-subcategory of $\CC_{c|},$ we obtain that $\tau$ can be restricted to a natural transformation that we denote by the same letter 
\begin{equation}
\tau : \Delta^1 \times \DD_{c|} \longrightarrow \DD_{c|},
\end{equation}
and obtain a commutative diagram 
\begin{equation}
\begin{tikzcd}[column sep = 15mm]
\Delta^1 \times \Ner(\Delta_+) 
\ar[r,"\Delta^1\times \MM'_c"]
\ar[d,"\sigma"]
&
\Delta^1 \times \DD_{c|} 
\ar[d,"\tau"]
\\
 \Ner(\Delta_+) 
\ar[r,"\MM'_c"]
&
\DD_{c|}
\end{tikzcd}
\end{equation}

(C) The inclusion $J\to [0]\star J$ defines natural transformations 
\begin{equation}
\xi: \Id_{\Delta} \Rightarrow \varphi, \hspace{5mm} 
\xi_+ : \Id_{\Delta_+} \Rightarrow 
\varphi_+ 
\end{equation}
such that the diagram 
\begin{equation}\label{eq:xi-xi_+}
\begin{tikzcd}[column sep=15mm]
{[1]\times \Delta} 
\ar[r]
\ar[d,"\xi"]
&
{[1] \times \Delta_+}
\ar[d,"\xi_+"]
\\
\Delta
\ar[r]
& 
\Delta^+
\end{tikzcd}
\end{equation}
is commutative.  Using again the diagram \eqref{eq:M-homo}, and considering the map $\Delta^1 \to \Ner(\Delta_+)$ defined by the morphism $\eta^0: [-1]\to [0]$, we obtain that the diagram 
\begin{equation}
\begin{tikzcd}[column sep=15mm]
\Delta^1 \times \Ner(\Delta_+) 
\ar[r,"\Delta^1\times \MM_c"]
\ar[d,"\xi_+"]
&
\Delta^1 \times \CC
\ar[d,"\eta^\MM"]
\\
\Ner(\Delta_+)
\ar[r,"\MM_c"]
&
\CC
\end{tikzcd}
\end{equation}
is commutative. By Proposition \ref{prop:natual-to-natural-dual}, the natural transformations $\xi_+: \Id_{\Delta_+} \Rightarrow \varphi_+$ 
and
$\eta^\MM : \Id \Rightarrow \MM_0 $ define two natural transformations. The first is just equal to $\xi_+$ after the identification of $\Delta_+$ with $(\Delta_+)_{[-1]/}$. The second is denoted by $\zeta.$ 
\begin{equation}
\xi_+ : \Id_{\Delta_+} \Rightarrow \varphi_+, \hspace{10mm} \zeta : \Id \Rightarrow \Psi.
\end{equation}
Proposition \ref{prop:natual-to-natural-dual} says that the diagram 
\begin{equation}
\begin{tikzcd}[column sep = 20mm]
\Delta^1 \times \Ner(\Delta_+) 
\ar[r,"\Delta^1 \times (\MM_c)_{[-1]|}"] 
\ar[d,"\xi_+"]
& 
\Delta^1 \times \CC_{c|} 
\ar[d,"\zeta"] 
\\ 
\Ner(\Delta_+) 
\ar[r,"(\MM_c)_{[-1]|}"]
&
\CC_{c|}
\end{tikzcd}
\end{equation}
is commutative up to a natural equivalence. 
Composing this with \eqref{eq:xi-xi_+}, we obtain a commutative diagram. 
\begin{equation}
\begin{tikzcd}[column sep = 20mm]
\Delta^1 \times \Ner(\Delta) 
\ar[r,"\Delta^1 \times \MM'_c"] 
\ar[d,"\xi"]
& 
\Delta^1 \times \CC_{c|} 
\ar[d,"\zeta"] 
\\ 
\Ner(\Delta) 
\ar[r,"\MM'_c"]
&
\CC_{c|}.
\end{tikzcd}
\end{equation}
As usual, we can restrict to $\DD_{c|}$ and obtain a commutative diagram  
\begin{equation}
\begin{tikzcd}[column sep = 20mm]
\Delta^1 \times \Ner(\Delta) 
\ar[r,"\Delta^1 \times \MM'_c"] 
\ar[d,"\xi"]
& 
\Delta^1 \times \DD_{c|} 
\ar[d,"\zeta"] 
\\ 
\Ner(\Delta) 
\ar[r,"\MM'_c"]
&
\DD_{c|}.
\end{tikzcd}
\end{equation}

(D) Consider a map $f:c\to x$, where $x$ is an object of  $\RR.$ We need to check that the morphism $\zeta_f : f \to \Psi(f)$ is left invertible in 
$\CC_{c|}$. Since $\Psi=(\eta^\MM_c)_* \circ (\MM_0)_{c|}$, we obtain that $\Psi(f)$ is a composition of $\MM_0(f):\MM_0(c)\to \MM_0(x)$ and $\eta^\MM_c : c\to \MM_0(c).$ Using that the map $\DD_{c/} \to \DD_{c|}$ is an equivalence, it is sufficient to check that the corresponding morphism is left invertible in the ordinary  under-category $\DD_{c/}$. The corresponding morphism in $\DD_{c/}$ is a $2$-simplex of $\CC$ that can be depicted as follows
\begin{equation}
\begin{tikzcd}[column sep = 15mm]
 & 
 x
\ar[rd,"\eta_x^\MM"]
\ar[d,phantom,"\zeta_f"]
 &
 \\
 c
 \ar[ru,"f"]
 \ar[rr,"\Psi(f)"']
 &
 \ 
 &
 \MM_0(x).
\end{tikzcd}
\end{equation}
Since $x$ is from $\RR(\MM)$, there is a map $r: \MM_0(x) \to x$ such that $\id_x$ is a composition of $r$ and $\eta_x^\MM:x\to \MM_0(x)$ (Lemma \ref{lemma:R(M)}). Therefore, the fact that $\zeta_f$ is left invertible follows from the fact that there is a $3$-simplex  $\Delta^3\to \CC$ that can be depicted as follows
\begin{equation}
\begin{tikzcd}[column sep = 20mm, row sep = 15mm]
&
x
\ar[r,"\eta^\MM_x"]
\ar[rrd,"\id_x" description]
&
\MM_0(x)
\ar[rd,"r"]
& 
\\
c 
\ar[rrr,"f"'] 
\ar[ru,"f"]
\ar[rru,"\Psi(f)" description]
& & & x,
\end{tikzcd}
\end{equation}
where the restriction to  $\Delta^{\{0,1,2\}}$ represents $\zeta_f$, the restriction to $\Delta^{\{0,1,3\}}$ represents $\id_f$, and the restriction to  $\Delta^{\{1,2,3\}}$ represents the fact that $\id_x$ is a composition of $r$ and $\eta_x^\MM.$
This $3$-simplex can be obtained as a lifting of the map from the inner horn  $\Lambda^{3}_{1} \to \CC$. 
\end{proof}

\subsection{Monad completion as a codensity \texorpdfstring{$\infty$}{}-monad} 

Recall that for an $\infty$-monad $\MM$ on an $\infty$-category $\CC$, we consider four full subcategories of $\CC$
\begin{equation}\label{eq:iclusions:IASR}
\II(\MM) \subseteq \AA(\MM) \subseteq \SS(\MM) \subseteq \RR(\MM)\subseteq \CC, 
\end{equation}
where $\II(\MM)$ is spanned by objects from the image of $\MM_0$; $\AA(\MM)$ is spanned by objects that admit a structure of an $\MM$-algebra; $\SS(\MM)$ is spanned by objects $c$ such that the coaugmented cosimplicial object $\MM_c$ is right split;  and $\RR(\MM)$ is spanned by retracts of objects of $\II(\MM)$ (Proposition \ref{prop:ISAR}). 

\begin{theorem}\label{th:monad_completion}
Let $\CC$ be an $\infty$-category that admits totalizations and $\MM$ be an $\infty$-monad in $\CC$. Assume that $\DD$ is a full $\infty$-subcategory of $\CC$ such that $\II(\MM)\subseteq \DD \subseteq \RR(\MM)$. Then the codensity $\infty$-monad $T_{\DD}$ exists and is equivalent to the $\MM$-completion 
\begin{equation}
    T_{\DD} \simeq \widehat{\MM}
\end{equation}
as a coaugmented functor. 
\end{theorem}
\begin{proof}
Let $x$ be an object of $\SS(\MM)$. Then Proposition \ref{prop:split_cosimplicial} implies that $\hat \eta_x : x \to \widehat \MM(x)$ is an equivalence.  Hence $\widehat \MM$ is $\SS(\MM)$-preserving. In particular, it is $\II(\MM)$-preserving. The class of objects $x$ such that $\hat \eta_x$ is an equivalence is closed under retracts. Therefore, $\widehat \MM$ is $\DD$-preserving. By Proposition \ref{prop:coaug_equiv} there is an equivalent coaugmented functor $(T,\eta)\simeq (\widehat \MM,\hat \eta)$ such that $T(x)=x$ and $\eta_x=\id_x$ for $x\in \RR(\MM).$ 
By Lemma \ref{lemma:BK-cofinal} we obtain that, for any object $c$ of $\CC$, the diagram 
\begin{equation}
\begin{tikzcd}
\DD_{c/} 
\ar[r]
\ar[d]
& 
\CC 
\\
(\DD_{c/})^\triangleleft 
\ar[ru]
&
\end{tikzcd}
\end{equation} 
exhibits $T(c)$ as a limit cone. Then $T=T_{\DD}$.
\end{proof}

\begin{corollary}
\label{cor:monad_completion} 
Under the assumption of Theorem \ref{th:monad_completion}, the codensity  $\infty$-monads $T_{\II(\MM)},$ $T_{\AA(\MM)}$, $T_{\SS(\MM)},$ $T_{\RR(\MM)}$ exist and equivalent to the $\MM$-completion 
\begin{equation}
T_{\II(\MM)} \simeq T_{\AA(\MM)} \simeq T_{\SS(\MM)} \simeq T_{\RR(\MM)} \simeq \widehat{\MM}.
\end{equation}
\end{corollary}
\begin{proof} This 
follows from Theorem \ref{th:monad_completion} and the inclusions \eqref{eq:iclusions:IASR}. 
\end{proof}

Recall that a parallel pair of morphisms $d_0,d_1 : c_1\to c_0$ in an ordinary category $C$ is called reflective, if there is a morphism $s_0 : c_0 \to c_1$ such that $d_0s_0=d_1s_0={\sf id}_{c_0}.$ In other words, a parallel pair is reflective if it can be extended to a $1$-truncated simplicial object $\Delta^\op_{\leq 1}\to C$. A reflective coequalizer is a coequalizer of a reflective pair of morphisms. Equivalently, a reflective coequalizer is a colimit of a functor $\Delta^\op_{\leq 1}\to C$. Here we use the dual notion of a coreflective equalizer. A pair of morphisms $d^0,d^1:c^0 \to c^1$ is coreflective, if there is a morphism $s^0:c^1\to c^0$ such that $s^0d^0 = s^0d^1 = {\sf id}_{c^0}.$ A coreflective equalizer is an equalizer of a coreflective pair of morphisms, and equivalently, a limit of a functor $\Delta_{\leq 1} \to C.$  

\begin{corollary} Let $C$ be an ordinary category that admits coreflective equalizers, and $M$ be a monad on $C$. Assume that $D$ is a full subcategory of $C$ such that $\II(M)\subseteq D \subseteq \RR(M)$.  Then the codensity monad $T_D$ exist and equivalent to the coreflective equalizer of the natural transformations
\begin{equation}
   M\eta,\eta M: M \to M^2.
\end{equation}
\end{corollary}
\begin{proof}
The $1$-truncated simplex category $\Delta_{\leq 1}$ is a left cofinal subcategory of $\Delta$ in the sense of ordinary category theory. Therefore a limit of a functor from $\Delta$ to an ordinary category is equal to the coreflective  equaliser.  Therefore $\widehat{M}$ is the equalizer of two maps $M\eta,\eta M: M \to M^2$ and the statement follows from Theorem \ref{th:monad_completion}. 
\end{proof}

\subsection{Bousfield-Kan completion} 

For a commutative ring $R$, there are two versions of the Bousfield-Kan $R$-completion pointed version, which is an endofunctor on the category of pointed spaces $\Spc_*$, and unpointed version, which is an endofunctor on the category of spaces $\Spc$ without a choice of a base-point. In this subsection we show in detail how to present the unpointed version of the Bousfield-Kan $R$-completion as a codensity $\infty$-monad (Theorem \ref{th:BK}), and present a similar result about the pointed version, which can be proven similarly (Theorem \ref{th:BK-pointed}).

Let $R$ be a commutative ring. For a set $X$, we denote by $R\cdot X$ the free $R$-module generated by $X$, and denote by 
\begin{equation}
R_{\sf a} X \subseteq R \cdot X
\end{equation} 
the affine subspace of $R\cdot X$ consisting of linear combinations $\sum_{x\in X} r_x x$ such that $\sum_{x\in X} r_x=1$. It is easy to see that any choice of an element $x_0\in X$ defines a bijection 
\begin{equation}
R_{\sf a} X \cong (R\cdot X)/(R\cdot \{x_0\}).
\end{equation} 
In particular, any choice of an element of $X$ defines a structure of $R$-module on $R_{\sf a}X$. Note that $R\cdot -$ has a natural monad structure and $R_{\sf a}$ is its submonad.

For two sets $X,Y$ we have an isomorphism $R\cdot (X\times Y) \cong (R\cdot X)\otimes_R (R\cdot Y)$, and the bilinear map $(R\cdot X)\times (R\cdot Y) \to R\cdot (X\times Y)$ induces a map
\begin{equation}\label{eq:R-product}
R_{\sf a} X \times R_{\sf a} Y \longrightarrow R_{\sf a} (X\times Y). 
\end{equation}

The component-wise extension of the monad $R_{\sf a}$ to the category of simplicial sets is denoted by the same symbol 
\begin{equation}
R_{\sf a} : \sSet \longrightarrow \sSet.
\end{equation} 
Using the map \eqref{eq:R-product}, for any simplicial sets $X,Y$ and a map $\Delta^J \times X \to Y,$ we can naturally define a map
$
\Delta^J \times R_{\sf a}X \to R_{\sf a}\Delta^J \times R_{\sf a}X \to R_{\sf a}(\Delta^J\times X) \to R_{\sf a}Y.
$
This gives a structure of a simplicial functor on $R_{\sf a}$. Hence  $R_{\sf a}$ is a simplicial monad, and the cobar construction defines a simplicial functor (see Subsection \ref{subsection:monoida_and_monads_in _simplicial_cats})
\begin{equation}
R^\bullet_{\sf a}  : \Delta_+ \longrightarrow \End_\Delta(\sSet), 
\end{equation}
where $\Delta_+$ is treated as a simplicial category with discrete homs. Restricting to $\Delta \subseteq \Delta_+$ and using the adjunction, we obtain a simplicial functor 
$
\tilde R^\bullet_{\sf a} : \sSet \longrightarrow \sSet^{\Delta},
$
where $\sSet^{\Delta} = \Fun_\Delta(\Delta,\sSet)$ is the simplicial category of cosimplicial simplicial sets.

In a simplicial model category $\mathbf{A}$, for a cosimplicial object $A:\Delta \to \mathbf{A}$, there are two different notions that we  distinguish: the totalization of $A$ and the homotopy totalization of $A$. The homotopy totalization is just the homotopy limit 
\begin{equation}
\hTot(A) = \holim\: A.
\end{equation} 
Let us remind the definition of the totalization. Consider a simplicial functor
\begin{equation}
[-,=] : (\sSet^{\Delta})^\op \times \mathbf{A}^\Delta \longrightarrow \mathbf{A} 
\end{equation}
defined on objects as the end $[S,A]= \int_{[n]\in \Delta}(A^n)^{S^n}$ \cite[Example A.2.9.29]{lurie2009higher}.  
The totalization functor is a simplicial functor 
\begin{equation}
\Tot : \mathbf{A}^\Delta \longrightarrow \mathbf{A}
\end{equation}
defined by $\Tot(A)=[\Delta^{(-)},A]$, where $\Delta^{(-)}$ is the cosimplicial object of standard simplices $[n]\mapsto \Delta^{[n]}$. For a cosimplicial object $A$, there is a natural map 
\begin{equation}
\varphi_A : \Tot(A) \longrightarrow \hTot(A),
\end{equation}
which is a weak equivalence whenever $A$ is Ready fibrant \cite[Th.19.8.7]{hirschhorn2003model}, \cite[Cor.A.2.9.30]{lurie2009higher} (see also \cite[\S XI.4.4]{bousfield1972homotopy}).

The $R$-completion functor is a simplicial functor
\begin{equation}
R_\infty : \sSet \longrightarrow \sSet
\end{equation}
defined as the composition $R_\infty = \Tot \circ \tilde R^\bullet_{\sf a}.$ Using that $R_a X$ is either empty or has a structure of a simplicial $R$-module, and that the full subcategory of Kan complexes is closed under limits and ends, we see that $R_\infty X$ is a Kan complex, for any simplicial set  $X$.  The restriction of $R_\infty$ to Kan complexes is denoted by the same symbol   
\begin{equation}
R_\infty : \Kan \longrightarrow \Kan.
\end{equation} 

Denote by $\mathfrak{U}$ the underlying universe that we use and consider a universe $\widehat{\mathfrak{U}}$ such that $\mathfrak{U}\in \widehat{\mathfrak{U}}$. 
We denote by $\widehat \sSet$ and $\widehat{\Kan}$ the simplicial categories of $\widehat{\mathfrak{U}}$-small simplicial sets, and $\widehat{\mathfrak{U}}$-small Kan complexes respectively.

\begin{lemma}\label{lemma:r-comletion:is:lim}
The $R$-completion treated as a simplicial functor $R_\infty : \Kan \to \widehat{\sSet}$ is the homotopy totalization  
\begin{equation}
R_\infty  \simeq  \hTot(R^\bullet_a|_\Delta)
\end{equation}
in the simplicial model category of simplicial functors $\Fun_\Delta(\Kan,\widehat{\sSet})$ with the projective model structure. 
\end{lemma}
\begin{proof} 
We denote by $\Ner(\Delta_{/-})$ the cosimplicial simplicial set defined by the nerve of the over-category $[n] \mapsto \Ner(\Delta_{/[n]})$. Then, in a simplicial model category $\AAA$, the homotopy limit of a cosimplicial object $A$ can be computed by the formula  $\hTot(A) = [\Ner(\Delta_{/-}), A]$ \cite[Def.18.1.8]{hirschhorn2003model}. There is a natural map of cosimplicial simplicial sets $\Ner(\Delta_{/-}) \to \Delta^{(-)}$, which induces a map $\varphi_A:\Tot(A) \to \hTot(A)$. The map $\varphi_A$ is a weak equivalence, if $A$ is Ready fibrant.  

Since the category $\Kan$ is $\widehat{\mathfrak{U}}$-small,  we can consider the simplicial model category $\AAA=\Fun_\Delta(\Kan,\widehat{\sSet})$ with the projective model structure.  Note that for a cosimplicial functor $F\in \AAA^\Delta$, a simplicial set $X$ and a cosimplicial simplicial set $S$, we have a natural isomorphism $[S,F](X)\cong [S,F(X)].$ In particular,
\begin{equation}
\Tot(F)(X)\cong \Tot(F(X)), \hspace{1cm} \hTot(F)(X) \cong \hTot(F(X)). 
\end{equation}

We need to prove that the natural transformation 
 $\Tot(R^\bullet_a|_\Delta) \to \hTot(R^\bullet_a|_\Delta)$ 
is an weak equivalence. 
It is sufficient to show that the component  
$\Tot(R^\bullet_aX|_\Delta ) \to \hTot(R^\bullet_a X|_\Delta)$ 
of this natural transformation is a weak equivalence, for any simplicial set $X.$ So, we need to prove that $R^\bullet_a X|_\Delta$ is Reedy fibrant cosimplicial simplicial set, for any $X$.

A cosimplicial simplicial set $S$ is called \emph{group-like}, if $S^n$ is equipped with a structure of simplicial group, the co-face maps $d^i : S^n \to S^{n+1}$ are morphisms of simplicial groups for $i\neq 0$, and the degeneracy maps $s^i: S^n \to S^{n-1}$ are morphisms of simplicial groups for all $i$. Group-like cosimplicial simplicial sets are fibrant \cite[\S X.4.9]{bousfield1972homotopy}. It is easy to see that $R^\bullet_a X|_\Delta$ is group-like because any choice of a basepoint $x_0\in X$ defines a structure of $R$-module on ${R}^{n+1}_aX$ that satisfies these conditions. The statement follows. 
\end{proof}

The $\infty$-category of spaces $\Spc$ is the homotopy coherent nerve of the simplicial category of Kan complexes $\Kan$. Composing the homotopy nerve of the cobar construction 
$
R_{\sf a}^\bullet : \Delta_+ \to \End_\Delta(\Kan)
$
with the morphism of simplicial monoids  
\begin{equation}
\Ner_{\Kan,\Kan} : \Ner(\End_\Delta(\Kan)) \longrightarrow \End(\Spc)
\end{equation}
(see Subsection \ref{subsection:monoida_and_monads_in _simplicial_cats}), we obtain an $\infty$-monad that we denote by 
\begin{equation}
\MM^R : \Ner(\Delta_+) \longrightarrow \End(\Spc). 
\end{equation}

\begin{proposition}\label{prop:BK-coden} If we treat the $R$-completion as a simplicial coaugmented endofunctor $R_\infty : \Kan \to \Kan$, then there is an equivalence of coaugmented endofunctors 
\begin{equation}
\widehat \MM_R \simeq \Ner_{\Kan,\Kan}(R_\infty).
\end{equation}
\end{proposition}
\begin{proof}
Consider the functor 
\begin{equation}
\Ner_{\Kan,\widehat{\sSet}} : \Ner(\Fun_\Delta(\Kan, \widehat{\sSet})) \longrightarrow \Fun(\Spc,\Ner(\widehat{\sSet}))
\end{equation} 
(see \eqref{eq:Ner_CC'}). 
If we denote by $\Fun_\Delta(\Kan, \widehat{\sSet})^\circ$ the full simplicial subcategory of $\Fun_\Delta(\Kan, \widehat{\sSet})$ generated by fibrant cofibrant functors in the projective model category structure, and use that $\Kan$ is $\widehat{\mathfrak{U}}$-small and locally Kan, we obtain that $\Ner_{\Kan,\widehat{\sSet}}$ can be restricted to an equivalence  of $\infty$-categories
\begin{equation}\label{eq:funct-equiv}
\Ner(\Fun_\Delta(\Kan, \widehat{\sSet})^\circ) \overset{\simeq}\longrightarrow \Fun(\Spc,\widehat{\Spc})
\end{equation}
\cite[Prop.4.2.4.4]{lurie2009higher}. 
Homotopy limits in the simplicial model category $\Fun_\Delta(\Kan,\widehat{\sSet})$ correspond to limits in the $\infty$-category $\Ner(\Fun_\Delta(\Kan,\widehat{\sSet})^\circ)$
\cite[Prop.A.3.3.12]{lurie2009higher}. Then Lemma \ref{lemma:r-comletion:is:lim} implies that the functor \eqref{eq:funct-equiv} sends the fibrant-cofibrant replacement of  $R_\infty$ to a functor equivalent to $\widehat{\MM}^R$. The statement follows.  
\end{proof}

In view of Proposition \ref{prop:BK-coden}, let us then  identify  $\widehat{\MM}^R$ with $R_\infty$ and denote it also by $R_\infty.$ Consider the full $\infty$-subcategory 
\begin{equation}
\KK(R) \subseteq \Spc
\end{equation}
of the $\infty$-category of spaces generated by the empty space and the spaces that homotopy equivalent to a product of Eilenberg–MacLane spaces
 $ \prod_{n\geq 0} K(M_n,n),$ where $M_n$ is an $R$-module. Here we treat $K(M_0,0)$ just as a discrete space. 

\begin{lemma}\label{lemma:K(R)} 
There are inclusions of full subcategories 
\begin{equation}
\II(\MM^R) \subseteq \KK(R) \subseteq \RR(\MM^R). 
\end{equation} 
\end{lemma}
\begin{proof} Prove that $\II(\MM^R) \subseteq \KK(R).$ 
An object of $\II(\MM^R)$ has the form $R_{\sf a}X$. If $X=\emptyset$, then $R_{\sf a}X = \emptyset $ is an object of $\KK(R).$ If $X\neq \emptyset,$ we can choose a basepoint in $X$ and obtain a structure of a simplicial $R$-module on $R_{\sf a}X$. In particular, $\pi_n(R_{\sf a}X)$ is an $R$-module, for any $n\geq 0.$ Since a chain complex of non-negatively graded abelian groups is quasi-isomorphic to the trivial chain complex of its homology groups, using the Dold-Kan correspondence,  we obtain that $R_{\sf a} X$ is homotopy equivalent to the product of $K(\pi_n(R_{\sf a}X) ,n).$  Therefore, $R_{\sf a} X$ is in $\KK(R)$. 

Prove that $\KK(R)\subseteq \RR(\MM^R)$. A non-empty space from $\KK(R)$ is homotopy equivalent to a space $\prod K(M_n,n)$, where $M_n$ is an $R$-module. Since $\MM^R_0(\emptyset)=\emptyset$, we only need to check that the map $X \to R_{\sf a} X$ is left invertible for $X=\prod K(M_n,n)$. It is sufficient to prove that $X\to R\cdot X$ is left invertible. We denote by ${\sf DK}:{\sf Ch}_{\geq 0} \to {\sf sAb}$ the Dold-Kan functor from the category of non-negatively graded chain complexes of abelian groups to the category of simplicial abelian groups. Then $X$ is homotopy equivalent to ${\sf DK}(M_*),$ where $M_*$ is the chain complex with trivial differentials. The space $R\cdot X$ is equivalent to ${\sf DK}(R\cdot M_*),$ where $R \cdot M_*$ is the chain complex with trivial differentials, whose $n$-th component $(R \cdot  M_*)_n=R \cdot  M_n$ is the free $R$-module generated by the underlying set of $M_n$. Therefore, the result follows from the fact that the inclusion  $M_n \to R \cdot M_n$ is left invertible for any $n$. 
\end{proof}

\begin{theorem}\label{th:BK} For any commutative ring $R$, the  Bousfield-Kan $R$-completion $R_\infty : \Spc \to \Spc$ 
is equivalent to the codensity $\infty$-monad $T_{\KK(R)}$ as a coaugmented functor. 
\end{theorem}
\begin{proof}
It follows from Theorem \ref{th:monad_completion} and Lemma \ref{lemma:K(R)}.  
\end{proof}

\begin{corollary}\label{cor:BK-universal} 
The Bousfield-Kan $R$-completion  satisfies the following two universal properties. 
\begin{itemize}
\item[(a)] $R_\infty$ is the terminal object in the $\infty$-category of $\KK(R)$-preserving coaugmented functors;

\item[(b)] The structure of a coaugmented functor on $R_\infty$ extends to an $\infty$-monad structure uniquely up to a contractible space of choices. This $\infty$-monad is the terminal object in the $\infty$-category of $\KK(R)$-preserving $\infty$-monads.
\end{itemize}
\end{corollary}
\begin{proof}
Theorem \ref{th:BK} says that $R_\infty$ is the codensity $\infty$-monad $T_{\KK(R)}$. Then the universal properties follow from Theorem \ref{th:D-preserving-coaugmented}. 
\end{proof}

The Bousfield-Kan $R$-completion sends the point space to itself $R_\infty *=*.$ Therefore, for any pointed space $X$, the space $R_\infty X$ has a natural base-point. So, we obtain a pointed version of the Bousfield-Kan $R$-completion functor 
\begin{equation}
R_\infty : \Spc_* \longrightarrow \Spc_*.
\end{equation}
Consider the full $\infty$-subcategory 
\begin{equation}
\KK_*(R) \subseteq \Spc_*
\end{equation}
of the $\infty$-category of pointed spaces generated by the spaces that homotopy equivalent to a product of Eilenberg–MacLane spaces
 $ \prod_{n\geq 0} K(M_n,n),$ where $M_n$ is an $R$-module. The following theorem can be proved similarly to Theorem \ref{th:BK}.

\begin{theorem}\label{th:BK-pointed} For any commutative ring $R$, the pointed version of the Bousfield-Kan $R$-completion $R_\infty : \Spc_* \to \Spc_*$ 
is equivalent to the codensity $\infty$-monad $T_{\KK_*(R)}$ as a coaugmented functor. 
\end{theorem}

\section{\bf Appendix}

\subsection{Technical lemmas about \texorpdfstring{$\infty$}{}-categories}

In this subsection, we collect some general lemmas about $\infty$-categories, which seem to be well known to experts in the field, but we could not find a reference in the literature.

\begin{lemma}\label{lemma:fiber:terminal}
Let $\CC$ be an $\infty$-category and  $t$ be its terminal object. Assume that $n\geq 1$ and $\varphi : \partial\Delta^n\to \CC$ is a map  such that the composition $\Delta^{\{n\}} \hookrightarrow \partial\Delta^n \overset{\varphi}\to \CC $ is defined by $t$. Then the fiber of the restriction map $\CC^{\Delta^n} \to \CC^{\partial \Delta^n}$ over $\varphi$ is a contractible Kan complex. 
\end{lemma}
\begin{proof}
Let us treat $\Delta^n$, $\partial \Delta^n$ and $\CC$ as pointed simplicial sets with base points $n$ and $t$. Then $\varphi\in \map_*(\partial \Delta^n,\CC)$ and we have the following pullback 
\begin{equation}
\begin{tikzcd}
\map_*(\Delta^n,\CC)\ar[r,hookrightarrow] \ar[d] & \CC^{\Delta^n} \ar[d] \\ 
\map_*(\partial \Delta^n,\CC)\ar[r,hookrightarrow] & \CC^{\partial \Delta^n} 
\end{tikzcd}
\end{equation}
Therefore, the fibers of the vertical maps coincide and it is sufficient to prove that the map $\map_*(\Delta^n,\CC) \to \map_*(\partial \Delta^n,\CC)$ is a trivial fibration. 

Since $\partial \Delta^n \cong  \Delta^{n-1} \sqcup_{\partial \Delta^{n-1}} \Lambda^{n}_n$, we obtain  $\CC^{\partial \Delta^n} \cong  \CC^{\Delta^{n-1}} \times_{\CC^{\partial \Delta^{n-1}}} \CC^{\Lambda^{n}_n}$ and 
\begin{equation}
\map_*(\partial \Delta^n,\CC) \cong \CC^{\Delta^{n-1}} \times_{\CC^{\partial \Delta^{n-1}}} \map_*(\Lambda^n_n,\CC).  
\end{equation}
Using the join-slice adjunction and the formulas $\Delta^n = \Delta^{n-1} \star \Delta^0$ and $\Lambda^n_n = \partial \Delta^{n-1} \star \Delta^0$, we obtain 
\begin{equation}
\map_*(\Delta^n,\CC) \cong (\CC_{/t})^{\Delta^{n-1}}, \hspace{1cm} \map_*(\Lambda^n_n,\CC)\cong (\CC_{/t})^{\partial \Delta^{n-1}}. 
\end{equation}
Hence it is sufficient to note that the map 
\begin{equation}
(\CC_{/t})^{\Delta^{n-1}} \longrightarrow \CC^{\Delta^{n-1}} \times_{\CC^{\partial\Delta^{n-1}}}(\CC_{/t})^{\partial \Delta^{n-1}}
\end{equation}
is a trivial fibration. Indeed, this  follows from the fact that $\CC_{/t}\to \CC$ is a trivial fibration, $\partial \Delta^{n-1} \to \Delta^{n-1}$ is a cofibration and 
 \cite[Prop.11.5]{goerss2009simplicial}. 
\end{proof}

\begin{lemma}\label{lemma:terminal_object_of_slice} 
Let $\CC$ be an $\infty$-category, $c$ be its object and $t$ be its terminal object.  Then any morphism $\eta:c\to t$ is a terminal object of $\CC^{c/}$ and $\CC_{c/}.$ 
\end{lemma}
\begin{proof} Since there is a canonical categorical equivalence $\CC_{c/}\to \CC^{c/}$, we need to prove the statement only for $\CC_{c/}$. Note that in general, an object $a$ of an $\infty$-category $\AA$ is terminal if and only if the map $\AA_{/a}\to \AA$ is a categorical equivalence. It is known that $a$ is terminal if and only if the map $\AA_{/a}\to \AA$ is a trivial fibration. On the other hand the map $\AA_{/a}\to \AA$ is a left fibration for any $a$ \cite[Th.1.4.23]{land2021introduction}. In particular, the map $\AA_{/a}\to \AA$ is a categorical fibration for any $a$. So it follows from the fact that a categorical fibration is a categorical equivalence if and only if it is a trivial fibration (trivial fibrations in Joyal model structure coincide with trivial fibrations in Kan model structure).

Therefore, in order to prove that $\eta:c\to t$ is terminal, it is sufficient to prove that $(\CC_{c/})_{/\eta}\to \CC_{c/}$ is a categorical equivalence. Since $t$ is a terminal object, the  functor $\CC_{/t}\to \CC$ is a categorical equivalence. Therefore $(\CC_{/t})_{\eta/}\to \CC_{c/}$ is also a categorical equivalence \cite[Prop.1.2.9.3]{lurie2009higher}. We claim that there is an isomorphism $(\CC_{/t})_{\eta/}\cong (\CC_{c/})_{/\eta}.$ It follows from the isomorphism of hom-sets from any simplicial set $X$
\begin{equation}
\Hom(X, (\CC_{/t})_{\eta/}) \cong \Hom_{\Delta^0\star \Delta^0/} (\Delta^0\star X\star \Delta^0,\CC)\cong \Hom(X, (\CC_{c/})_{/\eta}),
\end{equation}
where the map $\Delta^0\star \Delta^0\cong \Delta^1 \to \CC$ is defined by $\eta$. Therefore the map $(\CC_{c/})_{/\eta}\to \CC_{c/}$ is also a categorical equivalence.
\end{proof}
 
\begin{lemma}\label{lemma:slice_of_isofibrations}
Let $\FF:\CC\to \DD$ be an isofibration of $\infty$-categories, $c$ be an object of $\CC$ and $d=\FF(c)$. Then the induced functors of 
ordinary slice categories $\CC_{c/}\to \DD_{d/}$ and fat slice categories $\CC^{c/}\to \DD^{d/}$ are  isofibrations. 
\end{lemma}
\begin{proof} Let us prove the statement for the case of fat slice categories, and leave the proof for ordinary slice categories to the reader. 
Take a monomorphism of simplicial sets $A\to B$ which is a categorical equivalence. We need to show that the lifting problem 
\begin{equation}
\begin{tikzcd}
A\ar[r] \ar[d] & \CC^{c/} \ar[d] \\
B\ar[r]\ar[ru,dashed] & \DD^{d/}
\end{tikzcd}
\end{equation}
has a solution. By the  adjunction, this lifting problem is equivalent to  the following lifting problem. 
\begin{equation}
\begin{tikzcd}
\Delta^0 \ar[r]\ar[rr,"c",bend left = 1cm]  & \Delta^0\diamond  A \ar[r]\ar[d] & \CC \ar[d] \\
&\Delta^0 \diamond B \ar[r] \ar[ru,dashed] & \DD 
\end{tikzcd}
\end{equation}
The latter lifting problem can be solved because $\Delta^0\diamond  A \to \Delta^0 \diamond B$ is also a monomorphism, which is a categorical equivalence  \cite[Cor.2.5.17]{land2021introduction}  (for the case of ordinary slice categories we need to use that $\Delta^0\star A \to \Delta^0 \star B$  is also a monomorphism, which is a categorical equivalence \cite[Cor.2.5.20]{land2021introduction}).  
\end{proof}

\begin{lemma}[{sf. \cite[Prop. 5.2.6]{land2021introduction}}] \label{lemma:local_initial}
Let $\DD$ be a reflective $\infty$-subcategory of an $\infty$-category $\CC$ with a reflector $L: \CC\to \DD$. For any object $c$ of $\CC$, the unit $\eta_c : c \to L(c)$ is an initial object of $\DD_{c/}$. 
\end{lemma}
\begin{proof} Let $f:c\to d$ be a morphism to an object $d$ of $\DD$. Since $\DD$ is a full $\infty$-subcategory of $\CC$, we have  $\map_{\DD_{c/}}(\eta_c,f) = \map_{\CC_{c/}}(\eta_c,f)$. So we have to check that $\map_{\CC_{c/}}(\eta_c,f)$ is contractible. By \cite[Prop. 3.3.18]{land2021introduction}, the square 
\begin{equation}
\begin{tikzcd}
\map_{\CC_{c/}}(\eta,f)\ar[r] \ar[d] & \map_{\CC}(L(c),d) \ar[d,"\eta^*"] \\
\Delta^0\ar[r,"f"] & \map_{\CC}(c,d) 
\end{tikzcd}
\end{equation}
is a homotopy pullback square. Since the right-hand vertical map $\eta^*$ is a homotopy equivalence, we obtain that $\map_{\CC_{c/}}(\eta_c,f)$ is contractible. 
\end{proof}

\begin{lemma}\label{lemma:straightening}
Let $p:\EE\to \CC$ be a cartesian fibration, $f:x\to y$ be a morphism of $\CC,$ $K$ be a simplicial set and 
\begin{equation}
\FF:K\to \EE_y, \hspace{1cm}  \GG: K \to \EE_x
\end{equation}
be two functors to the fibers over $x$ and over $y$. Assume that there exists a natural transformation $\varphi$ from $\iota_x \circ \GG$ to $\iota_y \circ \FF$  
\begin{equation}
\varphi : \Delta^1 \times K \longrightarrow \EE, 
\end{equation}
where $\iota_x:\EE_x\to \EE$ and $\iota_y:\EE_y\to \EE$ are the canonical inclusions, 
such that for any object $k$ of $K$ the component of $\varphi_k:\GG(k)\to \FF(k)$ is a cartesian morphism over $f$. Then the diagram 
\begin{equation}
\begin{tikzcd}
& K\ar[rd,"\GG"] \ar[ld,"\FF"'] & \\
\EE_y\ar[rr,"f^*"] & & \EE_x
\end{tikzcd}
\end{equation}
is commutative up to a natural equivalence, where $f^*:\EE_y \to \EE_x$ is the map induced by $f$ via the straightening of $p.$  
\end{lemma}
\begin{proof}
We denote by $\sSet_+$ the category of simplicial sets. For marked simplicial set $X$ and a marked $\infty$-category $\AA$ we denote by $\Fun^\flat(X,\AA)$ the full subcategory of $\Fun(X,\AA)$ generated by maps sending marked morphisms to marked morphisms. In other words \begin{equation}
\Fun^\flat(X,\AA)_J = \Hom_{\sSet_+}((\Delta^J)^\flat \times X , \AA). 
\end{equation} 

For a cartesian fibration $p:\EE\to \CC$, we consider the simplicial set  $\Fun^\flat(X,\EE)=\Fun^\flat(X,\EE^\natural)$. For a map $f:X\to \CC^\sharp$, we denote by $\Fun^\flat_f(X,\EE)$ the fiber of the map $\Fun^\flat(X,\EE) \to \Fun^\flat(X,\CC^\sharp)$ over $f$.  It is known that, for a marked anodyne map $i:Y \to X$, the map 
\begin{equation}\label{eq:map^flat-trivial}
\Fun^\flat_{f}(X, \EE) \longrightarrow \Fun^\flat_{fi}(Y,\EE)
\end{equation}
is a trivial fibration \cite[Prop.3.2.16]{land2021introduction}.

The inclusion $\{1\} \times K^\flat \to (\Delta^1)^\sharp \times K^\flat$ is marked anodyne \cite[Prop.3.1.2.3]{lurie2009higher}. Therefore the map
\begin{equation}
\Fun^\flat_{\tilde f}((\Delta^1)^\sharp \times K^\flat,\EE)\to \Fun(K,\EE_y)  
\end{equation}
is a trivial fibration, where $\tilde f$ is the composition $\Delta^1 \times K \to \Delta^1 \xrightarrow{f} \CC$. By \cite[Prop.3.3.1]{land2021introduction}, there exists a natural transformation $\psi: \Delta^1 \times \EE_y \to \EE$ from $\iota_x \circ f^*$ to $\iota_y$ such that for any object of $\EE_y$ the component of $\psi$ is a cartesian morphism over $f$. Consider the natural transformation  
\begin{equation}
\varphi' : \Delta^1 \times K \xrightarrow{\Delta^1 \times \GG} \Delta^1 \times \EE_y \xrightarrow{\psi} \EE    
\end{equation}
from $\iota_x \circ f^* \circ \GG$ to $\iota_y \circ \GG$. 
Then $\varphi, \varphi' \in \Fun^\flat((\Delta^1)^\sharp \times K^\flat,\EE)$ are two elements of the fiber over $\iota_y \circ \GG\in \Fun(K,\EE_y).$ Therefore, there is a natural equivalence $\varphi\simeq \varphi'$ over $f,$ which induces a natural equivalence $\iota_x \circ f^* \circ \GG \simeq \iota_x \circ \FF$ over $x.$ It follows that there is a natural equivalence $f^* \circ \GG \simeq  \FF$.  
\end{proof}

\subsection{Walking action object} \label{subsection:walking_monoid}

This subsection is devoted to the proof of Proposition \ref{prop:walking_action} from Section \ref{section:actions}. 

\medskip 

\noindent {\bf Proposition \ref{prop:walking_action}}. {\it Let $(E,C)$ be a strict action category and $(M,x)$ be an action object in $(E,C)$. Then there exists a unique morphism of strict action categories 
\begin{equation}
(\Delta_+,\Delta_\max)\longrightarrow (E,C)   
\end{equation}
sending the action object $([0],[0])$ to $(M,x)$.}  

\begin{proof} 
We know that there exists a unique strict monoidal functor $\varphi:\Delta_+\to E$ sending the monoid $[0]$ to the monoid $M.$ Therefore, we need to prove that there exists a unique functor $\psi:\Delta_\max\to C$ such that $(\varphi,\psi)$ is a morphism of strict action categories sending the action object $([0],[0])$ to $(M,x)$. Denote by $a_x:M \oslash x\to x$ the $M$-action on $x$ and by $a_x^n:M^{\otimes n} \oslash x \to x, n\geq 0$ the map defined recursively by 
\begin{equation}
a_x^{n+1}=a_x \circ ({\sf id}_M  \oslash  a_x^{n}), 
\end{equation}
where $a_x^0={\sf id}_x$. For the case when $M=[0]\in \Delta_+$ and $x=[0]\in \Delta_\max,$ we just set $a^n=a^n_{[0]}:[n]\to [0].$ It is easy to see that any morphism in $\Delta_\max$ can be uniquely presented as $f\star a^n$, where $f$ is a morphism of $\Delta_+$ and $n\geq 0.$ Then for morphism of strict action categories $(\varphi,\psi):(\Delta_+,\Delta_\max)\to (E,C)$ sending the action object $([0],[0])$ to $(M,x)$, we have  
\begin{equation}\label{eq:psi_x}
\psi([n])=M^{\otimes n} \oslash x ,\hspace{1cm} \psi(f \star a^n)=\varphi(f)\oslash a^n_x.   
\end{equation}
It follows that, if $\psi$ exists, then it is unique. 

If there is a functor $\psi:\Delta_\max\to C$ satisfying the formulas \eqref{eq:psi_x}, then $(\varphi,\psi)$ is a morphism of strict action categories sending the action object $([0],[0])$ to $(M,x)$. It follows that it is sufficient to prove that the formulas \eqref{eq:psi_x} define a well defined functor. It is easy to see that $\psi$ defined by the formulas \eqref{eq:psi_x} sends identity morphisms to identity morphisms. Therefore, we only need to prove that $\psi$ defined by the formulas \eqref{eq:psi_x} respects the composition. 

Further in the proof we use the ordinal notations $n=\{0,\dots,n-1\}=[n-1].$ Note that for any morphism $f:n\to m$ in $\Delta_+$ and any decomposition $m=m_1+m_2$, there is a unique decomposition $f=f_1\star f_2,$ where $f_1:n_1\to m_1$ and $f_2:n_2\to m_2$. Then the composition in $\Delta_\max$ is completely defined by the formula
\begin{equation}
(f\star a^s) \circ (g_1 \star g_2 \star a^t) = (f\circ g_1) \star a^{t+l_2}
\end{equation}
where  $n,m,l_1,l_2,s,t\geq 0$ and 
\begin{equation}
 f:m\to n, \hspace{5mm} g_1:l_1\to m, \hspace{5mm} g_2:l_2\to s   
\end{equation}
are morphisms in $\Delta_+$. 
Therefore, it is sufficient to prove the following equation 
\begin{equation}\label{eq:phi-comp}
(\varphi(f)\oslash a^s_x) \circ (\varphi(g_1\star g_2)\oslash a^t_x) \overset{?} = \varphi(f\circ g_1)\oslash a^{t+l_2}_x.
\end{equation}
Since $\oslash$ is a functor, for composable pairs of morphisms $(\alpha,\alpha')$ and $(\beta,\beta')$, we have the equation
\begin{equation}
(\alpha \oslash \beta) \circ (\alpha' \oslash \beta') = (\alpha\circ \alpha') \oslash (\beta \circ \beta'). 
\end{equation}
Then, using that $\varphi(g_1\star g_2)\oslash a^t_x = \varphi(g_1)\oslash (\varphi(g_2) \oslash a^t_x)$, the equation \eqref{eq:phi-comp} can be reduced to the equation 
\begin{equation}\label{eq:phi-comp2}
a^s_x \circ (\varphi(g)\oslash a^t_x) \overset{?} = a^{t+l}_x,
\end{equation}
where $l,s,t\geq 0$ and  $g:l\to s$.

Let us prove the equation \eqref{eq:phi-comp2} for $g={\sf id}_s$ and $l=s$. For $s=0$ and $s=1$ the equation is obvious, and for $s\geq 2$ it can be proved by induction: 
\begin{align}
\label{eq:phi-comp-2-id}
a^s_x \circ ({\sf id}_{M^{\otimes s}}\oslash a^t_x) &= a_x \circ ({\sf id}_M \oslash a_x^{s-1}) \circ ({\sf id}_{M} \oslash ({\sf id}_{M^{\otimes s-1}} \oslash a^{t}_x) ) \\
& = a_x \circ ( {\sf id}_M \oslash( a^{s-1}_x \circ ({\sf id}_{M^{\otimes s-1}}\oslash a^t_x)))\\
& = a_x \circ ({\sf id}_M\circ a_x^{t+s-1}) \\
&= a^{t+s}_x.
\end{align}
Therefore, we obtain that 
\begin{equation}\label{eq:id-a}
a^s_x \circ ({\sf id}_{M^{\otimes s}}\oslash a^t_x) = a^{t+s}_x.
\end{equation}
We claim that, using \eqref{eq:id-a}, the equation 
\eqref{eq:phi-comp2} can be reduced to the equation 
\begin{equation}\label{eq:phi-comp3}
a^s_x \circ (\varphi(g)\oslash {\sf id}_x) \overset{?} = a^{l}_x,
\end{equation}
for any $g:l\to s$. Indeed, 
\begin{align}
 a^s_x \circ (\varphi(g)\oslash a^t_x) & = a^s_x \circ (\varphi(g) \oslash {\sf id}_x) \circ ({\sf id}_{M^{\otimes l}} \oslash a^t_x)\\
 &= a^l_x \circ ({\sf id}_{M^{\otimes l}} \oslash a^t_x) \\
 & = a_x^{t+l}.
\end{align}

Thus, we have reduced the problem to checking the equation \eqref{eq:phi-comp3} for any map $g:l\to s$. It is easy to see that the axioms of the $M$-action \eqref{eq:action:axioms} say that the equation \eqref{eq:phi-comp3} is satisfied for the maps  $g=\mu_{1}:2\to 1$ and $g=\eta_{1}:0\to 1.$ The equation also holds for identity maps. Therefore, it is sufficient to prove that the class of maps, such that the equation \eqref{eq:phi-comp3} is satisfied, is closed under compositions and joins. Indeed, if   $g:l\to s$ and $g':s\to k$ satisfy the equation \eqref{eq:phi-comp3}, then the composition $g'\circ g :l\to k$ satisfies this equation: 
\begin{align}
a^k_x \circ (\varphi(g'\circ g)\oslash {\sf id}_x) &= a^k_x \circ (\varphi(g')\oslash {\sf id}_x) \circ (\varphi(g)\oslash {\sf id}_x)\\
&= a_x^{s} \circ (\varphi(g)\oslash {\sf id}_x) \\
&=a_x^{l}.
\end{align}
And if $g_1:l_1\to s_1$ and $g_2:l_2\to s_2$ satisfy the equation \eqref{eq:phi-comp3}, then their join $g_1\star g_2:l_1+l_2\to s_1+s_2$ satisfies this equation:  
\begin{align}
a^{s_1+s_2}_x\circ (\varphi(g_1\star g_2) \oslash {\sf id}_x) & =  a^{s_1}_x \circ ({\sf id}_{M^{\otimes s_1}} \oslash a_x^{s_2}) \circ (\varphi(g_1) \oslash (\varphi(g_2) \oslash {\sf id}_x ) ) \\
& = a_x^{s_1} \circ ( \varphi(g_1) \oslash (a_x^{s_2} \circ (\varphi(g_2) \oslash {\sf id}_x))) \\ 
& = a_x^{s_1} \circ (\varphi(g_1) \oslash a_x^{l_2}) \\
& =a_x^{l_1+l_2}.
\end{align}
\end{proof}

\end{document}